\documentclass[11pt]{amsart}
\usepackage{txfonts}
\usepackage{amsmath,amsthm,verbatim}
\usepackage{amscd,amssymb}
\usepackage[all]{xy} 
\usepackage{graphicx,calrsfs} 
\usepackage[colorlinks,plainpages,urlcolor=blue]{hyperref} 
\usepackage{breakurl}
\usepackage{bbm}
\usepackage{bbm}
\usepackage{mathrsfs}
\usepackage{xstring}
\usepackage{tikz}
\usetikzlibrary{matrix,arrows,decorations.pathmorphing,calc,shapes}
\usepackage[inline]{enumitem}

\newtheorem{theorem}{Theorem}[section]
\newtheorem{corollary}[theorem]{Corollary}
\newtheorem{lemma}[theorem]{Lemma}
\newtheorem{prop}[theorem]{Proposition}

\theoremstyle{definition}
\newtheorem{definition}[theorem]{Definition}
\newtheorem{example}[theorem]{Example}
\newtheorem{remark}[theorem]{Remark}

\newtheorem*{ack}{Acknowledgments}

\DeclareMathOperator{\im}{im}

\DeclareMathOperator{\id}{id}

\DeclareMathOperator{\Ext}{Ext}

\DeclareMathOperator{\ab}{ab}

\DeclareMathOperator{\Prim}{Prim}
\DeclareMathOperator{\gr}{gr}
\DeclareMathOperator{\Sym}{Sym}
\DeclareMathOperator{\Hilb}{Hilb}

\DeclareMathOperator{\Aut}{Aut}
\DeclareMathOperator{\Iso}{Iso}
\DeclareMathOperator{\Tr}{Tr}
\DeclareMathOperator{\QTr}{QTr}

\newcommand{\dga}{\ensuremath{\textsc{dga}}}
\newcommand{\cdga}{\ensuremath{\textsc{dga}}}
\newcommand{\cga}{\ensuremath{\textsc{cga}}}
\newcommand{\N}{\mathbb{N}}
\newcommand{\R}{\mathbb{R}}
\newcommand{\Q}{\mathbb{Q}}
\newcommand{\C}{\mathbb{C}}
\newcommand{\Z}{\mathbb{Z}}

\newcommand{\K}{\mathbb{K}}

\renewcommand{\k}{\Bbbk}
\newcommand{\g}{\mathfrak{g}}
\newcommand{\h}{\mathfrak{h}}
\newcommand{\M}{\mathcal{M}}

\newcommand{\fa}{\mathfrak{a}}

\newcommand{\fh}{\mathfrak{h}}
\newcommand{\fg}{\mathfrak{g}}
\newcommand{\fn}{\mathfrak{n}}
\newcommand{\fM}{\mathfrak{M}}
\newcommand{\fm}{\mathfrak{m}}
\newcommand{\fL}{\mathfrak{L}}
\newcommand{\tr}{\mathfrak{tr}}
\newcommand{\qtr}{\mathfrak{qtr}}
\newcommand{\fr}{\mathfrak{r}}

\newcommand{\Lie}{\mathfrak{lie}}

\newcommand{\bL}{\mathbf{L}}

\newcommand{\cC}{\mathcal{C}}

\newcommand{\cM}{\mathcal{M}}
\newcommand{\cN}{\mathcal{N}}

\newcommand{\cF}{\mathcal{F}}
\newcommand{\cG}{\mathcal{G}}

\DeclareMathAlphabet{\pazocal}{OMS}{zplm}{m}{n}

\newcommand{\surj}{\twoheadrightarrow}
\newcommand{\inj}{\hookrightarrow}
\newcommand{\cstar}{\,\hat{*}\,}
\newcommand{\abs}[1]{\left| #1 \right|}

\newcommand\isom{\xrightarrow{
 \,\smash{\raisebox{-0.6ex}{\ensuremath{\scriptstyle\simeq}}}\,}}
\newcommand{\bwedge}{\mbox{\normalsize $\bigwedge$}}
\newcommand{\bcup}{\mbox{\normalsize $\bigcup$}}
\newcommand{\boplus}{\mbox{\normalsize $\bigoplus$}}
\newcommand{\bprod}{\mbox{\normalsize $\prod$}}
\def\dot{\mathchar"013A}  
\newcommand{\hdot}{{\raise1.2pt\hbox to0.3em{\Large $\dot$}}} 
\newcommand{\qA}{\mathrm{q}A}
\newcommand{\qg}{{\mbox{\rm{\small{q}}}}\fg}
\newcommand{\q}{\mathrm{q}}

\begin{document}

\title[Formality properties of groups and Lie algebras]%
{Formality properties of finitely generated groups and \\ Lie algebras}

\author[Alexander~I.~Suciu]{Alexander~I.~Suciu$^1$}
\address{Department of Mathematics,
Northeastern University,
Boston, MA 02115, USA}
\email{\href{mailto:a.suciu@northeastern.edu}{a.suciu@northeastern.edu}}
\urladdr{\href{http://web.northeastern.edu/suciu/}%
{http://web.northeastern.edu/suciu/}}
\thanks{$^1$Supported in part by the National Science Foundation 
(grant DMS--1010298), the National Security Agency 
(grant H98230-13-1-0225), and the Simons Foundation 
(collaboration grant for mathematicians 354156)}

\author{He Wang}
\address{Department of Mathematics and Statistics MS0084,
University of Nevada, Reno, NV 89557, USA}
\email{\href{mailto:wanghemath@gmail.com}%
{wanghemath@gmail.com}, \href{mailto:hew@unr.edu}%
{hew@unr.edu}}
\urladdr{\href{http://wolfweb.unr.edu/homepage/hew/}%
{http://wolfweb.unr.edu/homepage/hew/}}



\begin{abstract}
We explore the graded-formality and filtered-formality properties of  
finitely generated groups by studying the various Lie algebras 
over a field of characteristic $0$ attached to such groups, including 
the Malcev Lie algebra, the associated graded Lie algebra, 
the holonomy Lie algebra, and the Chen Lie algebra.  
We explain how these notions behave with 
respect to split injections, coproducts, direct products, as well as field 
extensions, and how they are inherited by solvable and nilpotent quotients. 
A key tool in this analysis is the $1$-minimal model of the group, and 
the way this model relates to the aforementioned Lie algebras. 
We illustrate our approach with examples drawn 
from a variety of group-theoretic and topological contexts, such 
as finitely generated torsion-free nilpotent 
groups, link groups, and fundamental groups of Seifert fibered manifolds.  
\end{abstract}

\maketitle

\section{Introduction}
\label{sect:intro}

The main focus of this paper is on the formality properties of 
finitely generated groups, as reflected in the structure of the 
various graded or filtered Lie algebras, as well as commutative, 
differential graded algebras attached to such groups. 

\subsection{From groups to Lie algebras}
\label{subset:groups to Lie}
Throughout, we will let $G$ be a finitely generated group, and we will  
let $\k$ be a coefficient field of characteristic $0$. Our main focus will 
be on several $\k$-Lie algebras attached to such a group, and the way 
they all connect to each other.  

By far the best known of these Lie algebras 
is the {\em associated graded Lie algebra}, $\gr(G;\k)$, 
introduced by W.~Magnus in the 1930s, and further developed by E.~Witt, P.~Hall, 
 M.~Lazard, and many others, cf.~\cite{Magnus-K-S} and references therein.  
This is a finitely generated graded Lie algebra, whose graded pieces are the 
successive quotients of the lower central series of $G$ (tensored with $\k$),  
and whose Lie bracket is induced from the group commutator.  The quintessential 
example is the free Lie algebra $\Lie(\k^n)$, which is the associated graded 
Lie algebra of the free group on $n$ generators, $F_n$. 

Closely related is the {\em holonomy Lie algebra}, $\fh(G;\k)$, 
introduced by Kohno in \cite{Kohno}, building on work of Chen \cite{Chen73}, 
and further studied by Markl--Papadima \cite{Markl-Papadima} and
Papadima--Suciu \cite{Papadima-Suciu04}.  This is a quadratic 
Lie algebra, obtained as  the quotient of the free Lie algebra on 
$H_1(G;\k)$ by the ideal generated by the image of 
the dual of the cup product map in degree $1$.  
The holonomy Lie algebra comes equipped with a natural epimorphism 
$\Phi_G\colon \fh(G; {\k}) \surj \gr(G;\k)$, and thus can 
be viewed as the quadratic approximation to the associated 
graded Lie algebra.

The most intricate of these Lie algebras (yet, in many ways, the 
most important) is the {\em Malcev Lie algebra}, $\fm(G;\k)$. 
As shown by Malcev in \cite{Malcev51}, every finitely generated, 
torsion-free nilpotent group $N$ is the fundamental group 
of a nilmanifold, whose corresponding $\k$-Lie algebra is $\fm(N;\k)$.  
Taking now the nilpotent quotients of $G$, we may define 
$\fm(G;\k)$ as the inverse limit of the resulting tower of nilpotent 
Lie algebras, $\fm(G/\Gamma_k G;\k)$.  By construction,  the 
Malcev Lie algebra $\fm(G;\k)$, endowed with the inverse limit filtration,
is a complete, filtered Lie algebra; the prounipotent group corresponding 
to this pronilpotent Lie algebra is denoted by $\fM(G;\k)$.
In \cite{Quillen68, Quillen69}, 
Quillen showed that $\fm(G;\k)$ is the set of all primitive elements in
$\widehat{\k{G}}$, the completion of the group algebra of $G$ 
with respect to the filtration by powers of the augmentation ideal, 
and that the associated graded Lie algebra of $\fm(G;\k)$ with respect to the 
inverse limit filtration is isomorphic to $\gr(G;\k)$. Furthermore, 
the set of all group-like elements in $ \widehat{\k{G}}$, 
with multiplication and filtration inherited from $\widehat{\k{G}}$, 
forms a complete, filtered group isomorphic to $\fM(G;\k)$.

\subsection{Formality notions}
\label{subsec:formal}

In his foundational paper on rational homotopy theory \cite{Sullivan}, 
Sullivan associated to each path-connected space $X$ a `minimal 
model,' $\cM(X)$, which can be viewed as an algebraic approximation 
to the space.  If, moreover, $X$ is a CW-complex with 
finitely many $1$-cells, then 
the Lie algebra dual to the first stage of the minimal model 
is isomorphic to the Malcev Lie algebra $\fm(G;\Q)$ associated 
to the fundamental group $G=\pi_1(X)$.
The space $X$ is said to be {\em formal}\/ if the commutative, 
graded differential algebra $\cM(X)$ is quasi-isomorphic to the 
cohomology ring $H^{\hdot}(X;\Q)$, endowed with the zero differential. 
If there exists a $\dga$~ morphism from the $i$-minimal model $\cM(X,i)$
to $H^{\hdot}(X;\Q)$  
inducing isomorphisms in cohomology up to degree $i$ and 
a monomorphism in de
gree $i+1$, then $X$ is called {\em $i$-formal}.  

A finitely generated group $G$ is said to be \emph{$1$-formal}\/ 
(over $\Q$) if it has a classifying space $K(G,1)$ which is $1$-formal.
The study of the various Lie algebras attached to the fundamental 
group of a space provides a fruitful way to look at the formality 
problem.  Indeed, the group $G$ is $1$-formal if and only if the
Malcev Lie algebra $\fm(G;\Q)$ is isomorphic to the rational 
holonomy Lie algebra of $G$, completed with respect to 
the lower central series (LCS) filtration.  

We find it useful to separate the $1$-formality property of a 
group $G$ into two complementary properties: 
graded-formality and filtered-formality.  More precisely, we say 
that $G$ is {\em graded-formal}\/ (over $\k$) if the associated graded Lie 
algebra $\gr(G;\k)$ is isomorphic, as a graded Lie algebra, to  
the holonomy Lie algebra $\h(G;\k)$.  Likewise, 
we say that $G$ is \emph{filtered-formal}\/ (over $\k$) 
if the Malcev Lie algebra $\fm=\fm(G;\k)$ is isomorphic, as a 
filtered Lie algebra, to the completion of its associated 
graded Lie algebra, $\widehat{\gr}(\fm)$, where both $\fm$ and 
$\widehat{\gr}(\fm)$ are endowed with the respective inverse limit filtrations. 
As we show in Proposition \ref{prop:qwformal},
the group $G$ is $1$-formal if and only if it is both 
graded-formal and filtered-formal.

All four possible combinations of these formality properties do occur. 
First, examples of $1$-formal groups include  right-angled 
Artin groups, groups with first Betti number equal to $0$ or $1$, 
fundamental groups of compact K\"ahler manifolds, and 
fundamental groups of complements of complex algebraic 
hypersurfaces. 
 Second, there are many torsion-free, nilpotent 
groups (Examples \ref{ex:filt massey} and \ref{ex:upper}) 
as well as fundamental groups of link complements 
(Example \ref{ex:link}) which are filtered-formal, but not graded-formal.  
 Third, there are also finitely presented groups, such as those from 
Examples \ref{ex:qtr groups}, \ref{ex:quadrelation},  
and \ref{ex:semiproduct} which are graded-formal 
but not filtered-formal.  
Fourth, there are  groups which enjoy none of these 
formality properties; indeed, if $G_1$ is a group of the second type 
and $G_2$ is a group of the third type, then 
Theorem \ref{thm:product formality} below shows that 
the product $G_1\times G_2$ and the free product $G_1* G_2$ are 
neither graded-formal, nor filtered-formal.

\subsection{Field extensions and formality}
\label{subsec:fieldextension} 

We start by reviewing in \S\ref{sect:Liealgebras} some 
basic notions pertaining to filtered and graded Lie algebras.  
We say that a Lie algebra $\g$ (over a field $\k$ of characteristic $0$) 
is {\em filtered formal}\/ if it 
admits complete, separated filtration, and there exists    
a filtration-preserving isomorphism $\g\cong \widehat{\gr}(\g)$ 
which induces the identity on associated graded Lie algebras.
Our first result, which generalizes a recent theorem of Cornulier \cite{Cornulier14}, 
shows that filtered-formality behaves well with respect to field extensions.  
The proof we give in Theorem \ref{thm:ffdescent} is based 
on recent work of Enriquez \cite{Enriquez}  and Maassarani \cite{Maassarani}.  

\begin{theorem} 
\label{thm:intro fflie}
Let $\fg$ be a $\k$-Lie algebra endowed with a complete, separated 
filtration such that $\gr(\fg)$ is finitely generated in degree $1$.   
If $\k\subset \K$ is a field extension, then $\fg$ 
is filtered-formal if and only if the 
$\K$-Lie algebra $\fg \otimes_{\k} \K$ is  filtered-formal. 
\end{theorem}

We continue in \S\ref{sect:koszul} with a review of the notions 
of quadratic and Koszul algebras.  In \S\ref{sec:formality}, 
we analyze in detail the relationship between the $1$-minimal model 
$\cM(A,1)$ and the dual Lie algebra $\fL(A)$ of a differential graded 
$\k$-algebra $(A,d)$.  The reason for doing this is a  
result of Sullivan \cite{Sullivan}, which gives a functorial isomorphism of 
pronilpotent Lie algebras, $\fL(A)\cong \fm(G;\k)$, provided $\cM(A,1)$ 
is a $1$-minimal model for a finitely generated group $G$. 
The book by F\'elix, Halperin, and Thomas \cite{FHT2} 
provides a good reference for this subject. 

Of particular interest is the case when $A$ is a  connected, 
graded commutative algebra with $A^1$ finite-dimensional, endowed 
with the differential $d=0$. In Theorem \ref{thm:model-holonomy}, 
we show that $\fL(A)$ is isomorphic (as a  filtered Lie algebra) 
to the degree completion of the holonomy Lie algebra of $A$.
In the case when $A^{\le 2}=H^{\le 2}(G;\k)$, for some 
finitely generated group $G$, this result recovers the aforementioned 
characterization of the $1$-formality property of $G$. In Theorem 
\ref{thm:filtmin} we give an alternate interpretation of filtered 
formality:  A group $G$ is filtered-formal if 
and only if $G$ has a $1$-minimal model whose differential is 
homogeneous with respect to the canonical Hirsch weights.

As is well-known, a space $X$ with finite Betti numbers is formal 
over $\Q$ if and only it is formal over $\k$, for any field $\k$ of characteristic $0$.  
This foundational result was proved independently and in various degrees of generality 
by Halperin and Stasheff \cite{Halperin-Stasheff79}, Neisendorfer and Miller
\cite{Neisendorfer-Miller78}, and Sullivan \cite{Sullivan}. 
Motivated by these classical results, as well as the aforementioned 
work of Cornulier, we investigate 
the way in which the formality properties of 
spaces and groups behave under field extensions.  
Our next result, which is a combination of  
Corollary \ref{cor:pfext}, Corollary \ref{lem:gf-ext}, 
Proposition \ref{prop:field-filtered}, and Corollary \ref{cor:1fdescent}, 
can be stated as follows.
 
\begin{theorem} 
\label{thm:intro fieldextension}
Let $X$ be a path-connected space, with finitely generated fundamental group $G$.  
Let $\k$ be a field of characteristic zero, and let  $\k\subset \K$ be a field extension. 

\begin{enumerate}
\item  \label{ffe1}
Suppose $X$ has finite Betti numbers $b_1(X),\dots,b_{i+1}(X)$.  
Then $X$ is $i$-formal over $\k$ if and only if $X$ is $i$-formal over $\K.$
\item  \label{ffe2}
The group $G$ is $1$-formal over $\k$ if and only if $G$ is $1$-formal over $\K.$
\item  \label{ffe3}
The group $G$ is graded-formal over $\k$ if and only if $G$ is graded-formal over $\K.$
\item  \label{ffe4}
The group $G$ is filtered-formal over $\k$ if and only if $G$ is filtered-formal over $\K.$
\end{enumerate}
\end{theorem}

In summary, under appropriate finiteness conditions, all the formality 
properties that we study in this paper are independent 
of the ground field $\k\supset \Q$. Hence, we will sometimes 
avoid mentioning the coefficient 
field when referring to these formality notions. 
The descent property for partial formality from 
Theorem \ref{thm:intro fieldextension}, part \eqref{ffe1} 
has been used in \cite{PS16} to establish the $(n-1)$-formality 
over $\Q$ of compact Sasakian manifolds of dimension $2n+1$. 

\subsection{Propagation of formality}
\label{subsec:propagate} 

Next, we turn our attention to the way in which the various formality 
notions for groups behave with respect to split injections, coproducts, and direct products.  
Our first result in this direction is a combination of Theorem \ref{thm:graded1formality} 
and \ref{thm:formality}, and can be stated as follows.
 
\begin{theorem} 
\label{thm:intro formality}
Let $G$ be a finitely generated group, and let $K\le G$ be a 
subgroup.   Suppose there is a split monomorphism 
$\iota\colon K\rightarrow G$. Then:
\begin{enumerate}
\item \label{inj1}
If $G$ is graded-formal, then $K$ is also graded-formal. 
\item \label{inj2}
If $G$ is filtered-formal, then $K$ is also filtered-formal. 
\item \label{inj3}
If $G$ is $1$-formal, then $K$ is also $1$-formal. 
\end{enumerate}
\end{theorem}

In particular, if a semi-direct product $G_1\rtimes G_2$ has one 
of the above formality properties, then $G_2$ also has that property;   
in general, though, $G_1$ will not, as illustrated in Example \ref{ex:semiproduct1}.

As shown by Dimca et al.~\cite{Dimca-Papadima-Suciu}, both the 
product and the coproduct of two $1$-formal groups is again $1$-formal. 
Also, as shown by Plantiko \cite{Plantiko96}, the product and 
coproduct of two graded-formal groups is again graded-formal.
We sharpen these results in the next theorem, which is a combination of 
Propositions \ref{prop:productquadratic} and \ref{prop:productweakly}. 

\begin{theorem}
\label{thm:product formality}
Let  $G_1$ and $G_2$ be two finitely generated groups. 
The following conditions are equivalent.
\begin{enumerate}
\item \label{gff1}  $G_1$ and $G_2$ are graded-formal 
(respectively, filtered-formal, or $1$-formal).
\item \label{gff2}  $G_1* G_2$  is graded-formal 
(respectively, filtered-formal, or $1$-formal).
\item \label{gff3}  $G_1\times G_2$ is graded-formal 
(respectively, filtered-formal, or $1$-formal).
\end{enumerate}
\end{theorem}

Both Theorem \ref{thm:intro formality} and Theorem \ref{thm:product formality} can
be used to decide the formality properties of new groups from those of known groups.
In general, though, even when both $G_1$ and $G_2$ are $1$-formal, we cannot 
conclude that an arbitrary semi-direct product $G_1\rtimes G_2$ is $1$-formal
(see Example \ref{ex:semiproduct}).
The various formality properties are not necessarily inherited 
by quotient groups. However, as we shall see in 
Theorem \ref{thm:chenlie intro} and  Theorem \ref{thm:formalityQuo}, 
respectively, filtered-formality is passed on to the derived 
quotients and to the nilpotent quotients of a group.

\subsection{Derived series and Lie algebras} 
\label{intro:chen}
 
In \S \ref{section:chen lie}, we investigate some of the relationships 
between the lower central series and derived series of a group 
and the derived series of the corresponding Lie algebras.
In \cite{Chen51}, Chen studied the lower central series quotients 
of the maximal metabelian quotient of a finitely generated free group, 
and computed their graded ranks.  More generally, following 
Papadima and Suciu \cite{Papadima-Suciu04}, 
we may define the \textit{Chen Lie algebras}\/ of a group $G$ as 
the associated graded Lie algebras of its solvable quotients,  
$\gr(G/G^{(i)};\k)$. Our next theorem (which combines  
Theorem \ref{thm:chenlieiso} and Corollary \ref{cor:chenlie bis}) 
sharpens and extends the main result of \cite{Papadima-Suciu04}.

\begin{theorem} 
\label{thm:chenlie intro}
Let $G$ be a finitely generated group. 
For each $i\ge 2$, the projection $G\surj G/G^{(i)}$ induces a 
natural epimorphism of graded $\k$-Lie algebras, 
$\Psi_G^{(i)}\colon \gr(G;\k)/\gr(G;\k)^{(i)} 
\surj \gr(G/G^{(i)};\k)$. Moreover, 
\begin{enumerate}
\item \label{ch1}
If $G$ is a filtered-formal group, then each solvable quotient 
$G/G^{(i)}$ is also filtered-formal, and the map $\Psi_G^{(i)}$ 
is an isomorphism.

\item \label{ch2}
If $G$ is a $1$-formal group, then 
$\h(G;\k)/\h(G;\k)^{(i)} \cong \gr(G/G^{(i)};\k)$.
\end{enumerate}
\end{theorem}

Given a finitely presented group $G$, the solvable quotients $G/G^{(i)}$ 
need not be finitely presented (see \cite{PS17}).  Thus, finding presentations 
for the Chen Lie algebra $\gr(G/G^{(i)})$ can be an arduous task.  Nevertheless, 
Theorem \ref{thm:chenlie intro} provides a method for finding such  
presentations, under suitable formality assumptions.  The theorem 
can also be used as an obstruction to $1$-formality.

\subsection{Nilpotent groups and Lie algebras} 
\label{subsec:tfnilp}

Our techniques apply especially well to the class of 
 finitely generated, torsion-free nilpotent groups. Carlson and Toledo 
\cite{Carlson-Toledo95} studied the $1$-formality properties of such 
groups, while Plantiko \cite{Plantiko96} gave a sufficient conditions 
for such groups to be non-graded-formal. For nilpotent Lie algebras, 
the notion of filtered-formality has been studied by Leger \cite{Leger63}, 
Cornulier \cite{Cornulier14}, Kasuya \cite{Kasuya}, and others.  
In particular, it is shown in \cite{Cornulier14} that the systolic 
growth of a finitely generated nilpotent group $G$ 
is asymptotically equivalent to its growth if and only if the Malcev Lie 
algebra $\fm(G;\k)$ is filtered-formal (or, `Carnot'), while in 
\cite{Kasuya} it is shown that the variety of flat connections 
on a filtered-formal (or, `naturally graded'), $n$-step nilpotent 
Lie algebra $\mathcal{g}$ has a singularity at the origin cut 
out by polynomials of degree at most $n+1$.

We investigate in \S \ref{sect:nilp} the filtered-formality of 
nilpotent groups and Lie algebras.  The next result combines Theorem
\ref{thm:step2nilpotent} and Proposition \ref{prop:koszul-formality}.

\begin{theorem}
\label{thm:nilp intro}
Let $G$ be a finitely generated, torsion-free nilpotent group. 
\begin{enumerate}
\item \label{ffn-1}
If $G$ is a $2$-step nilpotent group with torsion-free 
abelianization, then $G$ is filtered-formal.
\item \label{ffn-2}
Suppose $G$ is filtered-formal. Then 
the universal enveloping algebra $U(\gr(G;\k))$ is Koszul 
if and only if $G$ is abelian.
\end{enumerate}
\end{theorem}

As mentioned in \S\ref{subsec:propagate}, nilpotent quotients 
of filtered-formal groups are filtered-formal; 
in particular, each $n$-step, free nilpotent group $F/\Gamma_n F$ 
is filtered-formal.   A classical example is the unipotent group $U_n(\Z)$,  
which is known to be filtered-formal by Lambe and Priddy \cite{Lambe-Priddy82},
but not graded-formal for $n\geq 3$.
In \cite{Cornulier14}, Cornulier showed that the filtered-formality 
of a finite-dimensional nilpotent Lie algebra is independent of the ground field, 
thereby answering a question of Johnson \cite{Jo}.  
The much more general Theorem \ref{thm:intro fflie} 
allows us to recover this result in Proposition \ref{prop:carnot}. 

\subsection{Further applications} 
\label{subsec:apps}

We end in \S \ref{sect:seifert} with a detailed study of fundamental groups 
of (orientable) Seifert fibered manifolds from a rational homotopy viewpoint. 
Starting from the minimal model of such a manifold $M$,  
as described in \cite{Putinar98}, we find a presentation 
for the Malcev Lie algebra $\fm(\pi_1(M);\k)$, and we use this information 
to derive a presentation for $\gr(\pi_1(M);\k)$. As an application, we 
show that Seifert manifold groups are filtered-formal, and determine 
precisely which ones are graded-formal. The techniques 
developed here have been used in \cite{PS16} to prove a more general result 
about the filtered-formality of Sasakian groups.

This work was motivated in good part by papers 
\cite{Bartholdi-E-E-R, Calaque-E-E} of Etingof et al.~on 
triangular and quasi-triangular groups, also 
known as the (upper) pure virtual braid groups.  
In  \cite{SW-pvb}, we apply the techniques developed 
here to study the formality properties of such groups. 
Related results for the McCool groups (or, the 
welded pure braid groups) and other 
braid-like groups are given in \cite{SW-braids, SW-mccool}.
Finally, the relationship between filtered formality and expansions 
of groups is explored in \cite{SW-expansion}.

\section{Filtered and graded Lie algebras}
\label{sect:Liealgebras}
 
In this section we study the interactions between filtered 
Lie algebras, their completions, and their associated graded 
Lie algebras, mainly as they relate to the notion of filtered-formality.

\subsection{Graded Lie algebras}
\label{subsec:grlie}

We start by reviewing some standard material on Lie algebras, 
following the exposition from 
\cite{Ekedahl-Merkulov11, Polishchuk-Positselski, Quillen69, Serre}. 

Fix a ground field $\k$ of characteristic $0$.  Let $\g$ be a Lie algebra over 
$\k$, i.e., a $\k$-vector space $\g$ endowed with a bilinear operation 
$[\,,\,]\colon \g\times \g\to \g$ satisfying the Lie identities.  We say that 
$\g$ is a {\em graded Lie algebra}\/ if $\g$ decomposes as 
$\g=\bigoplus_{i\ge 1} \g_i$ and the Lie bracket sends $\g_i\times \g_j$ 
to $\g_{i+j}$, for all $i$ and $j$.  A morphism of graded Lie algebras is a 
$\k$-linear map $\varphi\colon \g\to \h$ which preserves the Lie brackets 
and the degrees. In particular, $\varphi$ induces $\k$-linear maps 
$\varphi_i\colon \g_i\to \h_i$ for all $i\ge 1$.

The most basic example of a graded Lie algebra is constructed as follows.  
Let $V$ a $\k$-vector space. The tensor algebra $T(V)$ has a natural 
Hopf algebra structure, with comultiplication $\Delta$ and counit $\varepsilon$ 
the algebra maps given by $\Delta(v)= v\otimes 1+1\otimes v$ and $\varepsilon(v)= 0$, 
for $v\in V$. The {\em free Lie algebra}\/ on $V$ is the set  of primitive 
elements, i.e., $\Lie(V)=\{x \in T(V) \mid \Delta(x)=x \otimes 1+
1\otimes x\}$, with Lie bracket $[x,y]=x\otimes y-y\otimes x$ 
and grading induced from $T(V)$. 

A Lie algebra $\fg$ is said to be \textit{finitely generated}\/ if there 
is an epimorphism $\varphi\colon \Lie(V)\to \fg$ for some finite-dimensional 
$\k$-vector space $V$. 
If, moreover, the Lie ideal $\mathfrak{r}=\ker (\varphi)$ is finitely 
generated as a Lie algebra, then $\fg$ is called \textit{finitely presented}. 
Now suppose all elements of $V$ are assigned degree $1$ in $T(V)$.  Then 
the inclusion $\iota\colon \Lie(V)\to T(V)$ identifies $\Lie_1(V)$ with $T_1(V)=V$.  
Furthermore, $\iota$ maps $\Lie_2(V)$ to $T_2(V)=V\otimes V$ 
by sending $[v,w]$ to $v\otimes w-w\otimes v$ for each $v,w\in V$;   
we thus may identify $\Lie_2(V)\cong V\wedge V$ by sending 
$[v,w]$ to $v\wedge w$.   

If $\fg=\Lie(V)/\fr$, with $V$ a (finite-dimensional) vector space 
concentrated in degree $1$,  
then we say $\fg$ is {\em (finitely) generated in degree $1$}.  
If, moreover, the Lie ideal $\mathfrak{r}$ is homogeneous, 
then $\fg$ is a graded Lie algebra.  In particular, if $\g$ is 
finitely generated in degree $1$ and the homogeneous ideal 
$\mathfrak{r}$ is generated in degree $2$, then we say 
 $\fg$ is a \textit{quadratic Lie algebra}. 

\subsection{Filtrations}
\label{subsec:series}
We will be very much interested in this work in Lie algebras 
endowed with a filtration, usually but not always enjoying 
an extra `multiplicative' property.  At the most basic level, 
a {\em filtration}\/ $\cF$ on a Lie algebra $\fg$ is a nested 
sequence of Lie ideals, $\g=\cF_1\g\supset \cF_2\g\supset \cdots$.  

A well-known such filtration is the {\em derived series}, $\cF_i\g=\g^{(i-1)}$, 
defined by $\fg^{(0)}=\fg$ and $\fg^{(i)}=[\fg^{(i-1)}, \fg^{(i-1)}]$ for $i\geq 1$. 
The derived series is preserved by Lie algebra maps.  
The quotient Lie algebras $\fg/\fg^{(i)}$ are solvable; moreover, 
if $\fg$ is a graded Lie algebra, all these solvable quotients 
inherit a graded Lie algebra structure. 

The existence of a filtration $\cF$ on a Lie algebra $\fg$  
makes $\fg$ into a topological vector space, by defining 
a basis of open neighborhoods of an element $x\in\fg$ 
to be $\{x+\cF_k \fg\}_{k\in \N}$. The fact that each basis 
neighborhood $\cF_k\fg$ is a Lie subalgebra implies that the Lie 
bracket map $[\,,\,]\colon \g \times \g \to \g$ is continuous; thus, 
$\fg$ is, in fact, a topological Lie algebra. 
We say that $\fg$ is \emph{complete}\/ (respectively, {\em separated}) 
if the underlying topological vector space enjoys those properties. 

Given an ideal $\mathfrak{a}\subset \fg$, there is an induced filtration 
on it, given by $\cF_k \mathfrak{a}=\cF_k\g\cap \mathfrak{a}$.  Likewise, 
the quotient Lie algebra, $\fg/\mathfrak{a}$, has a naturally induced 
filtration with terms $\cF_k\g/\cF_k \mathfrak{a}$. 
Let $\overline{\mathfrak{a}}$ be the closure of $\mathfrak{a}$ 
in the filtration topology.   Then $\overline{\mathfrak{a}}$ is a closed ideal 
of $\g$. Moreover, by the continuity of the Lie bracket, we have that 
\begin{equation}
\label{eq:closureBracket}
\overline{[\bar{\fa},\bar{\fr}]}=\overline{[\fa,\fr]}.
\end{equation}
Finally, if $\g$ is complete (or separated), then $\fg/\overline{\mathfrak{a}}$ 
is also complete (or separated). 

\subsection{Completions}
\label{subsec:completion}
For each $j\ge k$, there is a canonical projection $\g/\cF_j\g\to \g/\cF_k\g$, 
compatible with the projections from $\g$ to its quotient Lie algebras $\g/\cF_k\g$.  
The {\em completion}\/ of the Lie algebra $\fg$ with respect to the filtration 
$\cF$ is defined as the limit of this inverse system, i.e., 
\begin{equation}
\label{eq:ghat}
\widehat{\fg}:= \varprojlim\nolimits_{k}\fg/\cF_k\fg=\Big\lbrace (g_1,g_2,\dots)\in 
\prod\nolimits_{i=1}^{\infty}\fg/\cF_i\fg \:\big|\: g_j \equiv g_k\bmod  \cF_k\fg\ 
\text{for all $j>k$}
\Big.\Big\rbrace.
\end{equation}

Using the fact that $\cF_k(\fg)$ is an ideal of $\fg$,
it is readily seen that $\widehat{\fg}$ is a Lie algebra, with Lie bracket 
defined componentwise.  Furthermore, $\widehat{\fg}$ has 
a natural inverse limit filtration, $\widehat{\cF}$, given by 
\begin{equation}
\label{eq:ghat filt}
\widehat{\cF}_k\widehat{\fg}:=\widehat{\cF_k\fg}= 
\varprojlim\nolimits_{i\ge k}\cF_k\fg/\cF_i\fg = 
\big\lbrace (g_1,g_2,\dots)\in 
\widehat{\fg}\mid g_i=0 \ \text{for all $i< k$}\big\rbrace.
\end{equation}

Note that $\widehat{\cF}_k\widehat{\fg} = \overline{\cF_k\fg} $, 
and so each term of the filtration $\widehat{\cF}$ is a closed 
Lie ideal of $\widehat{\fg}$.   Furthermore, the Lie algebra 
$\widehat{\g}$, endowed with this filtration, is both complete and separated. 

Let $\iota\colon \g\to \widehat{\g}$ be the canonical map to the completion.  
Then $\iota$ is a morphism of Lie algebras, preserving the respective 
filtrations.  Clearly, $\ker(\iota)=\bigcap_{k\ge 1} \cF_k \g$.  
Hence, $\iota$ is injective if and only if $\g$ is separated.  
Furthermore, $\iota$ is surjective if and only if $\g$ is complete. 

\subsection{Filtered Lie algebras}
\label{subsec:filt lie}

A \textit{filtered Lie algebra}\/ (over the field $\k$) 
is a Lie algebra $\fg$ endowed with a $\k$-vector filtration 
$\{\cF_k\fg\}_{k\geq 1}$ satisfying the `multiplicativity' condition 
\begin{equation}
\label{eq:mult filt}
[\cF_r\fg,\cF_s\fg]\subseteq \cF_{r+s}\fg
\end{equation}
for all $r, s\ge 1$.  
Obviously, this condition implies that each subspace $\cF_k\fg$ 
is a Lie ideal, and so, in particular, $\cF$ is a Lie algebra filtration. 
Let 
\vspace*{-3pt}
\begin{equation}
\label{eq:grfg}
\gr^{\cF}(\fg):=\boplus_{k\geq 1}\cF_k\fg/ \cF_{k+1}\fg.
\end{equation}
be the associated graded vector space to the filtration $\cF$ on $\g$. 
Condition \eqref{eq:mult filt} implies that the Lie bracket map on 
$\g$ descends to a map $[\,,\,]\colon \gr^{\cF}(\fg) \times \gr^{\cF}(\fg) \to \gr^{\cF}(\fg)$, 
which makes $\gr^{\cF}(\fg)$ into a graded Lie algebra, 
with graded pieces given by decomposition \eqref{eq:grfg}. 

A morphism of filtered Lie algebras is a linear map 
$\phi\colon \fg\to \fh$ preserving Lie brackets and the 
given filtrations, $\cF$ and $\mathcal{G}$. Such a map  
induces morphisms between  nilpotent quotients, 
$\phi_k\colon \fg/\cF_{k+1}\fg \to \fh/\cG_{k+1}\fh $, and 
a morphism of associated graded Lie algebras, 
$\gr(\phi)\colon \gr^{\cF}(\fg)\rightarrow \gr^{\cG}(\fh)$. 

If $\g$ is a filtered Lie algebra with a multiplicative 
filtration $\cF$, then its completion, $\widehat{\g}$, 
is again a filtered Lie algebra with the completed multiplicative filtration $\widehat{\cF}$. 
By construction, the canonical map to the completion, $\iota\colon \g\to \widehat{\g}$, 
is a morphism of filtered Lie algebras.  It is readily seen that the induced 
morphism, $\gr(\iota)\colon \gr^{\cF}(\fg) \to \gr^{\widehat{\cF}}(\widehat{\fg})$, 
is an isomorphism.  Moreover, if $\g$ is both complete 
and separated, then the map $\iota\colon \g\to \widehat{\g}$ itself is 
an isomorphism of filtered Lie algebras. 

\begin{lemma}
\label{lem:grfilt}
Let  $\phi\colon \g\to \h$ be a morphism of complete, separated, 
filtered Lie algebras, and suppose $\gr(\phi)\colon  \gr^{\cF}(\g)\to \gr^{\cG}(\h)$ 
is an isomorphism. Then $\phi$ is also an isomorphism.
\end{lemma}

\begin{proof}
By assumption, the homomorphisms 
$\gr_k(\phi)\colon \cF_{k}\g/\cF_{k+1}\g \to \cG_k\h/\cG_{k+1} \h$ 
are isomorphisms, for all $k\ge 1$.  An easy induction on $k$ 
shows that all maps $\phi_k\colon \fg/\cF_{k+1}\fg \to \fh/\cG_{k+1}\fh $ 
are isomorphisms. Therefore, the map 
$\hat{\phi} \colon \widehat{\g} \to \widehat{\h}$ is 
an isomorphism. On the other hand, both $\g$ and $\h$ are 
complete and separated, and so $\g=\widehat{\g}$ 
and $\h=\widehat{\h}$.  Hence $\phi=\hat{\phi}$, and we are done.
\end{proof}

\subsection{The degree completion}
\label{subsec:deg completion}
Any Lie algebra $\fg$ comes equipped with a lower central series (LCS) filtration, 
$\{\Gamma_k(\fg)\}_{k\geq 1}$, defined by $\Gamma_1(\fg)=\fg$ 
and $\Gamma_k(\fg)=[\Gamma_{k-1}(\fg),\fg]$ for $k\geq 2$.  
Clearly, this is a multiplicative filtration.  Any other such 
filtration $\{\cF_k(\fg)\}_{k\leq 1}$ on $\fg$ is coarser than this filtration; 
that is, $\Gamma_k\fg\subseteq \cF_k\fg$, for all $k\ge 1$.
Any Lie algebra morphism $\phi\colon \fg\rightarrow \fh$ preserves LCS filtrations.  
Furthermore, the quotient Lie algebras $\fg/\Gamma_k\fg$ are nilpotent. 
For simplicity, we shall write $\gr(\g):=\gr^{\Gamma}(\g)$ for the 
associated graded Lie algebra and $\widehat{\g}$ for the completion 
of $\g$ with respect to the LCS filtration $\Gamma$.  Furthermore, 
we shall take $\widehat{\Gamma}_k=\overline{\Gamma}_k$ as 
the canonical filtration on $\widehat{\g}$.

Every graded Lie algebra, $\fg=\bigoplus_{i\geq 1} \fg_i$, has 
a canonical decreasing filtration induced by the grading, 
$\cF_k\fg=\bigoplus_{i\geq k} \fg_i$. 
Moreover, if $\fg$ is generated in degree $1$, then 
this filtration coincides with the LCS filtration $\Gamma_k(\fg)$. 
In particular, the associated graded Lie algebra with respect to $\cF$  
coincides with $\g$. In this case,
the completion of $\fg$ with respect 
to the lower central series (or, degree) filtration is called the 
\textit{degree completion}\/ of $\fg$, and is simply denoted by 
$\widehat{\fg}$.  It is readily seen that $\widehat{\fg}\cong \prod_{i\geq 1} \fg_i$.  
Therefore, the morphism $\iota\colon \g\to \widehat{\g}$ is injective, 
and induces an isomorphism $\fg \cong \gr^{\widehat{\Gamma}}( \widehat{\fg} )$.  
Moreover, if $\fh$ is a graded Lie subalgebra of $\fg$, then 
$\widehat{\fh}=\overline{\fh}$ and 
\begin{equation}
\label{eq:subLiecompletion}
\gr^{\widehat{\Gamma}}(\widehat{\fh})=\fh.
\end{equation}

\begin{lemma}
\label{lem:presbar}
If $\fL$ is a free Lie algebra generated in degree $1$, and $\fr$ is a homogeneous 
ideal, then the projection $\pi\colon \fL\to \fL/\fr$ induces an isomorphism 
$\widehat{\fL}/\overline{\fr}\isom\widehat{\fL/\fr}$.
\end{lemma}

\begin{proof}
Without loss of generality, we may assume that $\fr\subset [\fL,\fL]$.  
The projection $\pi\colon \fL\to \fL/\fr$ extends to an epimorphism 
between the degree completions, 
$\hat{\pi}\colon \widehat{\fL}\to \widehat{\fL/\fr}$. This morphism takes 
the ideal generated by $\fr$ to $0$; thus, by continuity, it 
induces an epimorphism of complete, filtered Lie algebras, 
$\widehat{\fL}/\bar\fr\surj \widehat{\fL/\fr}$. 
Taking associated graded, we obtain an epimorphism 
$\gr(\hat\pi)\colon \gr(\widehat{\fL}/\bar\fr)\surj \gr(\widehat{\fL/\fr})=\fL/\fr$.  
This epimorphism admits a splitting, induced by the maps  
$\Gamma_n{\fL}+\fr\to  \widehat{\Gamma}_n{\widehat{\fL}}+\bar\fr$; 
thus, $\gr(\hat\pi)$ is an isomorphism.  
The claim now follows from Lemma \ref{lem:grfilt}.
\end{proof}

\subsection{Filtered-formality}
\label{subsec:filtered formal}
We now consider in more detail the relationship between a filtered 
Lie algebra $\fg$ and the completion of its associated graded Lie 
algebra, $\widehat{\gr}(\fg)$, endowed with the inverse limit filtration.  
Note that both Lie algebras share the same associated graded Lie 
algebra, namely, $\gr(\g)$.  In general, though, $\fg$ may not  be  
isomorphic to $\widehat{\gr}(\fg)$.  Of course, this happens  
if $\fg$ is not complete or separated, but it may happen 
even in the case when $\g$ is a (finite-dimensional) nilpotent 
Lie algebra. We shall illustrate this point in Examples \ref{ex:Cornulier} and 
\ref{ex:LPex7} below. 

The following definition will play a key role in the sequel. 

\begin{definition}
\label{def:filt formal}
A complete, separated, filtered Lie algebra $\fg$ is  \textit{filtered-formal}\/ 
if there is a filtered Lie algebra isomorphism $\fg\cong \widehat{\gr}(\fg)$ 
which induces the identity on associated graded Lie algebras.
\end{definition}

This notion appears in the work of Bezrukavnikov \cite{Bez} and Hain \cite{Hain97},  
as well as in the work of Calaque  et. al  \cite{Calaque-E-E} 
under the name of `formality', and in the work of Lee \cite{Lee}, 
under the name of `weak-formality'.  The reasons for our choice of 
terminology will become more apparent in \S \ref{sect:malcev}.

If $\fg$ is a filtered-formal Lie algebra, there exists a graded 
Lie algebra $\h$ such that $\fg$ is isomorphic to 
$\widehat{\fh}=\prod_{i\geq 1} \fh_i$.
Conversely, if $\fg=\widehat{\h}$ is the completion of a 
graded Lie algebra $\h=\bigoplus_{i\ge 1} \h_i$, then $\fg$ is filtered-formal. 
Moreover, if $\fh$ has homogeneous presentation $\fh=\Lie(V)/\mathfrak{r}$, 
with $V$ finitely generated and concentrated in degree $1$, then, 
by Lemma \ref{lem:presbar}, the complete, filtered Lie algebra $\fg=\prod_{i\ge 1} \h_i$ has 
presentation $\fg=\widehat{\Lie}(V)/\overline{\mathfrak{r}}$.  

\begin{lemma}
\label{lem:filtiso}
Let $\fg$ be a complete, separated, filtered Lie algebra. 
If there is a graded Lie algebra $\fh$ 
and a Lie algebra isomorphism $\fg\cong \widehat{\fh}$ preserving 
filtrations, then $\fg$ is filtered-formal.
\end{lemma}

\begin{proof}
By assumption, there exists a filtered Lie algebra isomorphism 
$\phi\colon \fg\rightarrow \widehat{\fh}$. The map $\phi$ induces 
an isomorphism of graded Lie algebras, $\gr(\phi)\colon \gr(\fg)\rightarrow \fh$.  
In turn, the map $\psi:= (\gr(\phi))^{-1}$ induces an isomorphism 
$\hat{\psi}\colon \widehat{\fh} \rightarrow \widehat{\gr}(\fg)$ 
of completed Lie algebras. Hence, the composite  
$\tilde\phi:=\hat{\psi}\circ\phi\colon \fg \to \widehat{\gr}(\fg)$
is an isomorphism of filtered Lie algebras inducing 
the identity on $\gr(\fg)$.  The conclusion follows 
from Lemma \ref{lem:grfilt}.
\end{proof}

\begin{corollary}
\label{cor:ff gen1}
Let $\fg$ be a complete, separated, filtered Lie algebra, and suppose  
the associated graded Lie algebra $\gr(\fg)$ is generated in degree $1$.  
Furthermore, suppose there is a morphism of filtered Lie algebras,  
$\phi\colon \fg \to \widehat{\gr}(\fg)$, such that $\gr_1(\phi)$ is an 
isomorphism. Then $\fg$ is filtered-formal.
\end{corollary} 

\begin{proof}
Consider the morphism $\gr(\phi)\colon \gr(\g) \to \gr(\g)$.  
Since $\gr(\g)$ is generated in degree $1$, and since $\gr_1(\phi)$ is an 
isomorphism, the map $\gr(\phi)$ is an isomorphism.  
By Lemma \ref{lem:grfilt}, the map $\phi$ itself is an isomorphism. 
The conclusion follows from Lemma \ref{lem:filtiso}.
\end{proof}

\subsection{Descent of filtered-formality}
\label{subsec:descentff}

We now show that filtered-formality is compatible with extension of 
scalars, and, more importantly, that filtered-formality enjoys a descent 
property, under some mild finiteness assumptions.  As usual, all the 
ground fields will be of characteristic $0$. We start with an easy lemma,
which follows from the fact that completion commutes with tensor products.

\begin{lemma}
\label{lem:ffascent}
Let $\fg$ be a filtered-formal $\k$-Lie algebra, 
and let  $\k\subset \K$ be a  field extension.  Then 
the $\K$-Lie algebra $\fg \otimes_{\k} \K$ is also filtered-formal. 
\end{lemma}

The proof of the next result is based on recent work of Enriquez \cite{Enriquez} 
and Maassarani \cite{Maassarani}.  The key tool is Proposition 7.6 from 
\cite{Enriquez}, which in turn was inspired by  work of Drinfeld \cite{Dr}.  
The structure of the proof follows to a large extent the approach from \cite{Maassarani}, 
where a particular example (the Malcev Lie algebra of the fundamental group 
of the orbit configuration space of a finite subgroup of $\operatorname{PSL}_2(\C)$ 
acting on $\mathbb{CP}^1$) is treated. 

\begin{theorem}
\label{thm:ffdescent}
Let $\fg$ be a complete, separated, filtered $\k$-Lie algebra 
such that $\gr(\fg)$ is finitely generated in degree $1$.   
If $\k\subset \K$ is a field extension, then $\fg$ 
is filtered-formal if and only if the 
$\K$-Lie algebra $\fg \otimes_{\k} \K$ is  filtered-formal. 
\end{theorem}

\begin{proof}
The forward implication follows at once from Lemma \ref{lem:ffascent}. 
For the backward implication, suppose $\fg \otimes_{\k} \K$ is filtered-formal 
over $\K$.  Again in view of Lemma \ref{lem:ffascent}, we may assume 
without loss of generality that $\k=\Q$ and $\K=\overline{\K}$.  

Set $\h=\widehat{\gr}(\fg)$, and let $\{\cG_k\}_{k\ge 1}$ be the inverse 
limit filtration on $\h$, coming from the degree filtration on $\gr(\g)$. 
For simplicity, let us write $\g^i=\fg/\cF_{i+1}\fg$ 
and $\h^i=\fh/\cG_{i+1}\fh$ for the respective quotient Lie algebras. 
As noted in  \cite[Lem.~6.1]{Maassarani}, the image of $\cF_k(\g)$ in $\g^i$ 
is $\Gamma_k(\g^i)$.  In particular, 
$\g^1$ is canonically isomorphic to the abelianizations of all $\g^i$  
and of $\g$.  A similar statement holds for  $\h$.

For each $i\ge 1$, let $T_i=\Iso_1(\g^i,\h^i)$ be the affine 
$\Q$-scheme of filtration-preserving Lie algebra isomorphisms from 
$\g^i$ to $\h^i$ inducing the identity on abelianizations.  
As shown in \cite[Prop.~6.2]{Maassarani}, these schemes form 
in natural way an inverse system; let $T=\varprojlim_{i} T_i$.   
Similarly, let $U_i=\Aut_1(\g^i)$ be 
the unipotent $\Q$-group of automorphisms of $\g^i$ inducing 
the identity on abelianization, and let 
$U=\varprojlim_{i} U_i$ be the corresponding prounipotent 
$\Q$-group scheme.   It is then readily seen that each $U_i$ 
is a torsor under the natural left action of $T_i$, i.e., the action 
of $U_i(\k)$ on $T_i(\k)$ is free and transitive 
whenever $\Q\subset \k$ is a field extension such that 
$T_i(\k)$ is non-empty.  Furthermore, as noted in 
\cite[Prop.~6.6]{Maassarani}, the $U_i$-actions on the torsors 
$T_i$ are compatible with the canonical projections;   
thus, $T$ is also a torsor under the action of $U$. 

By assumption, $\g \otimes_{\Q} \K$ is filtered-formal.  In view of 
Corollary \ref{cor:ff gen1}, this condition is equivalent to the 
existence of a filtered Lie algebra isomorphism 
$\g \otimes_{\Q} \K \to \h\otimes_{\Q} \K$ inducing the canonical 
identification $\g^1 \otimes_{\Q} \K = \h^1 \otimes_{\Q} \K$.  That is, 
our assumption is equivalent to the fact that $T(\K)\ne \emptyset$.   
It remains to show that $T(\Q)\ne \emptyset$.

When $\K=\C$, this claim follows at once from Proposition 7.6 in 
\cite{Enriquez}.  The proof of that proposition involves two steps: 
first a descent from $\C$ to $\overline{\Q}$, and then 
from $\overline{\Q}$  to $\Q$.  To handle an arbitrary extension 
$\Q\subset \K$, we only need to modify the first step, and descend 
from $\K=\overline{\K}$ to $\overline{\Q}$.  This is done by means of the 
same type of Hilbert's Nullstellensatz argument as the one sketched in \cite{Enriquez}; 
we refer to the proof of \cite[Cor.~5.8]{Neisendorfer-Miller78} for 
more details on how such an argument works.  
\end{proof}

As we shall see in Proposition \ref{prop:carnot}, the above theorem 
generalizes a recent result of Cornulier (Theorem 3.14 from \cite{Cornulier14}).

\subsection{Products and coproducts}
\label{subsec:prcopr}

The category of Lie algebras admits both products and coproducts.  
We end this section by showing that filtered-formality behaves 
well with respect to these operations. 

\begin{lemma}
\label{lem:filt prod}
Let  $\fm$ and $\fn$ be two filtered-formal Lie algebras. Then 
$\fm\times \fn$ is also filtered-formal.
\end{lemma}

\begin{proof}
By assumption, there exist graded Lie algebras $\fg$ and $\fh$ 
such that $\fm\cong \widehat{\fg}=\prod_{i\geq 1} \fg_i$ and 
$\fn\cong \widehat{\fh}=\prod_{i\geq 1} \fh_i$.
Then $\fm\times \fn$ is isomorphic to 
$\big(\bprod_{i\geq 1} \fg_i \big)\times \big(\bprod_{i\geq 1} \fh_i \big)=
 \bprod_{i\geq 1} (\fg_i\times \fh_i)=\widehat{\fg\times \fh}$.
Hence, $\fm\times \fn$ is filtered-formal.  
\end{proof}

Now let $*$ denote the usual coproduct (or, free product) of Lie algebras, 
and let $\hat{*}$ be the coproduct in the category of complete, filtered 
Lie algebras.  By definition, 
\begin{equation}
\label{eq:cstar}
\fm \cstar \fn =\widehat{{\,\fm*\fn}^{\phantom{a}}}=  
\varprojlim\nolimits_{k}\, (\fm * \fn)/\Gamma_k(\fm * \fn).
\end{equation}
We refer to Lazarev and Markl \cite{LazarevMarkl} for a detailed 
study of this notion.

\begin{lemma}
\label{lem:filt coprod}
Let  $\fm$ and $\fn$ be two filtered-formal Lie algebras. Then 
$\fm\cstar \fn$ is also filtered-formal.
\end{lemma}

\begin{proof}
As before, write $\fm=\widehat{\fg}$ and $\fn=\widehat{\fh}$, 
for some graded Lie algebras $\g$ and $\h$.
The canonical inclusions $\alpha\colon \fg\inj \fm$ and 
$\beta\colon\fh\inj \fn$ induce a monomorphism of filtered Lie algebras, 
$\widehat{\alpha*\beta}\colon \widehat{\fg *\fh} \rightarrow \widehat{\fm*\fn}$. 
Using \cite[(9.3)]{LazarevMarkl}, we infer that the induced morphism 
between associated graded Lie algebras, 
$\gr(\widehat{\alpha*\beta})\colon \gr(\widehat{\fg *\fh}) \rightarrow 
\gr(\widehat{\fm*\fn})$, is an isomorphism. 
Lemma \ref{lem:grfilt} now implies that  $\widehat{\alpha*\beta}$ 
is an isomorphism of filtered Lie algebras, 
thereby verifying the filtered-formality of $\fm \cstar \fn$.
\end{proof}

\section{Graded algebras and Koszul duality}
\label{sect:koszul}

The notions of graded and filtered algebras are defined completely analogously 
for an (associative) algebra $A$: the multiplication 
map is required to preserve the grading, respectively the filtration on $A$.    
In this section we discuss several relationships between Lie algebras 
and associative algebras, focussing on the notion of quadratic 
and Koszul algebras.  

\subsection{Universal enveloping algebras}
\label{subsec:univ}

Given a Lie algebra $\fg$ over a field $\k$ of characteristic $0$, 
let $U(\fg)$ be its universal enveloping algebra. This is the filtered 
algebra obtained as the quotient of the tensor algebra $T(\fg)$ 
by the (two-sided) ideal $I$ generated by all elements of the form 
$a\otimes b-b\otimes a-[a,b]$ with $a, b\in \fg$. By the 
Poincar\'{e}--Birkhoff--Witt theorem, the canonical map 
$\iota \colon \fg\to U(\fg)$ is an injection, and the induced 
map, $\Sym(\g)\to \gr(U(\g))$, is an isomorphism of graded 
(commutative) algebras. In this section, all tensor products are over $\k$.

Now suppose $\fg$ is a finitely generated, graded Lie algebra.  
Then $U(\g)$ is isomorphic (as a graded vector space) to a 
polynomial algebra in variables indexed by bases  
for the graded pieces of $\g$, with degrees set accordingly. 
Hence, its Hilbert series is given by 
\begin{equation}
\label{eq:PBW}
\Hilb(U(\fg),t)=\bprod_{i\geq 1}(1-t^i)^{-\dim(\fg_i)}.
\end{equation}
For instance, if $\g=\Lie(V)$ 
is the free Lie algebra on a finite-dimensional vector space $V$ 
with all generators in degree $1$, then 
$\dim(\fg_i)=\tfrac{1}{i}\sum_{d | i}\mu(d)\cdot n^{i/d}$, 
where $n=\dim V$ and $\mu\colon \N \to \{-1,0,1\}$ is the M\"obius function.

Finally, suppose $\fg=\Lie(V)/ \fr$ is a finitely presented, 
graded Lie algebra, with generators in degree $1$ and relation ideal $\fr$ 
generated by homogeneous elements $g_1,\dots, g_m$.  
Then $U(\fg)$ is the quotient of $T(V)$ by the two-sided ideal generated by 
$\iota(g_1), \dots, \iota(g_m)$, where $\iota\colon \Lie(V)\inj T(V)$ 
is the canonical inclusion. In particular, if $\fg$ is a quadratic 
Lie algebra, then $U(\fg)$ is a quadratic algebra.

\subsection{Quadratic algebras}
\label{subsec:quad}

Now let $A$ be a graded $\k$-algebra.  We will  assume throughout 
that $A$ is non-negatively graded, i.e., $A=\bigoplus_{i\ge 0} A_i$, 
and connected, i.e., $A_0=\k$.  Every such algebra may be realized 
as the quotient of a tensor algebra $T(V)$ by a homogeneous, 
two-sided ideal $I$. We will further assume that  $\dim V<\infty$. 

An algebra $A$ as above  is said to be \emph{quadratic}\/ if $A_1=V$ and the 
ideal $I$ is generated in degree $2$, i.e., $I=\langle I_2 \rangle$, 
where $I_2=I\cap (V\otimes V)$.  
Given a quadratic algebra $A=T(V)/I$, identify 
$V^*\otimes V^* \cong (V\otimes V)^*$, and 
define the {\em quadratic dual}\/ of $A$ to be the algebra 
$A^{!}=T(V^*)/I^{\perp}$,
where $I^{\perp}\subset T(V^*)$ is the ideal generated by the vector 
subspace $I_2^{\perp}:=\{\alpha\in V^*\otimes V^* \mid \alpha(I_2)=0\}$. 
Clearly, $A^{!}$ is also a quadratic algebra, and $(A^{!})^{!}=A$.    
For any graded algebra $A=T(V)/I$, we can define its quadrature closure
as $\qA=T(V)/\langle I_2\rangle$.  For more details on all this, 
we refer to \cite{Polishchuk-Positselski}.

\begin{prop}
\label{prop:holo closure}
Let $\fg$ be a finitely generated graded Lie algebra generated in degree $1$. 
There is then a unique, functorially defined quadratic Lie algebra, 
$\qg$, such that $U(\qg) = \q U(\fg)$. 
\end{prop}

\begin{proof}
Suppose $\g$ has presentation $\Lie(V)/\fr$. Then $U(\fg)$ has a 
presentation $T(V)/(\iota(\fr))$.
Set $\qg=\Lie(V)/\langle \fr_2 \rangle$, where $\fr_2=\fr\cap \Lie_2(V)$; then 
$U(\qg)$ has presentation $T(V)/\langle \iota(\fr_2) \rangle$. 
One can see that $\iota(\fr_2)=\iota(\fr)\cap V\otimes V$. 
\end{proof}

A {\em commutative graded algebra}\/ (for short, a \cga) is a 
graded $\k$-algebra as above, which in addition is graded-commutative, 
i.e., if $a\in A_i$ and $b\in A_j$, then $ab=(-1)^{ij} ba$.  
If all generators of $A$ are in degree $1$, then $A$ can be written as $A=\bigwedge(V)/J$, 
where $\bigwedge(V)$ is the exterior algebra on the $\k$-vector space $V=A_1$,  
and $J$ is a homogeneous ideal in $\bigwedge(V)$ with $J_1=0$.  
If, furthermore, $J$ is generated in degree $2$, then $A$ is a 
quadratic \cga.   The next lemma follows straight from the definitions.

\begin{lemma}
\label{lem:vee}
Let $W\subset V\wedge V$ be a linear subspace, and 
let $A=\bigwedge(V)/\langle W\rangle$ be the corresponding 
quadratic \cga. Then $A^!=T(V^*)/\langle \iota(W^{\vee})\rangle$, where 
\begin{equation}
\label{eq:jvee}
W^{\vee}:=\{\alpha\in V^*\wedge V^* \mid \alpha(W)=0\}=
W^{\perp}\cap (V^*\wedge V^* ) ,
\end{equation}
and $\iota \colon V^*\wedge V^* \inj V^*\otimes V^*$ is the 
inclusion map, given by 
$x\wedge y \mapsto x\otimes y-y\otimes x$.
\end{lemma}

For instance, if $A=\bigwedge(V)$, then $A^!=\Sym (V^*)$. 
Likewise, if $A=\bigwedge(V)/\langle V\wedge V\rangle=\k\oplus V$, 
then $A^!=T(V^*)$.

\subsection{Holonomy Lie algebras}
\label{subsec:hololie}

Let $A$ be a graded, graded-commutative algebra. Recall we are 
assuming that $A_0=\k$ and $\dim A_1<\infty$.  
Because of graded-commutativity, the multiplication map $A_1\otimes A_1\to A_2$ 
factors through a linear map $\mu_A\colon A_1\wedge A_1\to A_2$. 
Dualizing this map, and identifying $(A_1\wedge A_1)^* \cong A_1^*\wedge A_1^*$, 
we obtain a linear map, $\partial_A=(\mu_A)^*\colon A_2^*\rightarrow A_1^*\wedge A_1^*$.
Finally, identify $A_1^*\wedge A_1^*$ with $\Lie_2(A_1^*)$ via 
the map $x\wedge y \mapsto [x,y]$.

\begin{definition}
\label{def:holo alg}
The \textit{holonomy Lie algebra}\/ of $A$ is the quotient 
\begin{equation}
\label{eq:hololiealg}
\fh(A)=\Lie(A_1^*)/\langle\im (\partial_A)\rangle 
\end{equation}
of the free Lie algebra on $A_1^*$ by the ideal generated by the 
image of $\partial_A$ under the above identification.
Alternatively, using the notation from \eqref{eq:jvee}, we have that 
\begin{equation}
\label{eq:hololie bis}
\fh(A)=\Lie(A_1^*)/\langle \ker(\mu_A)^{\vee} \rangle.
\end{equation}
\end{definition}

Plainly, $\fh(A)$ is a quadratic Lie algebra.  
This construction is functorial:  if $\varphi\colon A\to B$ is a 
morphism of $\cga$s as above, the induced map,  
$\Lie(\varphi_1^*)\colon \Lie(B_1^*)\to \Lie(A_1^*)$, factors through a 
morphism of graded Lie algebras, $\fh(\varphi)\colon \fh(B)\to \fh(A)$;    
moreover, if $\varphi$ is injective, then $\h(\varphi)$ is surjective.
The Lie algebra $\h(A)$ depends only on information encoded in the 
map $\mu_A\colon A_1\wedge A_1\to A_2$.  
More precisely, let $\qA$ be the quadratic closure of $A$; then 
\begin{equation}
\label{eq:bara}
\qA=\bwedge(A_1) /\langle K\rangle,
\end{equation}
where $K=\ker(\mu_A)\subset A_1\wedge A_1$. 
Then $\qA$ is a commutative, quadratic algebra, 
which comes equipped with a canonical morphism $\qA\to A$, 
which is an isomorphism in degree $1$ and is injective in degree $2$.  
It is readily verified that the induced morphism 
between holonomy Lie algebras, $\fh(A)\to \fh(\qA)$, is an isomorphism.

The following proposition is a slight generalization of a result of Papadima--Yuzvinsky
(\cite[Lem.~4.1]{Papadima-Yuzvinsky}).  

\begin{prop}
\label{prop:Papadima-Y}
Let $A$ be a commutative graded algebra. 
Then $U(\fh(A))$ is a quadratic algebra, and $U(\fh(A))=(\qA)^!$.
\end{prop}

\begin{proof}
By the above, $\qA=\bigwedge(A_1)/\langle K \rangle$, 
where $K=\ker(\mu_A)$. 
On the other hand, by \eqref{eq:hololie bis} we have that 
$\fh(A)=\Lie(A_1^*)/\langle K^{\vee} \rangle$. 
Hence, by Lemma \ref{lem:vee},  
$U(\fh(A))=T(V^*)/\langle \iota(K^{\vee} )\rangle=(\qA)^!$.
\end{proof}

Combining Propositions \ref{prop:holo closure} and \ref{prop:Papadima-Y}, 
we obtain the following corollary, which expresses the quadratic closure of 
a Lie algebra as the holonomy Lie algebra of a certain quadratic algebra. 

\begin{corollary}
\label{cor:quadraticgr}
Let $\fg$ be a finitely generated graded Lie algebra generated in degree $1$.
Then $\fh\big((\q U(\fg))^{!} \big)=\qg$.
\end{corollary}

Work of L\"{o}fwall \cite[Thm.~1.1]{Lofwall} yields another interpretation of the 
universal enveloping algebra of the holonomy Lie algebra.  

\begin{prop}[\cite{Lofwall}]
\label{prop:yoneda}
Let $[\Ext^1_A(\k,\k)]:=\bigoplus_{i\ge 0} \Ext^i_A(\k,\k)_i$ be the linear strand 
in the Yoneda algebra of $A$.  Then $U(\h(A))\cong [\Ext^1_A(\k,\k)]$. 
\end{prop}

In particular, the graded ranks of  $\h=\h(A)$ 
are given by $\prod_{n\geq 1}(1-t^n)^{\dim \h_n} = \sum_{i\ge 0} b_{ii} t^i$, 
where $b_{ii}=\dim_{\k} \Ext^i_A(\k,\k)_i$.
The next proposition shows that every quadratic Lie algebra can be 
realized as the holonomy Lie algebra of a (quadratic) algebra. 

\begin{prop}
\label{prop:quad holo}
Let $\fg$ be a quadratic Lie algebra. There is then a commutative 
quadratic algebra $A$ such that $\fg=\fh(A)$. 
\end{prop}

\begin{proof}
By assumption, $\fg=\Lie(V)/\langle W \rangle$, where 
$W$ is a linear subspace of $V\wedge V$.
Define $A=\bigwedge(V^*)/\langle W^{\vee} \rangle$. Then, 
by \eqref{eq:hololie bis}, 
$\fh(A)=\Lie((V^*)^*)/\langle (W^{\vee})^{\vee} \rangle=
\Lie(V)/\langle W\rangle$.
\end{proof}

\subsection{Koszul algebras}
\label{subsec:koszul}
Given any connected, locally finite, graded algebra $A$,
the trivial $A$-module $\k$ has a free, graded $A$-resolution of the form 
$\cdots \xrightarrow{\varphi_3} A^{n_2}  \xrightarrow{\varphi_2}
A^{n_1} \xrightarrow{\varphi_1} A \to \k$; 
such a resolution is \emph{minimal}\/ if all the nonzero 
entries of the matrices $\varphi_i$ have positive degrees.  
The algebra $A$ is \emph{Koszul}\/ if the minimal 
$A$-resolution of $\k$ is linear, 
or, equivalently, $\Ext_A(\k,\k)=[\Ext_A^1(\k,\k)]$. Such an 
algebra is always quadratic, but the converse is far from true.
If $A$ is a Koszul algebra, then the quadratic dual $A^{!}$ is 
also a Koszul algebra, and the following `Koszul duality' formula holds:
\begin{equation}
\label{eq:kdual}
\Hilb(A,t)\cdot \Hilb(A^!,-t)=1.
\end{equation}

Furthermore, if $A$ is a graded algebra of the form 
$A=T(V)/I$, where $I$ is an ideal admitting a (noncommutative) 
quadratic Gr\"{o}bner basis, then $A$ is a Koszul algebra 
(see \cite{Froberg97}).

\begin{corollary}
\label{cor:kdual holo}
Let $A$ be a connected, commutative graded algebra.
If $\qA$ is a Koszul algebra, then $\Hilb(\qA,-t)\cdot \Hilb(U(\h(A)),t)=1$.
\end{corollary}

\begin{example}
\label{ex:non-koszul}
Consider the quadratic algebra $A=\bigwedge(u_1,u_2,u_3,u_4)/(u_1u_2-u_3u_4)$.  
Clearly, we have $\Hilb(A,t)=1 +4t + 5t^2$.  
If $A$ were Koszul, then formula \eqref{eq:kdual} would give
$\Hilb(A^!,t)=1 + 4 t + 11 t^2 + 24 t^3 + 41 t^4 + 44 t^5 - 29 t^6+\cdots $, 
which is impossible.
\end{example}

\begin{example}
\label{ex:BEER}
The quasitriangular Lie algebra $\qtr_n$ defined in \cite{Bartholdi-E-E-R} 
is generated by $x_{ij}$, $1\leq i\neq j\leq n$ with relations
$[x_{ij},x_{ik}]+[x_{ij},x_{jk}]+[x_{ik},x_{jk}]=0$ for distinct $i,j,k$ 
and $[x_{ij},x_{kl}]=0$ for distinct $i,j,k,l$.   
The Lie algebra $\tr_n$ is the quotient 
Lie algebra of $\qtr_n$ by the ideal generated by
$x_{ij}+x_{ji}$ for distinct $i\neq j$. In \cite{Bartholdi-E-E-R}, 
Bartholdi et al.~show that the quadratic dual algebras 
$U(\qtr_n)^!$ and $U(\tr_n)^!$ are Koszul, and compute their Hilbert series.
They also state that neither $\qtr_n$ nor $\tr_n$ 
is filtered-formal for $n\geq 4$, and sketch a proof of this assertion; we give  
a detailed proof of this fact in \cite{SW-pvb}.
\end{example}

\section{Minimal models and (partial) formality}
\label{sec:formality}

In this section, we discuss two basic notions in non-simply-connected rational 
homotopy theory: the minimal model and the (partial) formality properties 
of a differential graded algebra. 

\subsection{Minimal models of $\dga$s}
\label{subsec:minimalModel}

We follow the approach of Sullivan \cite{Sullivan}, Deligne et al.~\cite{DGMS},  
and Morgan \cite{Morgan}, as further developed in \cite{FHT,FHT2, Griffiths-Morgan13, 
Halperin-Stasheff79, Kohno, Macinic}.
We start with some basic algebraic notions.

\begin{definition}
\label{def:dga}
A \emph{differential graded algebra}\/ (for short, a \dga)  
over a field $\k$ of characteristic $0$ is a graded $\k$-algebra 
$A^{\hdot}=\bigoplus_{n\geq 0}A^n$ equipped with a differential 
$d\colon A\rightarrow A$ of degree $1$ satisfying  
$ab=(-1)^{mn}ba$  and $d(ab)=d(a)\cdot b+(-1)^{\abs{a}}a\cdot d(b)$ 
for any $a\in A^m$ and $b\in A^n$. 
We denote the \dga~ by $(A^{\hdot},d)$ or simply by $A^{\hdot}$ if there is no confusion.
 \end{definition}

A morphism $f\colon A^{\hdot}\rightarrow B^{\hdot}$ between two $\dga$s 
is a degree zero algebra map which commutes with the differentials.
 A \emph{Hirsch extension}\/ (of degree $i$) is a $\dga$ inclusion 
$\alpha\colon (A^{\hdot},d_A)\inj (A^{\hdot}\otimes \bigwedge(V),d)$, where 
$V$ is a $\k$-vector space concentrated in 
degree $i$, while $\bigwedge(V)$ is the free graded-commutative 
algebra generated by $V$, and $d$ sends $V$ into $A^{i+1}$. 
We say this is a {\em finite}\/ Hirsch extension if $\dim V<\infty$.  
Note that all tensor product in this section is over $\k$, hence we will denote
$\otimes$ for short.

\begin{definition}[\cite{Sullivan}]
\label{def:min}
A \dga ~$(A^{\hdot},d)$ is called \emph{minimal}\/ if $A^0=\k$, and  the 
following conditions are satisfied:
\begin{enumerate}
\item $A^{\hdot}=\bigcup_{j\geq0}A^{\hdot}_j$, where $A_0=\k$, and $A_{j}$ is a Hirsch 
extension of $A_{j-1}$, for all $j\geq 0$.
\item The differential is \emph{decomposable}, i.e.,  
$dA^{\hdot}\subset A^+\wedge A^+$, where $A^+=\bigoplus_{i\geq1}A^i$.
 \end{enumerate}
 \end{definition}
 
The first condition implies that $A^{\hdot}$ has an increasing, exhausting 
filtration by the sub-$\dga$s $A_j^{\hdot}$; equivalently, $A^{\hdot}$ is free as a 
graded-commutative algebra on generators of degree $\geq 1$. 
(Note that we use the lower-index for the filtration, and the upper-index 
for the grading.)  The second condition is automatically satisfied 
if $A$ is generated in degree $1$. 

A $\dga$ map $f\colon A\to B$ is said to be a {\em quasi-isomorphism}\/ 
if all the induced maps in  cohomology, $H^j(f)\colon H^{j}(A)\to  H^{j}(B)$,  
are isomorphisms.  Two $\dga$s $A$ and $B$ 
are \emph{weakly equivalent}\/ (written $A\simeq B$) if 
there is a zig-zag of quasi-isomorphisms  connecting them. 
Likewise, for an integer $i\ge 0$, a morphism 
$f\colon A\to B$ is an {\em $i$-quasi-isomorphism}\/ 
if $H^j(f)$ is an isomorphism for  $j\le i$ 
and  an injection for $j=i+1$. Furthermore,
$A$ and $B$ are {\em $i$-weakly equivalent}\/ ($A\simeq_i B$) 
if there is a zig-zag of $i$-quasi-isomorphism connecting $A$ to $B$.
The next two lemmas follow straight from the definitions.

\begin{lemma}
\label{lem:Hirschext}
Any $\dga$ morphism $\phi\colon (A,d_A)\to (B,d_B)$ 
extends to a $\dga$ morphism of Hirsch extensions, 
 $\bar{\phi}\colon (A,d_A)\otimes \bigwedge(x)\to 
(B,d_B)\otimes\bigwedge(y)$, provided that $d(y)=\phi(d(x))$. 
Moreover, if $\phi$ is a (quasi-) isomorphism, then so is $\bar{\phi}$.
\end{lemma}

\begin{lemma}
\label{lem:hi}
Let $\alpha\colon A\to B$ be the inclusion map of Hirsch extension 
of degree $i+1$.  Then $\alpha$ is an $i$-quasi-isomorphism.  
\end{lemma}

We say that a $\dga$ $B$ is a \emph{minimal model}\/ for a $\dga$ $A$ 
if $B$ is a minimal $\dga$ and there exists a quasi-isomorphism $f\colon B\rightarrow A$. 
Likewise, we say that a minimal $\dga$ $B$ is an \emph{$i$-minimal model}\/ for $A$ if  $B$ is 
generated by elements of degree at most $i$, and there 
exists an $i$-quasi-isomorphism $f\colon B\rightarrow A$. A basic result in 
rational homotopy theory is the following existence and uniqueness theorem, 
first proved for (full) minimal models by Sullivan \cite{Sullivan}, and in full generality 
by Morgan in \cite[Thm.~5.6]{Morgan}.

\begin{theorem}[\cite{Morgan, Sullivan}]
\label{thm:mm} 
Each connected $\dga$ $(A,d)$ has a minimal model $\cM(A)$, 
unique up to isomorphism. Likewise, for each $i\ge 0$, there is an 
$i$-minimal model $\cM(A,i)$, unique up to isomorphism.
\end{theorem}

It follows from the proof of Theorem \ref{thm:mm} that the minimal model  
$\cM(A)$ is isomorphic to a minimal model built from the $i$-minimal 
model $\cM(A,i)$ by means of Hirsch extensions in degrees $i+1$ and higher.  
Thus, in view of Lemma \ref{lem:hi}, $\cM(A,i)\simeq_i \cM(A)$.  

\subsection{Minimal models and holonomy Lie algebras}
\label{subsec:sullivan-holonomy}

Let $\cM=(\cM^{\hdot},d)$ be a minimal $\dga$ over $\k$, generated in degree $1$. 
Following \cite{Morgan, Kohno, FHT2}, let us consider the filtration
\begin{equation}
\label{eq:filtration-minimal}
\k=\cM_0\subset \cM_1\subset \cM_2\subset \cdots \subset \cM=\bcup_i \cM_i,
\end{equation}
where $\cM_1$ is the subalgebra of $\cM$ generated by $x\in \cM^1$ such that $dx=0$,
and $\cM_{i}$ is the subalgebra of $\cM$ generated by $x\in \cM^1$ such that $dx\in \cM_{i-1}$
for $i> 1$.  Each inclusion $\cM_{i-1}\subset \cM_{i}$ is a 
Hirsch extension of the form $\cM_{i}=\cM_{i-1}\otimes \bigwedge (V_{i})$, 
where $V_i:=\ker\big(H^2(\cM_{i-1})\to H^2(\cM)\big)$.
Taking the degree $1$ part of the filtration \eqref{eq:filtration-minimal}, 
we obtain the filtration $\k=\cM_0^1\subset \cM_1^1\subset  \cdots \subset \cM^1$.

Now assume each of the above Hirsch extensions is finite, i.e., 
$\dim(V_i)<\infty$ for all $i$.  Using the fact that 
$d(V_i)\subset \cM_{i-1}$, we see that each 
dual vector space $\fL_i=(\cM^1_i)^*$ 
acquires the structure of a $\k$-Lie algebra by setting
\vspace*{-3pt}
\begin{equation}
\label{eq:duality}
\langle [u^*,v^*], w\rangle = \langle  u^*\wedge v^*, dw \rangle 
\end{equation}
for $u,v,w\in \cM^1_i$.  Clearly, $d(V_1)=0$, and thus $\fL_1=(V_1)^*$ is 
an abelian Lie algebra. Using the vector space decompositions 
$\cM^1_i=\cM^1_{i-1}\oplus V_i$ and 
$\cM^2_i=\cM^2_{i-1}\oplus (\cM^1_{i-1}\otimes V_i)\oplus \bigwedge^2(V_i)$ 
we easily see that  the canonical projection $\fL_i\surj \fL_{i-1}$ (i.e., the dual 
of the inclusion map $\cM_{i-1}\inj\cM_i$) has kernel $V_i^*$, 
and this kernel is central inside $\fL_i$. 
Therefore, we obtain a tower of finite-dimensional 
nilpotent $\k$-Lie algebras, 
\vspace*{-3pt}
\begin{equation}
\label{eq:nilp lie tower}
\xymatrixcolsep{20pt}
\xymatrix{
0 & \fL_1 \ar@{->>}[l]& \fL_2\ar@{->>}[l] & \cdots\ar@{->>}[l] & \fL_i 
\ar@{->>}[l]& \cdots \ar@{->>}[l].
}.
\end{equation}

The inverse limit of this tower, $\fL=\fL(\cM)$, endowed with the inverse limit filtration, 
is a complete, filtered Lie algebra such that $\fL/\widehat{\Gamma}_{i+1} \fL=\fL_i$, 
for each $i\ge 1$.  Conversely, from a tower as in \eqref{eq:nilp lie tower}, we can construct 
a sequence of finite Hirsch extensions $\cM_i$ as in \eqref{eq:filtration-minimal}.  
Furthermore, the $\dga$ $\cM_i$, with differential defined by \eqref{eq:duality},  
coincides with the Chevalley--Eilenberg complex
 $(\bigwedge(\fL_i^*),d)$ associated to the finite-dimensional 
Lie algebra $\fL_i=\fL(\cM_i)$, as in \cite[\S II]{FHT2} and \cite[\S VII]{Hilton-Stammbach97}. 
In particular,
\begin{equation}
\label{eq:ce}
H^{\hdot}(\cM_i)\cong H^{\hdot}(\fL_i;\k)\, .
\end{equation}

The direct limit of the above sequence of Hirsch extensions, $\cM=\bigcup_i \cM_i$, 
is a minimal $\k$-$\dga$ generated in degree $1$, which we denote by $\cM(\fL)$.  
We obtain in this fashion an adjoint correspondence that sends $\cM$ to the  
pronilpotent Lie algebra $\fL(\cM)$  and conversely, sends a pronilpotent Lie 
algebra $\fL$ to the minimal algebra $\cM(\fL)$.     
Under this correspondence, filtration-preserving $\dga$ morphisms   
$\cM \to \cN$ get sent to filtration-preserving Lie morphisms $\fL(\cN) \to \fL(\cM)$, 
and vice-versa.

\subsection{Positive weights}
\label{subsec:posWeights}

Following \cite{BMSS, Morgan, Sullivan}, 
we say that a $\cga$ $A^{\hdot}$ has \emph{positive weights}\/ if each graded piece 
has a vector space decomposition $A^i=\bigoplus_{\alpha\in\Z}A^{i,\alpha}$ with 
$A^{1,\alpha}=0$ for $\alpha\leq 0$,  such that $xy\in A^{i+j,\alpha+\beta}$ for 
$x\in A^{i,\alpha}$ and $y\in A^{j,\beta}$.  Furthermore, we say that a $\dga$ 
$(A^{\hdot},d)$ has \emph{positive weights}\/ if the underlying $\cga$ $A^{\hdot}$ has 
positive weights, and the differential is homogeneous with respect to those 
weights, i.e., $d(x)\in A^{i+1,\alpha}$ for $x\in A^{i,\alpha}$.

Now let $(\cM^{\hdot},d)$ be a minimal $\dga$ generated in degree one, endowed 
with the canonical filtration $\{\cM_i\}_{i\ge 0}$ constructed in \eqref{eq:filtration-minimal}, 
where each sub-$\dga$ $\cM_i$ given by 
a Hirsch extension of the form $\cM_{i-1}\otimes \bigwedge(V_i)$. 
The underlying $\cga$ $\cM^{\hdot}$ possesses a natural set of positive 
weights, which we will refer to as the {\em Hirsch weights}: simply declare 
$V_i$ to have weight $i$, and extend those weights to $\cM^{\hdot}$ multiplicatively.  
We say that the $\dga$ $(\cM^{\hdot},d)$ has {\em positive Hirsch weights}\/ if 
the differential $d$ is homogeneous with respect to those weights.  If this is the case, 
each sub-$\dga$ $\cM_i$ also has positive Hirsch weights. 

\begin{lemma}
\label{lem:positivegraded}
Let $\cM=(\cM^{\hdot},d)$ be a minimal \dga~ generated in degree one, with dual Lie algebra $\fL$. 
Then $\cM$ has positive Hirsch weights if and only if
$\fL=\widehat{\gr}(\fL)$.
\end{lemma}

\begin{proof}
As usual, write $\cM=\bigcup \cM_i$, with $\cM_i=\cM_{i-1}\otimes \bigwedge(V_i)$.
Since $\cM$ is generated in degree $1$, the differential  is homogeneous 
with respect to the Hirsch weights if and only if
$d(V_s)\subset \bigoplus_{i+j=s}V_i\wedge V_j$, for all $s\ge 1$.  
Passing now to the dual Lie algebra $\fL=\fL(\cM)$ and using 
\eqref{eq:duality}, we see that this condition is equivalent to having 
$[V_i^*, V_j^*]\subset V_{i+j}^*$, for all $i,j\ge 1$. In turn, this is equivalent 
to saying that each  $\fL_i$ is a graded Lie algebra 
with $\gr_k(\fL_i)=V_k^*$, for each $k\leq i$, which means that the filtered Lie algebra 
$\fL=\varprojlim_i \fL_i$ coincides with the completion of its associated graded 
Lie algebra, $\widehat{\gr}(\fL)$.
\end{proof}

\begin{remark}
\label{rem:h-weights}
The property that 
the differential of $\cM$ be homogeneous with respect to 
the Hirsch weights is stronger than saying that the Lie algebra 
$\fL=\fL(\cM)$ is filtered-formal. The fact that this can happen 
is illustrated in Example \ref{ex:3step}. 
\end{remark}

\begin{remark}
\label{rem:h-weights-bis}
If a minimal $\dga$ is 
generated in degree $1$ and has positive weights, but these weights do not 
coincide with the Hirsch weights, then the dual Lie algebra need not be filtered-formal.  
This phenomenon is illustrated in Example \ref{ex:Cornulier}: 
there is a finitely generated nilpotent 
Lie algebra $\fm$ for which the Chevalley--Eilenberg 
complex $\cM(\fm)=\bigwedge(\fm^*)$ has positive weights, 
but those weights are not the Hirsch weights; moreover, $\fm$ is not 
filtered-formal.
\end{remark}

\subsection{Dual Lie algebra and holonomy Lie algebra}
\label{subsec:dualhol}

Let $(B^{\hdot},d)$ be a \dga, and let $A=H^{\hdot}(B)$ be its cohomology algebra. 
Assume $A$ is connected and $\dim A^1<\infty$, and   
let $\mu\colon A^1\wedge A^1\to A^2$ be the multiplication map.
By the discussion from \S\ref{subsec:minimalModel}, there is a 
$1$-minimal model $\cM(B,1)$ for  our $\dga$, unique up 
to isomorphism. 

A concrete way to build such a model can be found in 
\cite{DGMS, Griffiths-Morgan13, Morgan}. 
The first two steps of this construction are easy to describe. 
Set $V_1=A^1$ and define $\cM(B,1)_1=\bigwedge(V_1)$, 
with differential $d=0$. Next,  set $V_2= \ker(\mu)$ and define 
$\cM(B,1)_2=\bigwedge(V_1\oplus V_2)$, with $d|_{V_2}$ 
equal to the inclusion map $V_2\inj A^1\wedge A^1$.  
Let $\fL(B)=\fL(\cM(B,1))$ be the Lie algebra corresponding to the $1$-minimal 
model of $B$. The next proposition, which generalizes a result of Kohno 
(\cite[Lem.~4.9]{Kohno}), 
relates this Lie algebra to the holonomy Lie algebra 
$\fh(A)$ from Definition \ref{def:holo alg}.  

\begin{prop}
\label{prop:kohnolem}
Let $\phi\colon \bL \to \fL(B)$ be the morphism defined  by extending 
the identity map of $V_1^*$ to the free Lie algebra 
$\bL=\Lie(V_1^*)$, and let $J=\ker(\phi)$. 
There exists then an isomorphism of graded Lie algebras,  
$\fh(A)\cong \bL/\langle J\cap \bL_2\rangle$, where $\fh(A)$ is 
the holonomy Lie algebra of $A=H^{\hdot}(B)$.
\end{prop}

\begin{proof}
Let $\gr(\phi)\colon \bL\to \gr^{\widehat{\Gamma}}(\fL(B))$ be the associated 
graded morphism of $\phi$. Then the first graded piece,  
$\gr_1(\phi)\colon V_1^*\to V_1^*$, is the identity, while
the second graded piece, $\gr_2(\phi)$, can be identified 
with the Lie bracket map $V_1^*\wedge V_1^* \to V_2^*$, 
which is the dual of the differential $d\colon V_2\to V_1\wedge V_1$. 
From the construction of $\cM(B,1)_2$, there is an isomorphism 
$\ker (d^*)\cong \im (\mu^*)$. Since $J\cap \bL_2=\ker (\gr_2(\phi))$, 
we have that $\im (\mu^*)=J\cap \bL_2$, and the claim follows.
\end{proof}

\subsection{The completion of the holonomy Lie algebra}
\label{subsec:holohat}

Let $A^{\hdot}$ be a commutative graded $\k$-algebra with $A^0=\k$. 
Proceeding as above, by taking $B=A$ and $d=0$ so that $H^{\hdot}(B)=A$, 
we can construct a $1$-minimal model $\cM=\cM(A,1)$ for the algebra $A$ 
in a `formal' way, following the approach outlined by Carlson and Toledo in 
\cite{Carlson-Toledo95}.  (A construction of the full, bigraded minimal model 
of a $\cga$ can be found in \cite[\S 3]{Halperin-Stasheff79}.)
 
As before, set $\cM_1=(\bigwedge(V_1),d=0)$ where $V_1=A^1$, and 
$\cM_2=(\bigwedge(V_1\oplus V_2),d)$,
where $V_2=\ker \big(\mu\colon A^1\wedge A^1\to A^2\big)$ and 
$d\colon V_2\inj V_1\wedge V_1$ is 
the inclusion map.  After that, define inductively $\cM_{i}$ as 
$\cM_{i-1}\otimes \bigwedge (V_{i})$, where 
the vector space $V_{i}$ fits into the short exact sequence 
\begin{equation}
\label{eq:cartol}
\xymatrixcolsep{18pt}
\xymatrix{0 \ar[r]& V_{i} \ar[r]&  H^2(\cM_{i-1}) \ar[r]& \im(\mu) \ar[r]& 0},
\end{equation}
while the differential $d$ includes $V_{i}$ into $V_1\wedge V_{i-1}\subset \M_{i-1}$. 
In particular, the subalgebras $\cM_i$ constitute the canonical filtration 
\eqref{eq:filtration-minimal} of $\cM$, and the differential $d$ preserves the Hirsch 
weights on $\cM$.  For these reasons, we call $\cM=\cM(A,1)$ the {\em canonical}\/ 
$1$-minimal model of $A$. 

The next theorem relates the Lie algebra dual to 
the canonical $1$-minimal model of a $\cga$ as above to its holonomy Lie algebra.  
A similar result was obtained by Markl and Papadima in \cite{Markl-Papadima}; 
see also Morgan \cite[Thm.~9.4]{Morgan} and Remark \ref{rem:bezpap}. 

\begin{theorem}
\label{thm:model-holonomy}
Let $A^{\hdot}$ be a connected $\cga$ with $\dim A^1<\infty$.
Let $\fL(A):=\fL(\cM(A,1))$  be the Lie algebra corresponding to 
the canonical $1$-minimal model of $A$, and let $\fh(A)$ be the holonomy 
Lie algebra of $A$.  There exists then an isomorphism of complete, 
filtered Lie algebras between $\fL(A)$ and the degree completion 
$\widehat{\fh}(A)$.
\end{theorem}

\begin{proof}
By Definition \ref{def:holo alg}, the holonomy Lie algebra of $A$ has presentation 
$\fh(A)=\bL/\fr$, where $\bL=\Lie(V_1^*)$ and $\fr$ is the ideal generated by 
$\im(\mu^*)\subset \bL_2$.  It follows that, for each $i\geq 1$, the nilpotent quotient 
$\fh_{i}(A):= \fh(A)/\Gamma_{i+1}\fh(A)$ has presentation 
$\bL/(\fr+\Gamma_{i+1}\bL)$. 

Consider now the dual Lie algebra $\fL_i(A)=\fL(\cM_i)$.  
By construction, we have a vector space decomposition, 
$\fL_i(A)=\bigoplus_{s\le i} V_s^*$. 
The fact that $d(V_{s})\subset V_1\wedge V_{s-1}$ implies 
that the Lie bracket maps $V_1^*\wedge V_{s-1}^*$ onto 
$V_s^*$, for every $1<s\le i$.  In turn, this implies that $\fL_i(A)$ is an  
$i$-step nilpotent, graded Lie algebra generated in degree $1$, 
with  $\gr_{s}(\fL_{i}(A))=V_s^*$ for $s\le i$. 

Let $\fr_i$ be the kernel of the canonical projection $\pi_i\colon \bL\surj \fL_i(A)$.  
By the Hopf formula, there is an isomorphism of graded vector spaces 
between $H_2(\fL_i(A);\k)$ and $\fr_i/[\bL,\fr_i]$, 
the space of (minimal) generators for the homogeneous ideal $\fr_i$.
On the other hand,  $H^2(\cM_i)\cong H^2(\fL_i;\k)$, by \eqref{eq:ce}.  
Taking the dual of the exact sequence \eqref{eq:cartol}, 
we find that $H_2(\fL_i(A);\k)\cong \im(\mu^*) \oplus V_{i+1}^*$.
We conclude that the ideal $\fr_i$ is generated by $\im(\mu^*)$  
in degree $2$ and a copy of $V_{i+1}^*$ in degree $i+1$. 

Since  $\gr_2(\fr)=\im(\mu^*)$, we infer that 
$\bigoplus_{s\le i}\gr_{s} (\fr_i)= \bigoplus_{s\le i}\gr_{s} (\fr)$.
Since $\fL_i(A)$ is an $i$-step nilpotent Lie algebra, 
$\bigoplus_{s> i}\gr_{s} (\fr_i)=\Gamma_{i+1}\bL$. 
Therefore, $\Gamma_{i+1}\bL+\fr = \fr_i$, and  
thus the identity map of $\bL$ 
induces an isomorphism $\fL_i(A)\cong \fh_i(A)$, for each $i\ge 1$.  
Hence, $\fL(A)\cong \widehat{\fh}(A)$, as filtered Lie algebras.   
\end{proof}

\begin{corollary}
\label{cor:holomin}
The graded ranks of the holonomy Lie algebra of a connected, graded 
algebra $A$ are given by $\dim \fh_i(A) = \dim V_i$, where 
$\cM=\bwedge \big(\boplus _{i\ge 1} V_i \big)$ is the canonical $1$-minimal 
model of $(A,d=0)$. 
\end{corollary}

\subsection{Partial formality and field extensions}
\label{subsec:formal ext}

The following notion, introduced by Sullivan in \cite{Sullivan}, and further 
developed in \cite{DGMS, Griffiths-Morgan13, Macinic, Morgan}, will play 
a central role in our study.  

\begin{definition}
\label{def:partialformal}
A  $\dga$ $(A^{\hdot},d)$ over $\k$ is said to be {\em formal}\/ if there exists a 
quasi-isomorphism $\cM(A)\to (H^{\hdot}(A),d=0)$.   Likewise, $(A^{\hdot},d)$ is 
said to be {\em $i$-formal}\/ if there exists an  
$i$-quasi-isomorphism $\cM(A,i)\to (H^{\hdot}(A),d=0)$. 
\end{definition}

In \cite{Macinic}, M\u{a}cinic studies in detail these concepts. 
Evidently, if $A$ is formal, then it is $i$-formal, for all $i\ge 0$, and, 
if $A$ is $i$-formal, then it is $j$-formal for every $j\le i$.  
Moreover, $A$ is $0$-formal if and only if $H^0(A)=\k$. 

\begin{lemma}[\cite{Macinic}]
\label{lem:macinic}
A $\dga$ $(A^{\hdot},d)$ is $i$-formal if and only if $(A^{\hdot},d)$ is $i$-weakly 
equivalent to $H^{\hdot}(A)$ with zero differential.
\end{lemma} 

As a  corollary, we deduce that $i$-formality is invariant 
under $i$-weakly equivalence. 

\begin{corollary}
\label{cor:formal-equiv}
Suppose $A\simeq_i B$.  Then $A$ is 
$i$-formal if and only if $B$ is $i$-formal.  
\end{corollary}

Given a $\dga$~$(A,d)$ over a field $\k$ of characteristic $0$, 
and a field extension $\k\subset \K$, let $(A\otimes \K, d\otimes \id_\K)$ be 
the corresponding $\dga$ over $\K$. (If the underlying field $\k$ is understood, we
will usually omit  it from the tensor product $A\otimes_{\k} \K$.)  The following 
result will be crucial for us in the sequel.

\begin{theorem}[Thm.~6.8 in \cite{Halperin-Stasheff79}]
\label{thm:Halperin-Stasheff}
Let $(A^{\hdot},d_A)$ and $(B^{\hdot},d_B)$ be two $\dga$s over $\k$ 
whose cohomology algebras are connected and of finite type. 
Suppose there is an isomorphism of graded algebras, $f\colon H^{\hdot}(A)\to H^{\hdot}(B)$, 
and suppose  $f\otimes \id_{\K}\colon H^{\hdot}(A)\otimes \K\to H^{\hdot}(B)\otimes \K$
can be realized by a weak equivalence between $(A^{\hdot}\otimes \K,d_A\otimes \id_{\K})$ 
and $(B^{\hdot}\otimes \K,d_B\otimes \id_{\K})$. Then $f$ can be realized
by a weak equivalence between $(A^{\hdot},d_A)$ and $(B^{\hdot},d_B)$.
\end{theorem}

This theorem has an important corollary, based on the following 
lemma. For completeness, we provide proofs for these two results, 
which are stated without proof in \cite{Halperin-Stasheff79}.

\begin{lemma}[\cite{Halperin-Stasheff79}]
\label{lem:formal_id}
A \dga~ $(A^{\hdot},d_A)$ with $H^{\hdot}(A)$ of finite-type 
is formal if and only if the identity map of $H^{\hdot}(A)$ 
can be realized by a weak equivalence between $(A^{\hdot},d_A)$ and $(H^{\hdot}(A),d=0)$.
\end{lemma}

\begin{proof}
The backwards implication is obvious. So assume 
 $(A^{\hdot},d_A)$ is formal, that is, there is a zig-zag of quasi-isomorphisms 
between $(A^{\hdot},d_A)$ and $(H^{\hdot}(A),d=0)$.  This yields an  isomorphism 
in cohomology, $\phi\colon H^{\hdot}(A)\to H^{\hdot}(A)$.  
Composing the inverse of $\phi$ with the given zig-zag of quasi-isomorphisms 
defines a new weak equivalence between $(A^{\hdot},d_A)$ and $(H^{\hdot}(A),d=0)$, 
which induces the identity map in cohomology. 
\end{proof}

\begin{corollary}[\cite{Halperin-Stasheff79}]
\label{cor:Kformal}
A $\k$-\dga~ $(A^{\hdot},d_A)$ with $H^{\hdot}(A)$ of finite-type is formal if 
and only if the $\K$-\dga ~$(A^{\hdot}\otimes \K,d_A\otimes \id_{\K})$ is formal.
\end{corollary}

\begin{proof}
The forward implication is obvious. For the converse, 
suppose our $\K$-$\dga$ is formal. 
By Lemma \ref{lem:formal_id}, there exists a weak equivalence between 
$(A^{\hdot}\otimes \K,d_A\otimes \id_{\K})$ and $(H^{\hdot}(A)\otimes \K,d=0)$ 
inducing the identity on $H^{\hdot}(A)\otimes \K$.
By Theorem \ref{thm:Halperin-Stasheff}, the map 
$\id\colon H^{\hdot}(A)\to H^{\hdot}(A)$ can be realized 
by a weak equivalence between $(A^{\hdot},d_A)$ and $(H^{\hdot}(A),d=0)$.
That is, $(A^{\hdot},d_A)$ is formal (over $\k$).
\end{proof}

\subsection{Field extensions and $i$-formality}
\label{subsec:partial ext iformal}

We now use the aforementioned result of Halperin and Stasheff 
on full formality to establish an analogous result for partial formality. 
First we need an auxiliary construction, and a lemma.

Let $\cM(A,i)$ be the $i$-minimal model of a \dga~ $(A^{\hdot},d_A)$.
The degree $i+1$ piece, $\cM(A,i)^{i+1}$, is isomorphic to 
$\big(\ker (d^{i+1})\big)\oplus \cC_{i+1}$, where 
$d^{i+1} \colon \cM(A,i)^{i+1} \to \cM(A,i)^{i+2}$ 
is the differential, and $\cC_{i+1}$ is a complement 
to its kernel. It is readily checked that the vector subspace 
$\mathcal{I}_i:=
\cC_{i+1}\oplus \boplus_{s\ge i+2} \cM(A,i)^{s}$
is an ideal of $\cM(A,i)$, left invariant by the differential.
Consider the quotient \dga, 
$\cM[A,i]:=\cM(A,i) /\mathcal{I}_i$.  Additively, we have that 
$\cM[A,i]= \k\oplus \cM(A,i)^1 \oplus  \dots \oplus \cM(A,i)^i\oplus \ker (d^{i+1})$.

\begin{lemma}
\label{lem:formalityeq} 
Suppose that $\dim H^{i+1}(\M(A,i))< \infty$.  
The following statements are equivalent:
\begin{enumerate*}
\item \label{qf1}  $(A^{\hdot},d_A)$ is $i$-formal; 
\item \label{qf2} $\cM(A,i)$ is $i$-formal; 
\item \label{qf3} $\cM[A,i]$ is $i$-formal; and 
\item \label{qf4} $\cM[A,i]$ is formal.
\end{enumerate*}
\end{lemma}

\begin{proof}
Since $\cM(A,i)$ is an $i$-minimal model for $(A^{\hdot},d_A)$, the 
two $\dga$s are  $i$-quasi-isomorphic.  The equivalence  
\eqref{qf1} $\Leftrightarrow$ \eqref{qf2} follows from 
Corollary \ref{cor:formal-equiv}.

Now let $\psi\colon \cM(A,i)\to \cM[A,i]$ be the canonical projection.  
It is readily checked that the induced homomorphism, 
$H^j(\psi)\colon H^{j}(\cM(A,i))\to H^{j}(\cM[A,i])$, is an isomorphism 
in degrees up to and including $i+1$. 
In particular, this shows that $\cM(A,i)$ is an $i$-minimal model for 
$\cM[A,i]$.  The equivalence  
\eqref{qf2} $\Leftrightarrow$ \eqref{qf3} again follows from 
Corollary \ref{cor:formal-equiv}.

Implication \eqref{qf4} $\Rightarrow$ \eqref{qf3} is trivial, so it remains to establish  
\eqref{qf3} $\Rightarrow$ \eqref{qf4}.   Assume the $\dga$ $\cM[A,i]$ is $i$-formal.
Since $\cM(A,i)$ is an $i$-minimal model for 
$\cM[A,i]$,  there is an $i$-quasi-isomorphism 
$\beta$ as in diagram \eqref{eq:cdformal}.  
In particular, the homomorphism, $H^{i+1}(\beta) \colon 
H^{i+1}(\cM(A,i)) \to H^{i+1}(\cM[A,i])$, is injective.  
On the other hand, we know from the previous paragraph 
that  $H^{i+1}(\cM[A,i])$ and $H^{i+1}(\cM(A,i))$ have the same dimension.  
Since by assumption $\dim H^{i+1}(\M(A,i))< \infty$, we 
conclude that $H^{i+1}(\beta)$ is an isomorphism too.  
\vspace*{-3pt}
\begin{equation}
\label{eq:cdformal}
\begin{gathered}
\xymatrixrowsep{20pt}
\xymatrixcolsep{30pt}
\xymatrix{
\cM(A,i)\ar@{->>}[d]_{\psi} \ar[r]^(.42){\beta}\ar@{^{(}->}[rd]^{\alpha} &(H^{\hdot}(\cM[A,i]),0)\\
\cM[A,i]  & \cM \,. \ar@{.>}[l]^{\simeq}_{\phi} \ar@{.>}[u]^{\gamma}_{\simeq}\\
}
\end{gathered}
\end{equation}

Let $\cM=\cM(\cM[A,i])$ be the full minimal model of $\cM[A,i]$.  As 
mentioned right after Theorem \ref{thm:mm}, this model 
can be constructed by Hirsch extensions of degree $k\geq i+1$, 
starting from the $i$-minimal model of $\cM[A,i]$, which we 
can take to be  $\cM(A,i)$.  Hence, the inclusion map, $\alpha\colon \cM(A,i)\to \cM$,  
induces isomorphisms in cohomology up to degree $i$, and  a monomorphism 
in degree $i+1$.  Now, since $H^{i+1}(\cM)$ has the same dimension as 
$H^{i+1}(\cM[A,i])$, and thus as $H^{i+1}(\cM(A,i))$, 
the map $H(\alpha)$ is also an isomorphism in degree $i+1$. 

The $\dga$ morphism $\beta$ extends to a $\cga$ map 
$\gamma\colon \cM \to H^{\hdot}(\cM[A,i])$ as in diagram \eqref{eq:cdformal}, 
by sending the new generators to zero. 
Since the target of $\beta$ vanishes in degrees $k\ge i+2$ 
and has differential $d=0$, the map $\gamma$ is a $\dga$ 
morphism. Furthermore, since $\gamma \circ \alpha=\beta$, 
we infer that $\gamma$ induces isomorphisms in cohomology in degrees $k \leq i+1$. 
Since $H^{k}(\cM)= H^{k}(\cM[A,i])=0$ for $k\ge i+2$, we conclude that 
$H(\gamma)$ is an isomorphism in all degrees, i.e., $\gamma$ is a 
quasi-isomorphism.  

Finally, let $\phi\colon \cM\to \cM[A,i]$ be a quasi-isomorphism 
from the minimal model of $\cM[A,i]$ to this $\dga$. 
The maps $\phi$ and $\gamma$ define 
a weak equivalence between $\cM[A,i]$ and $ (H^{\hdot}(\cM[A,i]),0)$, 
thereby showing that $\cM[A,i]$ is formal. 
\end{proof}

Since $H^{\geq i+2}(\cM[A,i])=0$, the equivalence of conditions \eqref{qf3} and 
\eqref{qf4} in the above lemma also follows from the (quite different) proof 
of Proposition 3.4 from \cite{Macinic}; see Remark \ref{rem:maclem} for more 
on this.  We are now ready to prove descent for partial formality of $\dga$s.

\begin{theorem}
\label{thm:i-formalField}
Let $(A^{\hdot},d_A)$ be a \dga~over $\k$, and let $\k\subset \K$ be a 
field extension. Suppose $H^{\le i+1}(A)$ is finite-dimensional 
and $H^0(A)=\k$.  Then $(A^{\hdot},d_A)$ is $i$-formal 
if and only if $(A^{\hdot}\otimes \K,d_A\otimes \id_{\K})$ 
is $i$-formal.
\end{theorem}

\begin{proof}
Since $\dim H^{i+1}(\M(A,i)) \leq \dim H^{i+1}(A) < \infty$, we may 
apply Lemma \ref{lem:formalityeq} to infer that 
the \dga~$(A^{\hdot},d_A)$ is $i$-formal if and only if
the \dga~$\cM[A,i]$ is formal. 
By construction, $H^q(\cM[A,i])$ equals $H^q(A)$ for $q\le i$, 
injects into $H^{q}(A)$ for $q=i+1$, and vanishes 
for  $q>i+1$. Hence, in view of our hypothesis,  
$H^{\hdot}(\cM[A,i])$ is of finite-type. 
By Corollary \ref{cor:Kformal}, $\cM[A,i]$ is formal if and only if 
$\cM[A,i]\otimes \K$ is formal. 
By Lemma \ref{lem:formalityeq} again, 
this is equivalent to the $i$-formality of $\cM[A,i]\otimes \K$.
\end{proof}

\subsection{Formality notions for spaces}
\label{subsec:formal spaces}

To every space $X$, Sullivan \cite{Sullivan} associated in a functorial 
way a \dga~of `rational polynomial forms', denoted $A_{PL}^{\hdot}(X)$. 
As shown in \cite[\S 10]{FHT}, there is a natural identification 
$H^{\hdot}(A_{PL}^{\hdot}(X)) = H^{\hdot}(X,\Q)$ under which the respective 
induced homomorphisms in cohomology correspond. In particular, 
the weak isomorphism type of $A_{PL}^{\hdot}(X)$ depends only on 
the rational homotopy type of $X$. 

A \dga~$(A,d)$ over $\k$ is called a {\em model}\/ for the 
space $X$ if $A$ is weakly equivalent to Sullivan's algebra 
$A_{PL}(X;\k):=A_{PL}(X)\otimes_{\Q} \k$.  In other words,  
$\cM(A)$ is isomorphic to $\cM(X;\k) :=\cM(X) \otimes_{\Q} \k$, 
where $\cM(A)$ is the minimal model of $A$ and 
$\cM(X)$ is the minimal model of $A_{PL}(X)$. 
In the same vein, $A$ is an {\em $i$-model}\/ for $X$ if 
$(A,d)\simeq_i A_{PL}(X;\k)$.  
For instance, if $X$ is a smooth manifold, then the 
de Rham algebra $\Omega^{\hdot}_{dR}(X)$ is a model for $X$ 
over $\R$.

A space $X$ is said to be {\em formal}\/  
over $\k$ if the model $A_{PL}(X; \k)$ is formal, 
that is, there is a quasi-isomorphism $\cM(X;\k) \to (H^{\hdot}(X;\k),d=0)$.
Likewise, $X$ is said to be {\em $i$-formal}, for some $i\ge 0$, if 
there is an $i$-quasi-isomorphism $\cM(A_{PL}(X;\k),i) \to (H^{\hdot}(X;\k),d=0)$. 
Note that $X$ is $0$-formal if and only if $X$ is path-connected. 
Also, since a homotopy equivalence $X\simeq Y$ induces 
an isomorphism $H^{\hdot}(Y;\Q)\isom H^{\hdot}(X;\Q)$, it follows from 
Corollary \ref{cor:formal-equiv} that the $i$-formality property 
is preserved under homotopy equivalences.  

The following theorem of Papadima and Yuzvinsky 
\cite{Papadima-Yuzvinsky} nicely relates the properties 
of the minimal model of $X$ to the Koszulness of its 
cohomology algebra.

\begin{theorem}[\cite{Papadima-Yuzvinsky}]
\label{thm:py}
Let $X$ be a connected space with finite Betti numbers.  
If $\cM(X)\cong \cM(X,1)$, then $H^{\hdot}(X;\Q)$ is a Koszul algebra.
Moreover, if $X$ is formal, then the converse holds.
\end{theorem}

\begin{remark}
\label{rem:maclem}
In \cite[Prop.~3.4]{Macinic}, M\u{a}cinic shows that every 
$i$-formal space $X$ for which $H^{\geq i+2}(X;\Q)$ vanishes 
 is formal. In particular, the notions 
of formality and $i$-formality coincide for $(i+1)$-dimensional 
CW-complexes. In general, though, 
full formality is a much stronger condition than partial formality. 
\end{remark}

\begin{remark}
\label{rem:fm} 
There is a competing notion of $i$-formality, due to Fern\'{a}ndez and 
Mu\~{n}oz \cite{FM05}.  As explained in \cite{Macinic}, the two notions 
differ significantly, even for $i=1$. In what follows, we will use exclusively 
the classical notion of $i$-formality given above.
\end{remark}

As is well-known, the (full) formality property behaves well with respect to field 
extensions of the form $\Q\subset \k$.  Indeed, it follows from Halperin and Stasheff's 
Corollary \ref{cor:Kformal} that a connected space $X$ with finite Betti numbers 
is formal over $\Q$ if and only if $X$ is formal over $\k$.  
This result was first stated and proved by Sullivan \cite{Sullivan}, using different techniques.
An independent proof was given by Neisendorfer and Miller \cite{Neisendorfer-Miller78}
in the  simply-connected case. 

These classical results on descent of formality may be strengthened to a result 
on descent of partial formality. More precisely, using Theorem \ref{thm:i-formalField}, 
we obtain the following immediate corollary.

\begin{corollary}
\label{cor:pfext}
Let $X$ be a connected space with finite Betti numbers $b_1(X),\dots,b_{i+1}(X)$.  
Then $X$ is $i$-formal over $\Q$ if and only if $X$ is $i$-formal over $\k$.
\end{corollary}

\section{Groups, Lie algebras, and graded-formality}
\label{sect:gr holo}

We now turn to finitely generated groups, and to two graded Lie 
algebras attached to each such group. We put special emphasis on the relationship 
between these Lie algebras, leading to the notion of graded-formality.  

\subsection{Central filtrations on groups}
\label{subsec:gr}
We start with some general background on lower 
central series and the associated graded Lie algebra of a group. 
For more details on this classical topic, we refer to Lazard \cite{Lazard54} 
and Magnus et al.~\cite{Magnus-K-S}.
Let $G$ be a group.  For elements $x,y\in G$, let 
$[x,y]=xyx^{-1}y^{-1}$ be their group commutator. Likewise, for 
subgroups $H,K< G$, let $[H,K]$ be the subgroup of $G$ generated by 
all commutators $[x,y]$ with $x\in H$, $y\in K$.  

A \textit{(central) filtration}\/ on the group $G$ is a decreasing 
sequence  of subgroups, $G=\cF_1G> \cF_2G> \cF_3G>\cdots$, such that 
$[\cF_rG,\cF_sG]\subset \cF_{r+s}G$.  It is readily verified that, 
for each $k>1$, the group $\cF_{k+1}G$ is a normal 
subgroup of $\cF_{k}G$, and the quotient group $\gr^{\cF}_k(G)=\cF_kG/\cF_{k+1}G$ 
is abelian.   As before, let $\k$ be a field of characteristic $0$. The direct sum 
\begin{equation}
\label{eq:gr}
\gr^{\cF}(G;\k)=\boplus_{k\geq 1} \gr^{\cF}_k(G) \otimes_{\Z}\k 
\end{equation}
 is a graded Lie algebra over $\k$, with Lie bracket induced from 
the group commutator:  If $x\in \cF_r G $ and $y\in \cF_s G$, then 
$[x+\cF_{r+1}G,y+\cF_{s+1}G]=xyx^{-1}y^{-1}+\cF_{r+s+1}G$. 
We can view $\gr^{\cF}(-;\k)$ as a functor from groups to graded $\k$-Lie algebras.  
Moreover, $\gr^{\cF}(G;\K)=\gr^{\cF}(G;\k)\otimes_{\k} \K$, for any field 
extension $\k\subset \K$. (Once again, if the underlying ring in a tensor product
is understood, we will write $\otimes$ for short.)

Let $H$ be a normal subgroup of $G$, and let $Q=G/H$ be the quotient 
group. Define filtrations on $H$ and $Q$ by $\widetilde{\cF}_kH=\cF_k G\cap H$ 
and $\overline{\cF}_kQ=\cF_k G/\widetilde{\cF}_k H$, respectively.  
We then have the following classical result of Lazard.

\begin{prop}[Thm.~2.4 in \cite{Lazard54}]
\label{prop:lazard}
The canonical projection $G\surj G/H$ induces a natural 
isomorphism of graded Lie algebras, 
$\gr^{\cF}(G;\k)/\gr^{\widetilde{\cF}}(H;\k) 
\isom  \gr^{\overline{\cF}}(G/H;\k)$.
\end{prop}

\subsection{The associated graded Lie algebra}
\label{subsec:assoc gr}
Any group $G$ comes endowed with the lower central series (LCS) 
filtration $\{\Gamma_kG\}_{k\geq 1}$, defined inductively by 
$\Gamma_1G=G$ and $\Gamma_{k+1}G=[\Gamma_kG,G]$. 
If $\Gamma_kG\neq 1$ but $\Gamma_{k+1}G=1$, then 
$G$ is said to  be a \textit{$k$-step nilpotent group}.  In general, 
though, the LCS filtration does not terminate.

The Lie algebra $\gr(G;\k)=\gr^{\Gamma}(G;\k)$ is called the 
\textit{associated graded Lie algebra} (over $\k$) of the group $G$.  
For instance, if $F=F_n$ is a free group of rank $n$, then $\gr(F;\k)$ 
is the free graded Lie algebra $\Lie(\k^n)$.
A group homomorphism $f\colon G_1\rightarrow G_2$ induces a 
morphism of graded Lie algebras, $\gr(f;\k)\colon \gr(G_1;\k)\rightarrow 
\gr(G_2;\k)$; moreover, if $f$ is surjective, then $\gr(f;\k)$ is also surjective. 

For each $k\ge 2$, the factor group $G/\Gamma_k(G)$ is the 
maximal $(k-1)$-step nilpotent quotient of $G$. 
The canonical projection $G\to G/\Gamma_k(G)$ induces an epimorphism 
$\gr(G;\k) \to \gr(G/\Gamma_k(G);\k)$, which is an isomorphism in degrees 
$s< k$. 

From now on, unless otherwise specified, we will assume that the 
group $G$ is finitely generated.  That is, there is a 
free group $F$ of finite rank, and an epimorphism 
$\varphi\colon F\surj G$.  Let $R=\ker(\varphi)$; then $G=F/R$ 
is called a presentation for $G$.  Note that 
the induced morphism 
$\gr(\varphi;\k)\colon \gr(F;\k)\to \gr(G;\k)$ is surjective.  Thus, 
$\gr(G;\k)$ is a finitely generated Lie algebra, with generators 
in degree $1$.  

Let $H\triangleleft G$ be a normal subgroup, and let $Q=G/H$.  
Denote by $\widetilde{\Gamma}_r H=\Gamma_r{G}\cap  H$ the induced 
filtration on $H$. It is readily seen that the filtration $\overline{\Gamma}_rQ=
\Gamma_rG/\widetilde{\Gamma}_r H$ coincides with the LCS filtration on $Q$. 
Hence, by Proposition \ref{prop:lazard}, 
\vspace*{-3pt}
\begin{equation}
\label{eq:gammagr}
\gr(Q;\k) \cong \gr(G;\k)/\gr^{\widetilde{\Gamma}}(H;\k).  
\end{equation}

Now suppose $G=H\rtimes Q$ is a semi-direct product of groups.  In general, 
there is not much of a relation between the respective associated graded Lie algebras. 
Nevertheless, we have the following well-known result of Falk and Randell \cite{FalkRandell},
which shows that $\gr(G;\k)=\gr(H;\k)\rtimes \gr(Q;\k)$ for `almost-direct' products of groups. 

\begin{theorem}[Thm.~3.1 in \cite{FalkRandell}]
\label{thm:fr}
Let $G=H\rtimes Q$ be a semi-direct product of groups, and suppose $Q$ acts trivially 
on $H_{\ab}$.  Then the filtrations $\{\widetilde{\Gamma}_r H\}_{r\ge 1}$ and 
$\{\Gamma_r{H}\}_{r\ge 1}$ coincide, and there is a split exact sequence 
of graded Lie algebras,  $0\to \gr(H;\k) \to \gr(G;\k)\to \gr(Q;\k)\to 0$.
\end{theorem}

\subsection{The holonomy Lie algebra}
\label{subsec:holo lie}

The holonomy Lie algebra of a finitely generated group 
was introduced by K.-T.~Chen \cite{Chen73} and Kohno \cite{Kohno}, 
and was further studied in \cite{Markl-Papadima, Papadima-Suciu04, SW-holo}. 

\begin{definition}
\label{def:holo group}
Let $G$ be a finitely generated group.   The \emph{holonomy Lie algebra}\/  
of $G$ is the holonomy Lie algebra of the cohomology ring 
$A=H^{\hdot}(G;\k)$, that is, 
\begin{equation}
\label{eq:hololie}
\fh(G;\k)=\Lie(H_1(G;\k))/\langle\im (\partial_G)\rangle,
\end{equation}
where $\partial_G$ is the dual to the cup-product map 
$\cup_G\colon H^1(G;\k)\wedge H^1(G;\k)\to H^2(G;\k)$. 
\end{definition}

By construction, $\fh(G;\k)$ is a quadratic Lie algebra.  
If $f\colon G_1\to G_2$ is a group homomorphism, then 
the induced map in cohomology, $H^1(f)\colon H^1(G_2;\k)\to H^1(G_1;\k)$,  
yields a morphism of graded Lie algebras, $\fh(f;\k)\colon \fh(G_1;\k)\to \fh(G_2;\k)$.  
Moreover, if $f$ is surjective, then $\fh(f;\k)$ is also surjective. 
Finally, $\h(G;\K)=\h(G;\k)\otimes_{\k} \K$, for any field 
extension $\k\subset \K$

In the definition of the holonomy Lie algebra of $G$, we used the 
cohomology ring of a classifying space $K(G,1)$.  As the next 
lemma shows, we may replace this space by any other connected 
CW-complex with the same fundamental group. 

\begin{lemma}
\label{lem:holo inv}
Let $X$ be a connected CW-complex with $\pi_1(X)=G$.  Then 
$\fh(H^{\hdot}(X;\k))\cong \fh(G;\k)$.
\end{lemma}

\begin{proof}
We may construct a classifying space for $G$ by adding cells of 
dimension $3$ and higher to $X$ in a suitable way. The inclusion 
map, $j\colon X\to K(G,1)$, induces a map on cohomology rings, 
$H(j)\colon H^{\hdot}(K(G,1);\k)\to H^{\hdot}(X;\k)$, which is an isomorphism 
in degree $1$ and an injection in degree $2$. 
Consequently, $H^2(j)$ restricts to an isomorphism from 
$\im(\cup_G)$ to $\im(\cup_{X})$. Taking duals, we find that 
$\im(\partial_{X})=\im(\partial_G)$. The conclusion follows.
\end{proof}

In particular, if $K_G$ is the $2$-complex associated to a 
presentation of $G$, then $\fh(G;\k)$ is isomorphic to $\fh(H^{\hdot}(K_G;\k))$.   Let 
$\bar{\phi}_n(G):=\dim \fh_n(G;\k)$ be the dimensions of the graded 
pieces of the holonomy Lie algebra of $G$.  The next corollary is an 
algebraic version of the LCS formula from Papadima and Yuzvinsky \cite{Papadima-Yuzvinsky}, 
but with no formality assumption.

\begin{corollary}
\label{cor:holokoszul}
Let $X$ be a connected CW-complex with $\pi_1(X)=G$, let $A=H^{\hdot}(X;\k)$ 
be its cohomology algebra, and let $\qA$ be the quadratic closure of $A$. 
Then $\prod_{n\geq 1}(1-t^n)^{\bar{\phi}_n(G)} = \sum_{i\ge 0} b_{ii} t^i$, 
where $b_{ii}=\dim \Ext^i_A(\k,\k)_i$.  Moreover, if 
$\qA$ is a Koszul algebra, then 
$\prod_{n\geq 1}(1-t^n)^{\bar{\phi}_n}=\Hilb(\qA,-t)$.
\end{corollary}

\begin{proof}
The first claim follows from Lemma \ref{lem:holo inv}, the Poincar\'e--Birkhoff--Witt 
formula \eqref{eq:PBW}, and L\"{o}f\-wall's formula from Proposition \ref{prop:yoneda}.  
The second claim follows from the Koszul duality formula stated in 
Corollary \ref{cor:kdual holo}.
\end{proof}

\subsection{A comparison map}
\label{subsec:Phimap}

Once again, let $G$ be a finitely generated group.   The 
next lemma is known; for a proof and more details, 
we refer to \cite{SW-holo}.

\begin{lemma}[\cite{Markl-Papadima, Papadima-Suciu04}]
 \label{lem:holoepi}
There exists a natural epimorphism of graded $\k$-Lie algebras, 
$\Phi_G\colon \fh(G; {\k}) \surj \gr(G;\k)$, 
inducing isomorphisms in degrees $1$ and $2$.  Furthermore, 
this epimorphism is natural with respect to field extensions $\k\subset \K$.
\end{lemma}
 
\begin{prop}
 \label{prop:closure}
Let $V=H_1(G;\k)$. 
Suppose the associated graded Lie algebra $\g=\gr(G;\k)$ 
has presentation $\Lie(V)/\fr$. 
Then the holonomy Lie algebra $\fh(G;\k)$ has presentation 
$\Lie(V)/\langle \fr_2\rangle $, where $\fr_2=\fr\cap \Lie_2(V)$.  Furthermore, 
if $A=U(\g)$, then $\fh(G;\k)=\fh\big((\qA)^{!}\big)$. 
\end{prop}

\begin{proof}
We have a natural exact sequence, which was first established 
in a particular case by Sullivan \cite{Sullivan75}  
and in general by Lambe \cite{Lambe86},
\vspace*{-4pt}
\begin{equation}
\label{eq:sl}
\xymatrixcolsep{18pt}
\xymatrix{0\ar[r]& (\Gamma_2G/\Gamma_3 G\otimes \k)^* 
\ar^(.46){\beta}[r]& H^1(G;\k)\wedge H^1(G;\k) \ar^(.62){\cup}[r]& H^2(G;\k)},  
\end{equation}
where $\beta$ is the dual of Lie bracket product. 
Taking the dual of \eqref{eq:sl}, we find that  
$\im(\partial_G)=\ker(\beta^*)$.
Hence, $\langle \fr_2 \rangle =\langle \im(\partial_G)\rangle$
as ideals of $\Lie(V)$; thus, $\fh(G;\k)=\Lie(V)/ \langle \fr_2 \rangle$. 
The last claim follows from Corollary \ref{cor:quadraticgr}.
\end{proof}

Recall that we denote by $\phi_n(G)$ and $\bar\phi_n(G)$ the dimensions 
on the $n$-th graded pieces of $\gr(G;\k)$ and $\fh(G;\k)$, respectively.  
By Lemma \ref{lem:holoepi}, 
$\bar\phi_n(G)\ge \phi_n(G)$, for all $n\ge 1$, and equality always holds for 
$n\le 2$.  Nevertheless, these inequalities can be strict for $n\ge 3$.

\subsection{Graded-formality}
\label{subsec:grf}

We continue our discussion of the associated graded 
and holonomy Lie algebras of a finitely generated group 
with a formality notion that will be important in the sequel.  

\begin{definition}
\label{def:graded-1-formal}
A finitely generated group $G$ is {\em graded-formal}\/ (over $\k$) if 
the canonical projection $\Phi_G\colon \fh(G;\k)\surj \gr(G;\k)$ 
is an isomorphism of graded Lie algebras. 
\end{definition}

This notion was introduced by Lee in \cite{Lee}, where it is 
called graded $1$-formality.   Next, we give two alternate 
definitions, which oftentimes are easier to verify (see \cite{SW-holo} for a proof).

\begin{lemma}
\label{lem:graded-1-formal}
A finitely generated group $G$ is graded-formal over $\k$ if and only if 
either 
\begin{enumerate}
\item \label{g1f1}  
$\gr(G;\k)$ is quadratic, or
\item \label{g1f2}
$\dim_{\k} \h_n(G;\k) = \dim_{\k} \gr_n(G;\k)$, for all $n\ge 1$.
\end{enumerate}
\end{lemma}
 
The lemma implies that the definition 
of graded-formality is independent of the choice of coefficient field $\k$ 
of characteristic $0$. More precisely, we have the following corollary.

\begin{corollary}
\label{lem:gf-ext}
A finitely generated group $G$ is graded-formal over $\k$ if and only if 
it is graded-formal over $\Q$. 
\end{corollary}

\begin{proof}
The dimension of a finite-dimensional vector space does not change 
upon extending scalars.  The conclusion follows 
at once from Lemma \ref{lem:graded-1-formal}\eqref{g1f2}.
\end{proof}

\subsection{Split injections}
\label{subsec:split}

We are now in a position to state and prove the main result of 
this section, which proves the first part of Theorem \ref{thm:intro formality} 
from the Introduction.

\begin{theorem}
\label{thm:graded1formality}
Let $G$ be a finitely generated 
group. Suppose there is a split monomorphism $\iota\colon K\rightarrow G$.
If $G$ is a graded-formal group, then $K$ is also graded-formal. 
\end{theorem}

\begin{proof}  
In view of our hypothesis, we have an epimorphism $\sigma\colon G\surj K$ 
such that $\sigma\circ\iota=\id$.  In particular, $K$ is also finitely generated.  
Furthermore, the induced maps $\fh(\iota)$ and $\gr(\iota)$ are also injective.
Let $\pi\colon F\surj G$ be a presentation for $G$. 
There is then an induced presentation for $K$, 
given by the composition $\sigma\circ \pi\colon F\surj K$. 
By Lemma \ref{lem:holoepi}, there exist epimorphisms $\Phi_1$ 
and $\Phi$ making the following diagram commute:
\vspace*{-4pt}
\begin{equation}
\label{eq:cd iota}
\begin{gathered}
\xymatrixcolsep{22pt}
\xymatrixrowsep{19pt}
\xymatrix{
\fh(K;\k)\ar@{^(->}[d]^{\fh(\iota)} \ar@{->>}[r]^{\Phi_1} & \gr(K;\k)\ar@{^(->}[d]^{\gr(\iota)} 
\phantom{\,.} \\
\fh(G;\k) \ar@{->>}[r]^{\Phi} & \gr(G;\k) \,.
}
\end{gathered}
\end{equation}

If the group $G$ is graded-formal, then $\Phi$ is an isomorphism of 
graded Lie algebras.   Hence, the epimorphism $\Phi_1$ is also 
injective, and so $K$ is a graded-formal.
\end{proof}

\begin{theorem}
\label{thm:rtimes graded}
Let $G=K\rtimes Q$ be a semi-direct product of finitely 
generated groups. If $G$ is graded-formal, then $Q$ is graded-formal.  
If, moreover, $Q$ acts trivially on $K_{\ab}$,  
then $K$ is also graded-formal. 
\end{theorem}

\begin{proof}
The first assertion follows at once from 
Theorem \ref{thm:graded1formality}.   So 
assume $Q$ acts trivially on $K_{\ab}$.  
By Theorem \ref{thm:fr}, there exists a split exact sequence
of graded Lie algebras, which we record in the top row of the 
next diagram. 
\begin{equation}
\label{eq:ladder}
\begin{gathered}
\xymatrix{
0\ar[r] & \gr(K;\k) \ar[r] &\gr(G;\k)\ar@/_8pt/[l] \ar[r] 
& \gr(Q;\k)\ar@/_8pt/[l]\ar[r] &0\,\,\\
& \fh(K;\k) \ar[r]\ar@{->>}[u]_{\Phi_K}&\fh(G;\k)\ar@{.>}@/_8pt/[l] \ar[r]  
\ar[u]^{\cong}_{\Phi_G}
& \fh(Q;\k)\ar[r]\ar@/_8pt/[l] \ar[u]^{\cong}_{\Phi_Q} &0\, .
}
\end{gathered}
\end{equation}

Let $\iota\colon K\to G$ be the inclusion map.  
By the above, we have an epimorphism 
$\sigma\colon \gr(G;\k)\surj \gr(K;\k)$ 
such that $\sigma\circ \gr(\iota)=\id$, and so  
$\gr(K;\k)$ is finitely generated. 
By Proposition \ref{prop:closure}, the map $\sigma$ induces a 
morphism  $\bar\sigma\colon \fh(G;\k) \to \fh(K;\k)$ such that 
$\bar\sigma\circ \h(\iota)=\id$.  Consequently, $\h(\iota)$  
is injective. 
Therefore, the morphism $\Phi_K\colon \fh(K;\k)\rightarrow \gr(K;\k)$ 
is also injective, and so $K$ is graded-formal.
\end{proof}

If the second hypothesis of Theorem \ref{thm:rtimes graded} 
does not hold, the subgroup $K$ may not be graded-formal, even 
when $G$ is $1$-formal. This is illustrated by the following example, 
adapted from \cite{Papadima-Suciu09}. 

\begin{example}
\label{ex:semiproduct1}
Let $K=\langle x,y \mid [x, [x, y]], [y, [x, y]] \rangle$ be the discrete 
Heisenberg group. Consider the semidirect product $G =K\rtimes_{\phi} \Z$, 
defined by the automorphism  $\phi\colon K\rightarrow K$ 
given by $x \rightarrow y$, $y\rightarrow xy$.
We have that $b_1(G)=1$, and so $G$ is $1$-formal, 
yet  $K$  is not graded-formal.
\end{example}

\subsection{Products and coproducts}
\label{subsec:prodcoprod}

We conclude this section with a discussion of the functors 
$\gr$ and $\fh$ and the notion of graded-formality behave 
with respect to products and coproducts. 

\begin{lemma}[\cite{Lichtman80, Papadima-Suciu14product}]
\label{lem: product}
The functors $\gr$ and $\fh$ preserve products and coproducts, that is, 
we have the following natural isomorphisms of graded Lie algebras, 
\[
\begin{cases}
 \gr(G_1\times G_2;\k)\cong \gr(G_1;\k)\times \gr(G_2;\k)\\
 \gr(G_1* G_2;\k)\cong \gr(G_1;\k)* \gr(G_2;\k),
\end{cases}
\text{and}\quad 
\begin{cases}
\fh(G_1\times G_2;\k)\cong \fh(G_1;\k)\times \fh(G_2;\k)\\
\fh(G_1* G_2;\k)\cong \fh(G_1;\k)* \fh(G_2;\k).
\end{cases}
\]
\end{lemma}
\begin{proof}
The first statement on the $\gr(-)$ functor is well-known, while the second statement 
is the main theorem from \cite{Lichtman80}.   The statements regarding the $\fh(-)$ functor 
can be found in \cite{Papadima-Suciu14product}.
\end{proof}

Regarding graded-formality, we have the following result, which sharpens  
and generalizes Lemma 4.5 from Plantiko \cite{Plantiko96}, and proves the first 
part of Theorem \ref{thm:product formality} from the Introduction.

\begin{prop}
\label{prop:productquadratic}
Let  $G_1$ and $G_2$ be two finitely generated groups. Then 
the following conditions are equivalent.
\begin{enumerate}
\item \label{gf1}  $G_1$ and $G_2$ are graded-formal.
\item \label{gf2}  $G_1* G_2$  is graded-formal.
\item \label{gf3}  $G_1\times G_2$ is graded-formal.
\end{enumerate}
\end{prop}

\begin{proof}
Since there exist split injections from $G_1$ and $G_2$ to 
the product $G_1\times G_2$ and coproduct $G_1*G_2$, 
Theorem \ref{thm:graded1formality} shows that implications 
\eqref{gf2}$\Rightarrow$\eqref{gf1} and \eqref{gf3}$\Rightarrow$\eqref{gf1} 
hold. Implications 
\eqref{gf1}$\Rightarrow$\eqref{gf2} and \eqref{gf1}$\Rightarrow$\eqref{gf3} 
follow from Lemma \ref{lem: product} and the naturality of the map 
$\Phi$ from Definition \ref{def:graded-1-formal}.
\end{proof}

\section{Malcev Lie algebras and filtered-formality}
\label{sect:malcev}

In this section we consider the Malcev Lie algebra of a finitely generated group, 
and study the ensuing notions of filtered-formality and $1$-formality.

\subsection{Prounipotent completions and Malcev Lie algebras}
\label{subsec:mg}

Let $G$ be a finitely generated group, and let 
$\{\Gamma_kG\}_{k\geq 1}$ be its LCS filtration. The successive 
quotients of $G$ by these normal subgroups form a tower of 
finitely generated, nilpotent groups,  
$\cdots \to  G/\Gamma_4 G \to G/\Gamma_3 G \to G/\Gamma_2G = G_{\ab}$.

Let $\k$ be a field of characteristic $0$.
It is possible to replace each nilpotent quotient $N_k=G/\Gamma_k G$ by 
$N_k \otimes \k$, the (rationally defined) nilpotent Lie group associated to 
the discrete, torsion-free nilpotent group $N_k/{\rm tors}(N_k)$ via a 
procedure which will be discussed in \S\ref{subsec:nilpotent groups}.
The corresponding inverse limit, 
$\fM(G;\k)=\varprojlim\nolimits_k \big((G/\Gamma_k G)\otimes \k\big)$,
is a prounipotent, filtered Lie group over $\k$, 
called the \textit{prounipotent completion}, or \textit{Malcev completion}\/ of $G$ 
over $\k$.  We denote by $\kappa\colon G\to \fM(G,\k)$ the canonical 
homomorphism from $G$ to its completion.

Let $\mathfrak{Lie}((G/\Gamma_k G)\otimes \k)$ be 
the Lie algebra of the nilpotent Lie group $(G/\Gamma_k G)\otimes \k$.
The pronilpotent Lie algebra 
\vspace*{-3pt}
\begin{equation}
\label{eq:malcevLie}
\fm(G;\k):=\varprojlim\nolimits_k \mathfrak{Lie}((G/\Gamma_k G)\otimes \k),
\end{equation}
with the inverse limit filtration, is called the {\em Malcev Lie algebra}\/ of $G$ (over $\k$).  
By construction, $\fm(-;\k)$ is a functor 
from the category of finitely generated groups to the category of
complete, separated, filtered $\k$-Lie algebras.

\subsection{Quillen's construction}
\label{subsec:quillen}

A different approach was taken by Quillen in \cite[App.~A]{Quillen69}. 
Let us briefly recall his construction.  
The group-algebra $\k{G}$ has a natural Hopf algebra structure,
with comultiplication $\Delta\colon \k{G}\to \k{G}\otimes \k{G}$ 
given by $\Delta(g)=g\otimes g$ for $g\in G$, and counit the 
augmentation map $\varepsilon\colon \k{G}\to \k$. 
Moreover, the set of {\em group-like}\/ elements in this Hopf algebra, 
i.e., those elements $x$ for which $\Delta(x)=x\otimes x$, 
gets  identified with $G$ under the canonical inclusion $G\inj \k{G}$. 

The powers of the augmentation ideal $I=\ker (\varepsilon)$ 
form a descending, multiplicative filtration of $\k{G}$. The $I$-adic 
completion of the group-algebra,  $\widehat{\k{G}}=\varprojlim_k \k{G}/I^k$, 
comes equipped with a descending filtration, whose $k$-th term is 
$\widehat{I^k}=\varprojlim_{j\geq k} I^k/I^j$.
Define the completed tensor product $\widehat{\k{G}}\, 
\hat{\otimes}\,  \widehat{\k{G}}$ as the
completion of ${\k{G}}\otimes {\k{G}}$ with respect 
to the natural tensor product filtration.  
Applying the $I$-adic completion functor to the map 
$\Delta$ yields a comultiplication map   
$\widehat{\Delta} \colon \widehat{\k{G}}\to \widehat{\k{G}}\, 
\hat{\otimes}\,  \widehat{\k{G}}$, which makes $\widehat{\k{G}}$ 
into a complete Hopf algebra. 
It is then apparent that the canonical map 
to the completion, $\iota\colon \k{G}\to \widehat{\k{G}}$, 
is a morphism of filtered Hopf algebras. 

An element $x\in \widehat{\k{G}}$ is called {\em primitive}\/ if 
$\widehat{ \Delta} x=x \hat{\otimes} 1+1\hat{\otimes} x$.
The set of all primitive elements in $\widehat{\k{G}}$, with bracket 
$[x,y]=xy-yx$, and endowed with the induced filtration, is a 
 Lie algebra, isomorphic to the Malcev Lie algebra of $G$, that is, 
 $\fm(G;\k) \cong  \Prim \big( \widehat{\k{G}} \big)$.

The filtration topology on $\widehat{\k{G}}$ is a metric topology; 
hence, the filtration topology on $\fm(G;\k)$ is also metrizable, 
and thus separated. We shall denote by $\gr(\fm(G;\k))$ the 
associated graded Lie algebra of $\fm(G;\k)$ with respect 
to the induced inverse limit filtration.
 
The set of all group-like elements in $ \widehat{\k{G}}$, 
with multiplication inherited from $\widehat{\k{G}}$, 
forms a group, denoted $M(G;\k)$.  This group comes 
endowed with a complete, separated filtration, with terms 
$M(G;\k) \cap (1+\widehat{I^k})$. As shown by S.~Jennings 
and Quillen, there is a filtration-preserving 
isomorphism $M(G,\k)\cong\fM(G;\k)$, 
see  Massuyeau \cite{Massuyeau12} for details. 
Furthermore, there is a one-to-one, filtration-preserving 
correspondence between primitive elements 
and group-like elements via the exponential 
and logarithmic maps 
\vspace*{-3pt}
\begin{equation}
\label{eq:explog}
\xymatrix{
 \fM(G;\k)\subset 1+\widehat{I} \ar@/_0.8pc/[rr]|{~\log~} &  
  & \widehat{I} \supset\fm(G;\k)\ar@/_1.1pc/[ll]|{~\exp~} }.
\end{equation} 

Restricting the canonical map $\iota\colon \k{G}\to \widehat{\k{G}}$ 
to group-like elements, we obtain a 
homomorphism from $G$ to its prounipotent completion, 
$\kappa\colon G\to \fM(G;\k)$.  Composing this homomorphism 
with the logarithmic map, $\log\colon \fM(G;\k)\to \fm(G;\k)$, we 
obtain a filtration-preserving map, $\rho\colon G\to \fm(G;\k)$.  
As shown by Quillen in \cite{Quillen68}, the map $\rho$ 
induces an isomorphism of graded Lie algebras,  
\vspace*{-3pt}
\begin{equation}
\label{eq:quillen}
\xymatrixcolsep{18pt}
\xymatrix{\gr(\rho)\colon \gr(G;\k) \ar^(.52){\cong}[r] & \gr(\fm(G;\k))}.
\end{equation}
In particular, $\gr(\fm(G;\k))$ is generated in degree $1$. 

\subsection{Minimal models and Malcev Lie algebras}
\label{subsec:sullivan-malcev}

Every group $G$ has a classifying space $K(G,1)$, which can be chosen 
to be a connected CW-complex.  Such a CW-complex is unique up to homotopy, 
and thus, up to rational homotopy equivalence.  Hence, by the discussion 
from \S\ref{subsec:formal spaces} the weak 
equivalence type of the Sullivan algebra $A=A_{PL}(K(G,1))$ depends 
only on the isomorphism type of $G$.  
By Theorem \ref{thm:mm}, the \dga~$A\otimes_{\Q} \k$
has a $1$-minimal model, $\cM(A\otimes_{\Q} \k,1)$, 
unique up to isomorphism.   Moreover, the assignment 
$G\leadsto \cM(A\otimes_{\Q} \k,1)$  is functorial.  

Assume now that the group $G$ is finitely generated. 
Let $\cM=\cM(G;\k)$ be a $1$-minimal model of $G$, with the canonical filtration constructed 
in \eqref{eq:filtration-minimal}. 
The starting point is the finite-dimensional 
vector space $\cM_1^1=V_1:=H^1(G;\k)$.   
Each sub-$\dga$~$\cM_i$ is a Hirsch extension of $\cM_{i-1}$ 
by the finite-dimensional vector space $V_i:=\ker\big(H^2(\cM_{i-1})\to H^2(A)\big)$.
Define $\fL(G;\k)=\varprojlim_i \fL_i(G;\k)$ as the 
pronilpotent Lie algebra  associated to the $1$-minimal model 
$\cM(G;\k)$ in the manner described 
in \S \ref{subsec:sullivan-holonomy}, and note that the assignment 
$G\leadsto \fL(G;\k)$ is also functorial. 

\begin{theorem}[\cite{Cenkl-Porter81, Sullivan}]
\label{thm:sullivan}
There exist a natural isomorphism between the towers of nilpotent 
Lie algebras $\{\fL_{i}(G; \k)\}_i$ and $\{\fm(G/\Gamma_iG; \k)\}_i$. 
Hence, there is a functorial isomorphism $\fL(G;\k)\cong \fm(G;\k)$ of 
complete, filtered Lie algebras.
\end{theorem}  

For further details, we refer to \cite[Thm. 7.5]{FHT2} 
and \cite[Thm. 13.2]{Griffiths-Morgan13}.
The functorial  isomorphism $\fm(G;\k) \cong \fL(G;\k)$, together with the 
dualization correspondence $\fL(G;\k)\leftrightsquigarrow\cM(G;\k)$ define adjoint functors 
between the category of Malcev Lie algebras of finitely generated groups and 
the category of $1$-minimal models of finitely generated groups.

\subsection{Filtered-formality of groups}
\label{subsec:f1f}
We now define the notion of filtered-formality  for groups (or, 
weak formality  \cite{Lee}), based on the notion of filtered-formality for 
Lie algebras from Definition \ref{def:filt formal}.

\begin{definition}
\label{def:filt-1-formal}
A finitely generated group $G$ is said to be \emph{filtered-formal}\/  
(over $\k$) if its Malcev Lie algebra $\fm(G;\k)$, endowed with 
the inverse limit filtration from \eqref{eq:malcevLie}, is filtered-formal. 

\end{definition}

Here are some more direct ways to think of this notion. 

\begin{prop}
\label{prop:ffcrit}
A finitely generated group $G$ is filtered-formal over $\k$ if and only if 
either 
\begin{enumerate}
\item \label{ffc1}  
$\fm(G;\k) \cong \widehat{\gr}(G;\k)$ as filtered Lie algebras, or 
\item \label{ffc2}
$\fm(G;\k)$  admits a homogeneous presentation.
\end{enumerate}
\end{prop}

\begin{proof}
\eqref{ffc1} The isomorphism \eqref{eq:quillen} implies that 
$\gr(\fm(G;\k)) \cong \gr(G;\k)$.  The forward implication 
follows straight from the definitions, while the backward 
implication follows from Lemma \ref{lem:filtiso}. 

\eqref{ffc2} 
Choose a presentation $\gr(G;\k)=\Lie(H_1(G;\k))/\mathfrak{r}$, 
where $r$ is a homogeneous ideal.  Using Lemma \ref{lem:presbar}, 
we obtain a homogeneous presentation for the Malcev Lie algebra of $G$, 
of the form $\fm(G;\k)=\widehat{\Lie}(H_1(G;\k))/\overline{\mathfrak{r}}$. 
Conversely, if such a presentation exists, then 
$\fm(G;\k) \cong \widehat{\fg}$, where $\fg=\Lie(H_1(G;\k))/\mathfrak{r}$.
\end{proof}

The notion of filtered-formality can also be interpreted in terms of 
minimal models.  Let $\cM(G;\k)$ be the $1$-minimal model 
of $G$, endowed with the canonical filtration, 
which is the minimal \dga~ dual to the Malcev Lie algebra $\fm(G;\k)$  
under the correspondence described in \S\ref{subsec:sullivan-holonomy}.  
Likewise, let $\cN(G;\k)$ be the minimal \dga~ (generated in degree $1$) 
corresponding to the prounipotent Lie algebra $\widehat{\gr}(G;\k)$.  
Recall that both $\cM(G;\k)$ and $\cN(G;\k)$ come equipped with 
increasing  filtrations as in \eqref{eq:filtration-minimal}, which correspond to the 
inverse limit filtrations on $\fm(G;\k)$   and $\widehat{\gr}(G;\k)$, respectively.  

\begin{prop}
\label{prop:ffcrit2}
A finitely generated group $G$ is filtered-formal over $\k$ if and only if either
\begin{enumerate}
\item \label{ffc21}  
there is a filtration-preserving $\cdga$ isomorphism between $\cM(G;\k)$ 
and $\cN(G;\k)$, or 
\item \label{ffc22}
there is a $\cdga$ isomorphism $\cM(G;\k)\cong 
\cN(G;\k)$ inducing the identity on first cohomology.
\end{enumerate}
\end{prop}

\begin{proof}
\eqref{ffc21} Recall Proposition \ref{prop:ffcrit} that $G$ is filtered-formal if and only if 
$\fm(G;\k)\cong \widehat{\gr}(G;\k)$, as filtered Lie algebras.  Dualizing, 
this condition becomes equivalent to $\cM(G;\k)\cong \cN(G;\k)$, 
as filtered $\dga$s. 

\eqref{ffc22}  
Recall that $G$ is filtered-formal if and only if 
$\fm(G;\k)\cong \widehat{\gr}(\fm(G;\k))$ inducing the identity on their
associated graded Lie algebras. Likewise, 
both $\cM_1^1$ and $\cN_1^1$ can be canonically 
identified with $\gr_1(G;\k)^*=H^1(G;\k)$.    
The desired conclusion follows.
\end{proof}

Here is another description of filtered-formality, suggested to us by R.~Porter.

\begin{theorem}
\label{thm:filtmin}
A finitely generated group $G$ is filtered-formal over $\k$ if and only 
if the canonical $1$-minimal model $\cM(G;\k)$ is filtered-isomorphic 
to a $1$-minimal model $\cM$ with positive Hirsch weights.
\end{theorem}

\begin{proof}
First suppose $G$ is filtered-formal, and let $\cN=\cN(G;\k)$ be the minimal 
$\dga$ dual to $\fL=\widehat{\gr}(G,\k)$.  By Proposition \ref{prop:ffcrit2},  
this $\dga$ is a $1$-minimal model for $G$.  Since by construction 
$\fL=\widehat{\gr}(\fL)$, Lemma \ref{lem:positivegraded} shows 
that the differential on $\cN$ is homogeneous with respect to the 
Hirsch weights.

Now suppose $\cM$ is a $1$-minimal model for $G$ over $\k$, 
with homogeneous differential on Hirsch weights.  
By Lemma \ref{lem:positivegraded} 
again, the dual Lie algebra $\fL(\cM)$ is filtered-formal.
On the other hand,  the assumption that $\cM\cong \cM(G;\k)$ 
and Theorem \ref{thm:sullivan} together imply that $\fL(\cM)\cong \fm(G;\k)$.   
Hence, the group $G$ is filtered-formal by Definition \ref{def:filt-1-formal}.
\end{proof}

We conclude this section by showing that the definition of filtered-formality is 
independent of the choice of coefficient field $\k$ of characteristic $0$. 
We would like to thank Y.~Cornulier for asking whether 
the next result holds, and for pointing out the connection it 
would have with \cite[Thm.~3.14]{Cornulier14}. 

\begin{prop}
\label{prop:field-filtered}
Let $G$ be a finitely generated group, and let $\Q\subset \k$ 
be a field extension.  Then $G$ is filtered-formal over $\Q$ 
if and only if $G$ is filtered-formal over $\k$.
\end{prop}

\begin{proof}
Write $\fm=\fm(G;\Q)$, and let $\fg=\gr(G;\Q)$, which 
we identify with $\gr^{\widehat{\Gamma}}(\fm)$.
The claim follows from Theorem \ref{thm:ffdescent}.
\end{proof}

\section{Filtered-formality and $1$-formality}
\label{sec:filtered-1-formality}

In this section, we consider the $1$-formality property of finitely generated groups, 
and the way it relates to Massey products, graded-formality, and filtered-formality. 
We also study the way various formality properties behave under free and direct 
products, as well as retracts.

\subsection{$1$-formality of groups}
\label{subsec:1-formal}
 
Let $\k$ be a field of characteristic $0$. We start with a basic definition.  

\begin{definition}
\label{def:1-f}
A finitely generated group $G$ is called \emph{$1$-formal}\/ (over $\k$) if a 
classifying space $K(G,1)$ is $1$-formal over $\k$. 
\end{definition}

Since any two classifying spaces for $G$ are homotopy equivalent, 
the discussion from \S\ref{subsec:formal spaces} shows that this 
notion is well-defined.  A similar argument shows that 
the $1$-formality property of a path-connected space $X$ 
depends only on its fundamental group, $G=\pi_1(X)$. 
 
The next, well-known theorem provides an equivalent, purely group-theoretic 
definition of $1$-formality.   Although proofs can be found in the literature 
(see for instance Markl--Papadima \cite{Markl-Papadima}, Carlson--Toledo 
\cite{Carlson-Toledo95}, and Remark \ref{rem:bezpap} below), 
we provide here an alternative proof, based on Theorem \ref{thm:model-holonomy} 
and the discussion from \S\ref{subsec:sullivan-malcev}.

\begin{theorem}
\label{thm:1-formal}
A finitely generated group $G$ is $1$-formal over $\k$ 
if and only if the 
Malcev Lie algebra of $G$ is isomorphic  to the degree 
completion of the holonomy Lie algebra $\fh(G;\k)$.
\end{theorem}

\begin{proof}
Let $\cM(G;\k)=\cM(A_{PL}(K(G,1)),1)\otimes_{\Q}\k$ be the $1$-minimal model of $G$.
The group $G$ is $1$-formal if and only if 
there exists a \dga~ morphism $\cM(G;\k)\to (H^{\hdot}(G;\k),d=0)$ inducing an
isomorphism in first cohomology and a monomorphism in second 
cohomology, i.e., $\cM(G;\k)$ is a $1$-minimal model for $(H^{\hdot}(G;\k),d=0)$. 

Let $\fL(G;\k)$ be the dual Lie algebra of $\cM(G;\k)$. By Theorem \ref{thm:sullivan},
the Malcev Lie algebra of $G$ is isomorphic to $\fL(G;\k)$.
By Theorem \ref{thm:model-holonomy}, the degree completion
of the holonomy Lie algebra of $G$ is isomorphic to $\fL(G;\k)$. 
This completes the proof.
\end{proof}

\begin{remark}
\label{rem:bezpap}
Theorem \ref{thm:1-formal} admits the following generalization: 
if $G$ is a finitely generated group, 
and if $(A,d)$ is a connected $\dga$ with $\dim A^1<\infty$ 
whose $1$-minimal model is isomorphic to $\cM(G;\k)$, 
then the Malcev Lie algebra $\fm(G;\k)$ is isomorphic 
to the completion with respect to the degree filtration of the 
Lie algebra $\fh(A,d):=\Lie((A^1)^*)/\langle \im((d^1)^*+\mu_A^*)\rangle$.  
Proofs of this result are given in \cite{BMPP, PS17}; related results 
can be found in \cite{Bez, Bibby-H14, Polishchuk-Positselski}. 
\end{remark}

An equivalent formulation of Theorem \ref{thm:1-formal} is given by Papadima and Suciu in 
\cite{Papadima-Suciu09}: 
A finitely generated group $G$ is $1$-formal over $\k$ if and only if 
its Malcev Lie algebra $\fm(G;\k)$ is isomorphic 
to the degree completion of a quadratic Lie algebra, 
as filtered Lie algebras.   For instance, if $b_1(G)$ equals $0$ or $1$, 
then $G$ is $1$-formal.

Clearly, finitely generated free groups are $1$-formal; indeed, 
if $F$ is such a group, then $\fm(F;\k)\cong \widehat{\Lie}(H_1(F;\k))$.  
Other well-known examples of $1$-formal groups include fundamental groups 
of compact K\"ahler manifolds \cite{DGMS},  
fundamental groups of complements of complex algebraic 
hypersurfaces \cite{Kohno}, and the pure braid groups 
of surfaces of genus not equal to one \cite{Bez, Hain97}.

\subsection{Massey products}
\label{subsec:massey}
A well-known obstruction to $1$-formality is provided by the higher-order 
Massey products (introduced in \cite{Massey98}). For our purposes, we will discuss 
here only triple Massey products of degree $1$ cohomology classes. 

Let $\gamma_1$, $\gamma_2$ and $\gamma_3$ be cocycles of degrees 
$1$ in the (singular) chain complex $C^{\hdot}(G;\k)$, with cohomology classes $u_i=[\gamma_i]$ 
satisfying $u_1\cup u_2=0$ and $u_2 \cup u_3=0$. That is, we assume there are
$1$-cochains $\gamma_{12}$ and $\gamma_{23}$ such that 
$d \gamma_{12}=\gamma_1\cup \gamma_2$ and 
$d \gamma_{23}=\gamma_2\cup \gamma_3$. It is readily seen that the $2$-cochain 
$\omega=\gamma_{12}\cup \gamma_3+\gamma_1\cup \gamma_{23}$ is, in fact,  
a cocycle.  The set of all cohomology classes $[\omega]$ obtained in this way 
is the \emph{Massey triple product}\/ $\langle u_1, u_2, u_3\rangle$. 
Due to the ambiguity in the choice of $\gamma_{12}$ and $\gamma_{23}$, 
the Massey triple product $\langle u_1, u_2, u_3\rangle$ is a representative 
of the coset 
\begin{equation}
\label{eq:massey prod}
H^2(G;\k) \slash(u_1 \cup H^1(G;\k)+H^1(G;\k)\cup u_3).
\end{equation}

\begin{remark}
\label{rem:massey-magnus}
In  \cite[Thm.~2]{Porter80}, Porter gave a topological method for computing 
cup products and higher-order Massey products in $H^2(G;\k)$. 
Building on work of Dwyer \cite{Dwyer}, 
 Fenn and Sjerve  \cite{Fenn-S}  gave explicit formulas 
for Massey products in a com\-mutator-relator group. 
\end{remark}

If a group $G$ is $1$-formal, then all triple Massey products vanish 
in the quotient $\k$-vector space from \eqref{eq:massey prod}.
However, if $G$ is only graded-formal, these Massey products need 
not vanish. As we shall see in Example \ref{ex:1rel 1f}, even a one-relator 
group $G$ may be graded-formal, yet not $1$-formal.  

\begin{example}
\label{ex:1rel 1f}
Let $G=\langle x_1,\dots,x_5\mid [x_1,x_2][x_3,[x_4,x_5]]=1\rangle$. 
Using \cite[Thm.~8.3]{SW-holo}, we see that the group $G$ is 
graded-formal.  On the other hand, by  \cite{Fenn-S}, the triple 
Massey product $\langle u_3,u_4,u_5\rangle$ 
is non-zero (modulo indeterminacy).  
Thus, $G$ is not $1$-formal, and so $G$ is not filtered-formal. 
\end{example}

\subsection{Filtered-formality, graded-formality and $1$-formality}
\label{subsec:ff-gf-1f}
The next result pulls together the various formality notions 
for groups, and establishes the basic relationship among them. 
 
\begin{prop}
\label{prop:qwformal}
A finitely generated group $G$ is $1$-formal if and only if $G$ is 
graded-formal and filtered-formal.
\end{prop}

\begin{proof}
First suppose $G$ is $1$-formal.  Then, by Theorem \ref{thm:1-formal},
 $\fm(G;\k) \cong \widehat{\h}(G;\k)$, 
and thus, $\gr(G;\k)\cong \h(G;\k)$, by \eqref{eq:quillen}.  Hence, 
$G$ is graded-formal, by Lemma \ref{lem:graded-1-formal}\eqref{g1f1}.  
It follows that $\fm(G;\k) \cong \widehat{\gr}(G;\k)$, and hence  
$G$ is filtered-formal, by Proposition \ref{prop:ffcrit}.

Now suppose $G$ filtered-formal. 
Then, by Proposition \ref{prop:ffcrit}, we have that $\fm(G;\k) \cong \widehat{\gr}(G;\k)$. 
Thus, if $G$ is also graded-formal, $\fm(G;\k) \cong \widehat{\h}(G;\k)$.  
Hence, $G$ is $1$-formal.
\end{proof}

Using Propositions \ref{prop:field-filtered} and \ref{prop:qwformal},
and Corollary \ref{lem:gf-ext}, we obtain the following corollary.  

\begin{corollary}
\label{cor:1fdescent}
A finitely generated group $G$  
is $1$-formal over $\Q$ if and only if $G$ is $1$-formal over $\k$.
\end{corollary}

In other words, the $1$-formality property of a finitely generated group is 
independent of the choice of coefficient field of characteristic $0$. 

A filtered-formal group need not be $1$-formal.   
Examples include some of the free 
nilpotent groups from Example \ref{ex:nstep} and the unipotent 
groups from Example \ref{ex:upper}.  In fact, the triple Massey products 
in the cohomology of a filtered-formal group need not vanish 
(modulo indeterminacy). 

\begin{example}
\label{ex:filt massey}
Let $G=F_2/\Gamma_3 F_2=\langle x_1, x_2 \mid 
[x_1,[x_1, x_2]]=[x_2,[x_1, x_2]]=1\rangle$ be the 
Heisenberg group. Then $G$ is filtered-formal, 
yet has non-trivial triple Massey products, $\langle u_1,u_1,u_2\rangle$ 
and $\langle u_2,u_1,u_2\rangle$, in $H^2(G;\k)$.  
Hence, $G$ is not graded-formal.
\end{example}

As shown by in Hain in \cite{Hain97, Hain15} the Torelli groups in genus 
$4$ or higher are $1$-formal, but the Torelli group in 
genus $3$ is filtered-formal, yet not graded-formal. 
The next two examples show that there are groups which 
are graded-formal but not filtered-formal.

\begin{example}
\label{ex:qtr groups}
In \cite{Bartholdi-E-E-R}, Bartholdi et al.~consider two infinite families 
of groups corresponding to the {Y}ang--{B}axter equations. 
The first are the quasitriangular groups 
$\QTr_n$, which have presentations with generators $x_{ij}$ $(1\leq i\neq j \leq n)$,  
and relations $x_{ij}x_{ik}x_{jk}=x_{jk}x_{ik}x_{ij}$ and
$x_{ij}x_{kl}=x_{kl}x_{ij}$ for distinct $i,j,k,l$.  The second are 
the triangular groups $\Tr_n$, each of which is the quotient of $\QTr_n$ by 
the relations of the form $x_{ij}=x_{ji}$ for $i\neq j$.  As shown by Lee in \cite{Lee}, 
the groups $\QTr_n$ and $\Tr_n$ are all graded-formal.  On the other hand, 
as indicated in \cite{Bartholdi-E-E-R}, these groups are non-$1$-formal 
(and hence, not filtered-formal) for all $n\ge 4$.  
A detailed proof of this fact is given in \cite[Cor.~6.7]{SW-pvb}.
\end{example}

\begin{example}
\label{ex:quadrelation}
Let $G$ be the group with generators 
$x_1, \dots, x_4$  and relators $[x_2,x_3]$, 
$[x_1,x_4]$, and $[x_1,x_3][x_2,x_4]$.  
As noted in \cite{SW-holo}, $G$ is graded-formal.
On the other hand, using the Tangent Cone theorem of 
\cite{Dimca-Papadima-Suciu}, one can show that the group 
$G$ is not $1$-formal.  Therefore, $G$ is not filtered-formal.
\end{example}

\subsection{Examples from link theory}
\label{subsec:link}

Let $L=(L_1,\dots , L_n)$ be an $n$-component link in $S^3$.  
The link group, $G=\pi_1(X)$, is the fundamental group of 
the complement, $X=S^3\setminus \bigcup_{i=1}^n L_i$.
In general, a link group is not $1$-formal.  
This phenomenon was first detected by W.S.~Massey 
by means of his higher-order products \cite{Massey98}, but the 
absence of graded-formality and especially filtered-formality can be even 
harder to detect.  

\begin{example}
\label{ex:link}
Let $L$ be either the $2$-component Whitehead link or the 
$3$-component Borromean links. It follows from work of 
Hain \cite{Hain85} that the corresponding link 
groups $G$ are not graded formal, and thus not $1$-formal. 
The non-$1$-formality of these groups can also  be detected 
by suitable higher-order Massey products. We refer 
to \cite{Suciu-3manifold} for a complete answer to the 
formality question for complements of $2$-component links.
\end{example}

Next, we give an example of a link group which is graded-formal, 
yet not filtered-formal. 

\begin{example}
\label{ex:semiproduct}
Let $L$ be the link of $5$ great circles in $S^3$ corresponding to the 
arrangement of transverse planes through the origin of $\R^4$ 
denoted as $\mathcal{A}(31425)$ in Matei--Suciu \cite{MateiSuciu00}.  
As noted in \cite{SW-holo}, work of Berceanu and 
Papadima \cite{Berceanu-Papadima94} implies 
that the link group $G$ is graded-formal.  On the other hand, as noted in 
\cite[Ex.~8.2]{Dimca-Papadima-Suciu}, the Tangent Cone theorem 
does not hold for this group, and thus $G$ is not $1$-formal.
Consequently, $G$ is not filtered-formal.
\end{example}

\subsection{Propagation of filtered-formality}
\label{subsec:propff}
The next theorem shows that filtered-formality is inherited 
upon taking nilpotent quotients.
In \S\ref{sect:nilp}, we will focus on exploring the formality properties of
torsion-free nilpotent groups.

\begin{theorem}
\label{thm:formalityQuo}
Let $G$ be a finitely generated group, and suppose $G$ 
is filtered-formal.  Then all the nilpotent quotients 
$G/\Gamma_i(G)$ are filtered-formal.
\end{theorem}

\begin{proof}
Set $\g=\gr(G;\k)$ and $\fm=\fm(G;\k)$, and write 
$\g=\bigoplus_{k\geq 1} \fg_k$.   For each $i\ge 1$, 
the canonical projection $\phi_i\colon G\surj G/\Gamma_iG$ 
induces an epimorphism of complete, filtered Lie algebras, 
$\fm(\phi_i)\colon \fm \surj \fm(G/\Gamma_iG;\k)$. 
In each degree $k\ge i$, we have that 
$\widehat{\Gamma}_k \fm(G/\Gamma_iG;\k)=0$, 
and so $\fm(\phi_i)(\widehat{\Gamma}_k\fm)=0$.
Therefore, there exists an induced epimorphism 
$\Phi_{k,i}\colon \fm/\widehat{\Gamma}_k\fm 
\surj \fm(G/\Gamma_iG;\k)$.

Passing to the associated graded, we obtain an epimorphism  
$\gr(\fm/\widehat{\Gamma}_k\fm) \surj \gr(\fm(G/\Gamma_iG;\k))$,
which is readily seen to be an isomorphism for $k=i$. 
Using now Lemma \ref{lem:grfilt}, we conclude that  the map $\Phi_{i,i}$ is an 
isomorphism of (complete, separated) filtered Lie algebras.

On the other hand, our filtered-formality assumption on $G$ allows 
us to identify $\fm\cong \widehat{\fg}=\prod_{k\geq 1} \fg_k$.  
Let $\iota\colon \g\to \widehat{\g}$ be the canonical morphism. 
By construction,  we have isomorphisms 
$\iota_k \colon \g/\Gamma_k\g \to \widehat{\fg}/\widehat{\Gamma}_k\widehat{\fg}$
for all $k\ge 1$.  Thus, 
$\fm/\widehat{\Gamma}_k\fm\cong \widehat{\fg}/\widehat{\Gamma}_k\widehat{\fg}\cong 
\fg/\Gamma_k\fg$, for all $k\ge 1$.
Using these identifications for $k=i$, together with the isomorphism 
$\Phi_{i,i}$ from above, we obtain isomorphisms 
$\fm(G/\Gamma_iG;\k)\cong \fg/\Gamma_i\fg\cong \gr(G/\Gamma_iG;\k)$, 
thereby showing that the nilpotent quotient $G/\Gamma_iG$ is filtered-formal.
\end{proof}

\begin{prop}
\label{prop:formal stallings}
Suppose $\phi\colon G_1\to G_2$ is a homomorphism 
between two finitely generated groups, inducing an isomorphism 
$H_1(G_1;\k)\to H_1(G_2;\k)$ and an epimorphism 
$H_2(G_1;\k)\to H_2(G_2;\k)$. 
\begin{enumerate}
\item \label{fs0}
 If $G_2$ is $1$-formal, then $G_1$ is also $1$-formal.
\item \label{fs1}
If $G_2$ is filtered-formal, then $G_1$ is also filtered-formal. 
\item \label{fs2}
If $G_2$ is graded-formal, then $G_1$ is also graded-formal. 
\end{enumerate}
\end{prop}

\begin{proof}
A celebrated theorem of Stallings \cite{Stallings} 
(see also \cite{Dwyer, Freedman-Hain-Teichner}) 
insures that the homomorphism $\phi$ induces isomorphisms 
$\phi_k\colon (G_1/\Gamma_k G_1)\otimes \k\to 
(G_2/\Gamma_k G_2)\otimes \k$, 
for all $k\ge 1$. Hence, $\phi$ induces an isomorphism 
$\fm(\phi)\colon \fm(G_1;\k)\to \fm(G_2;\k)$ between the 
respective Malcev completions, thereby proving claim \eqref{fs0}. 
Using now the isomorphism \eqref{eq:quillen}, the other 
two claims follow at once.
\end{proof}

\subsection{Split injections}
\label{subsec:splif formal}

We are now ready to state and prove the main result of this section, 
which completes the proof of Theorem \ref{thm:intro formality} 
from the Introduction.

\begin{theorem} 
\label{thm:formality}
Let $G$ be a finitely generated group, and let 
$\iota\colon K\to G$ be a split injection.
\begin{enumerate}
\item \label{f1}
If $G$ is filtered-formal, then $K$ is also filtered-formal. 
\item \label{f2}
If $G$ is $1$-formal, then $K$ is also $1$-formal. 
\end{enumerate}
\end{theorem}

\begin{proof}  
By hypothesis, we have an epimorphism $\sigma\colon G\surj K$ 
such that $\sigma\circ\iota=\id$.  It follows that the induced maps 
$\fm(\iota)$ and $\widehat{\gr}(\iota)$ are also split injections. 

Let $\pi\colon F\surj G$ be an epimorphism from a finitely generated 
free group $F$ to $G$. Let  $\pi_1:=\sigma\circ \pi\colon F\surj K$;  
there is then a map $\iota_1\colon F\rightarrow F$ 
which is a lift of $\iota$, that is, $\iota\circ \pi_1=\pi\circ \iota_1$.
Consider the following diagram (we suppress the coefficient field 
$\k$ of characteristic zero from the notation).   
\begin{equation}
\label{eq:cube}
\begin{tikzpicture}[ 
baseline=(current  bounding  box.center),scale=0.9,every node/.style={scale=0.9}]
\matrix (m) [matrix of math nodes, row sep={2.8em,between origins}, 
column sep={4.3em,between origins}]{
 & J_1\vphantom{\hat{\Lie}} && \widehat{\Lie}(F)&& \widehat{\gr}(K)\vphantom{\widehat{\Lie}} \\
I_1\vphantom{\hat{\Lie}}  && \widehat{\Lie}(F) && \fm(K)\vphantom{\widehat{\Lie}} & \\ 
 & J && \widehat{\Lie}(F)&& \widehat{\gr}(G)\vphantom{\widehat{\Lie}} \\
I && \widehat{\Lie}(F) && \fm(G)\vphantom{\widehat{\Lie}} & \\ };
\path[->,>=angle 90]
(m-2-3) edge node[left,pos=0.6]{ $\id~$ } (m-1-4)
(m-4-3) edge node[left,pos=0.6]{$\id~$} (m-3-4)
(m-4-3) edge [-,line width=8pt,draw=white](m-2-3)
(m-4-3) edge (m-2-3)
 (m-2-1) edge(m-4-1)
 (m-1-2)  edge (m-3-2) 
 (m-1-4) edge  (m-3-4)
; 
\path[right hook->,>=angle 90]
(m-1-2) edge (m-1-4)
(m-3-2) edge (m-3-4)
(m-4-1) edge (m-3-2)
(m-4-1) edge (m-4-3)
;
\path[dotted,right hook->,>=angle 90]
(m-2-1) edge (m-1-2);
\path[->, >=angle 90]
(m-4-5) edge node[left, pos=0.65]{$\Phi~$} node[right,pos=0.4]{$\cong$}(m-3-6)
(m-4-5) edge (m-2-5)
;
\path[dotted,->>,>=angle 90]
(m-2-5) edge node[left]{$\Phi_1~$} (m-1-6)
;
\path[->, >=angle 90]
(m-4-3) edge [-,line width=10pt,draw=white](m-2-3)
(m-2-3) edge  node[pos=0.8,left]{$\fm(\iota_1)$}(m-4-3)
; 
\path[->>, >=angle 90]
(m-1-4) edge (m-1-6)
(m-3-4) edge (m-3-6)
(m-4-3) edge (m-4-5)
(m-2-3) edge [-,line width=8pt,draw=white](m-2-5)
(m-2-3) edge (m-2-5)
;
\path[right hook->,>=angle 90]
(m-2-1) edge [-,line width=8pt,draw=white] (m-2-3)
(m-2-1) edge (m-2-3)
(m-4-5) edge [-,line width=10pt,draw=white](m-2-5)
(m-2-5) edge node[pos=0.8,left]{$\fm(\iota)$} (m-4-5)
(m-1-6) edge node[right]{$\widehat{\gr}(\iota)$} (m-3-6)
;
\end{tikzpicture}
\end{equation}

We have $\fm(\iota_1)=\widehat{\gr}(\iota_1)$.
By assumption, $G$ is filtered-formal; hence, there exists a 
filtered Lie algebra isomorphism 
$\Phi\colon \fm(G) \to \widehat{\gr}(G)$ 
as in diagram \eqref{eq:cube},  
which induces the identity on associated graded algebras.  It 
follows that $\Phi$ is induced from the identity map of 
$\widehat{\Lie}(F)$ upon projecting onto source and target, 
i.e., the bottom right square in  the diagram commutes.

First, we show that the identity map 
$\id\colon \widehat{\Lie}(F)\rightarrow \widehat{\Lie}(F)$ 
in the above diagram
induces an inclusion map  $I_1\rightarrow J_1$.  Suppose there is an element 
$c\in \widehat{\Lie}(F)$ such that $c\in I_1$ and $c\notin J_1$, that is, $[c]= 0$
in $\fm(K)$ and $[c]\neq 0$ in $\widehat{\gr}(G)$.
Since $\widehat{\gr}(\iota)$ is injective, we have that $\widehat{\gr}(\iota)([c])\neq 0$, that is, 
$\widehat{\gr}(\iota_1)(c)\notin I$. We also have $\fm(\iota)([c])=0\in \fm(G)$, that is, 
$\fm(\iota_1)(c)\in J$. This contradicts the fact that the inclusion $I\inj J$ is induced by 
the identity map. Thus,  $I_1\subset J_1$.

In view of the above, we may define a Lie algebra morphism 
$\Phi_1\colon  \fm(K)\rightarrow\widehat{\gr}(K)$ as the quotient 
of the identity on $\widehat{\Lie}(F)$. 
By construction, $\Phi_1$ is an epimorphism.  
We also have $\widehat{\gr}(\iota)\circ\Phi_1=\Phi\circ\fm(\iota)$. 
Since the maps $\fm(\iota)$, $\widehat{\gr}(\iota)$ and $\Phi$ are all injective, 
the map $\Phi_1$ is also injective. Therefore, $\Phi_1$ is an 
isomorphism, and so the group $K$ is filtered-formal.
 
Finally, part \eqref{f2} follows at once from part \eqref{f1} and 
Theorem \ref{thm:graded1formality}.
\end{proof}

This completes the proof of Theorem \ref{thm:intro formality} from the Introduction.
As we shall see in Example \ref{ex:qtr groups}, this theorem is useful for deciding 
whether certain infinite families of groups are $1$-formal.  

We now proceed with the proof of Theorem \ref{thm:product formality}.  First, 
we need a lemma. 

\begin{lemma}[\cite{Dimca-Papadima-Suciu}]
\label{lem:ml isos}
Let  $G_1$ and $G_2$ be two finitely generated groups. Then 
$\fm(G_1\times G_2;\k)\cong \fm(G_1;\k)\times \fm(G_2;\k)$ and 
$\fm(G_1*G_2;\k)\cong \fm(G_1;\k)\cstar \fm(G_2;\k)$. 
\end{lemma}

\begin{prop} 
\label{prop:productweakly}
For finitely generated groups $G_1$ and $G_2$, the 
following conditions are equivalent.
\begin{enumerate}
\item \label{ff1}  $G_1$ and $G_2$ are filtered-formal.
\item \label{ff2}  $G_1* G_2$  is filtered-formal.
\item \label{ff3}  $G_1\times G_2$ is filtered-formal.
\end{enumerate}
\end{prop} 

\begin{proof}
Since there exist split injections from $G_1$ and $G_2$ to 
the product $G_1\times G_2$ and coproduct $G_1*G_2$, 
we may apply Theorem \ref{thm:formality} to conclude that implications 
\eqref{ff2}$\Rightarrow$\eqref{ff1} and \eqref{ff3}$\Rightarrow$\eqref{ff1} 
hold. Implications 
\eqref{ff1}$\Rightarrow$\eqref{ff2} and \eqref{ff1}$\Rightarrow$\eqref{ff3} 
follow from Lemmas \ref{lem:filt prod}, \ref{lem:filt coprod}, and \ref{lem:ml isos}.
\end{proof}

As we shall see in Example \ref{ex:semiproduct}, 
the implication \eqref{ff1}$\Rightarrow$\eqref{ff3} 
from Proposition \ref{prop:productweakly} 
cannot be strengthened from direct products to 
arbitrary semi-direct products. More precisely, 
there exist split extensions of the form $G=F_{n}\rtimes_{\alpha} \Z$, 
for certain automorphisms $\alpha\in \Aut(F_n)$, such that the group 
$G$ is not filtered-formal, although of course both $F_n$ and 
$\Z$ are $1$-formal. 

The next corollary follows at once from Propositions \ref{prop:productquadratic} 
and \ref{prop:productweakly}.
\begin{corollary}
\label{cor:nongrfilt}
Suppose $G_1$ and $G_2$ are finitely generated groups such 
that $G_1$ is not graded-formal and $G_2$ is not filtered-formal. 
Then the product $G_1\times G_2$ and the free product 
$G_1* G_2$ are neither graded-formal, nor filtered-formal.  
\end{corollary}

As mentioned in \S\ref{subsec:formal}, concrete examples of 
groups which do not possess either formality property 
can be obtained by taking direct products of groups 
which enjoy one property but not the other. 

\section{Derived series and  Lie algebras}
\label{section:chen lie}

We now study some of the relationships between the derived 
series of a group and the derived series of the corresponding 
Lie algebras.

\subsection{Derived series}
\label{subsec:derived}
Consider the derived series of a group $G$, starting at 
$G^{(0)}=G$,  $G^{(1)}=G'$,  and $G^{(2)}=G''$, 
and defined inductively by $G^{(i+1)}=[G^{(i)},G^{(i)}]$.  Note 
that any homomorphism $\phi\colon G\to H$ takes $G^{(i)}$ to $H^{(i)}$. 
The quotient groups, $G/G^{(i)}$, are solvable; in particular, $G/G'=G_{\ab}$, 
while $G/G''$ is the maximal metabelian quotient of $G$. 

Assume that $G$ is a finitely generated group, and fix a coefficient 
field $\k$ of characteristic $0$.
\begin{prop}[\cite{SW-holo}]
\label{prop:derivedquo}
The holonomy Lie algebras $\fh(G/G^{(i)};\k)$ of the derived quotients of $G$ are isomorphic
to $\fh(G;\k)/\fh(G;\k)^{\prime}$ for $i=1$, and are isomorphic to
$\fh(G;\k)$ for $i\ge 2$.
\end{prop}

The next theorem is the Lie algebra version of Theorem 3.5 
from \cite{Papadima-Suciu04}.
\begin{theorem}[\cite{Papadima-Suciu04}]
\label{thm:chenLieAlgebra}
For each $i\ge 2$, there is an isomorphism of complete, separated filtered Lie algebras,
$
\fm(G/G^{(i)};\k)\cong \fm(G;\k)/\overline{\fm(G;\k)^{(i)}},
$
where $\overline{\fm(G;\k)^{(i)}}$ denotes the closure of $\fm(G;\k)^{(i)}$ 
with respect to the filtration topology on $\fm(G;\k)$.
\end{theorem}

\subsection{Chen Lie algebras}
\label{subsec:chen lie}
As before, let $G$ be a finitely generated group. 
For each $i\ge 2$, the \textit{$i$-th Chen Lie algebra}\/ of $G$ is defined to 
be the associated graded Lie algebra of the corresponding solvable quotient, 
$\gr(G/G^{(i)};\k)$.  
Clearly, this construction is functorial. 

The quotient map, $q_i\colon G\surj G/G^{(i)}$, induces a surjective 
morphism $\gr(q_i)$ between associated graded Lie algebras
  $\gr_k(G;\k)$ and  $\gr_k(G/G^{(i)};\k)$.  Plainly,  
this morphism is the canonical identification in degree $1$.  
In fact, the map  $\gr(q_i)$ is an isomorphism for each $k\leq 2^i-1$, see \cite{SW-holo}.

We now specialize to the case when $i=2$, 
originally studied by K.-T. Chen in \cite{Chen51}. 
The {\em Chen ranks}\/ of $G$ are defined as 
$\theta_k(G):=\dim_{\k} (\gr_k(G/G^{\prime\prime};\k))$. 
For free groups, Chen showed that
$\theta_k(F_n)=(k-1)\binom{n+k-2}{k}$ 
for all $k\ge 2$.
By analogy, let us define the {\em holonomy Chen ranks}\/ of $G$ as 
$\bar\theta_k(G)=\dim_{\k} (\h/\h'')_k$, where $\h=\h(G;\k)$.  
It is readily seen that $\bar\theta_k(G)\ge \theta_k(G)$, 
with equality for $k\le 2$. 

\subsection{Chen Lie algebras and formality}
\label{subsec:chen formal}
We are now ready to state and prove the main result 
of this section, which (together with the first corollary following it) 
proves Theorem \ref{thm:chenlie intro} from the Introduction.

\begin{theorem} 
\label{thm:chenlieiso}
Let $G$ be a finitely generated group. 
For each $i\ge 2$, the quotient map $q_i\colon G\surj G/G^{(i)}$ induces a 
natural epimorphism of graded $\k$-Lie algebras, 
\[
\xymatrix{\Psi_G^{(i)}\colon \gr(G;\k)/\gr(G;\k)^{(i)} \ar@{->>}[r]
& \gr(G/G^{(i)};\k)}.
\] 
Moreover, if $G$ is a filtered-formal group, then $\Psi_G^{(i)}$ is an isomorphism
and the solvable quotient $G/G^{(i)}$ is filtered-formal.
\end{theorem}

\begin{proof}
The map $q_i\colon G\surj G/G^{(i)}$ induces a 
natural epimorphism $\gr(q_i)$ of graded $\k$-Lie algebras 
$\gr(G;\k)$ and $\gr(G/G^{(i)};\k)$.
By Proposition \ref{prop:lazard}, this epimorphism factors   
through an isomorphism, 
$\gr(G;\k)/\widetilde{\gr} (G^{(i)};\k) \isom \gr(G/G^{(i)};\k)$, 
where $\widetilde{\gr}$ denotes the graded Lie algebra 
associated with the filtration 
$\widetilde{\Gamma}_k G^{(i)}= \Gamma_k G \cap G^{(i)}$. 

On the other hand, 
as shown by Labute in \cite[Prop.~10]{Labute08}, the 
Lie ideal $\gr(G;\k)^{(i)}$ is contained in $\widetilde{\gr} (G^{(i)};\k)$.  
Therefore, the map $\gr(q_i)$ factors through the claimed epimorphism 
$\Psi_G^{(i)}$, as indicated in the following commuting diagram,
\begin{equation}
\label{eq:psigi}
\xymatrixcolsep{32pt}
\xymatrixrowsep{20pt}
\xymatrix{
\gr(G;\k) \ar@{->>}[r] \ar@{->>}_{\gr(q_i)}[dr]
& \gr(G;\k)/\gr(G;\k)^{(i)}   \ar@{->>}[r] \ar@{->>}[d]^{\Psi_G^{(i)}}
& \gr(G;\k)/\widetilde{\gr}(G^{(i)};\k)  \ar^{\cong}[dl] \\
& \gr(G/G^{(i)};\k) \, . &  \\
}
\end{equation}

Suppose now that $G$ is filtered-formal, and set 
$\fm=\fm(G;\k)$ and $\fg=\gr(G;\k)$. We may identify $\widehat{\fg} \cong \fm$.
Let $\fg\inj \widehat{\fg}$ be the inclusion into 
the completion. Passing to solvable quotients, we obtain a morphism 
of filtered Lie algebras, $\varphi^{(i)}\colon \fg/\fg^{(i)}\to \fm/\overline{\fm^{(i)}}$.
Passing to the associated graded Lie algebras, we obtain the following 
commuting diagram:
\begin{equation}
\label{eq:mapgr}
\begin{gathered}
\xymatrixcolsep{20pt}
\xymatrixrowsep{20pt}
\xymatrix{\fg/\fg^{(i)} \ar[d]^(.42){\gr(\varphi^{(i)})}
\ar[r]^(.4){\Psi_G^{(i)}}&  \gr(G/G^{(i)};\k) \ar[d]^(.45){\cong} \\
\gr(\fm/\overline{\fm^{(i)}})  \ar[r]^(.42){\cong}& \gr(\fm(G/G^{(i)};\k) ).
}
\end{gathered}
\end{equation}

All the graded Lie algebras in this diagram are generated in degree $1$, and 
all the morphisms induce the identity in this degree. Therefore, the diagram 
is commutative.  Moreover, the right vertical arrow from \eqref{eq:mapgr} 
is an isomorphism by Quillen's isomorphism \eqref{eq:quillen}, while the lower 
horizontal arrow is an isomorphism by Theorem \ref{thm:chenLieAlgebra}.

Recall that, by assumption, $\fm=\widehat{\fg}$; therefore, the inclusion  
of filtered Lie algebras $\fg\inj \widehat{\fg}$ induces a morphism between 
the following two exact sequences, 
\begin{equation}
\label{eq:gradediso}
\xymatrixcolsep{18pt}
\xymatrixrowsep{16pt}
\begin{gathered}
\xymatrix{ 0 \ar[r] & \widetilde{\gr}(\overline{\fm^{(i)}}) \ar[r]  
&\gr(\fm)\ar[r] & \gr(\fm)/\widetilde{\gr}(\overline{\fm^{(i)}}) \ar[r] &0 
\phantom{\,.}\\
0\ar[r] &\fg^{(i)}\ar[r]\ar@{.>}[u]&\fg\ar[r]\ar[u]^{\cong}&\fg/\fg^{(i)}
\ar[r]\ar@{.>}[u] &0 \,.
}
\end{gathered}
\end{equation}
Here $\widetilde{\gr}$ means taking the associated graded 
Lie algebra corresponding to the induced filtration. 
Using formulas \eqref{eq:closureBracket} and \eqref{eq:subLiecompletion}, 
it can be shown that 
$\widetilde{\gr}(\overline{\fm^{(i)}})= \fg^{(i)}$. Therefore, the morphism
$\fg/\fg^{(i)}\to \gr(\fm)/\widetilde{\gr}(\overline{\fm^{(i)}})$ is an isomorphism.
We also know that $\gr(\fm/\overline{\fm^{(i)}})=
\gr(\fm)/\widetilde{\gr}(\overline{\fm^{(i)}})$.
Hence, the map $\gr(\varphi^{(i)})$ is an isomorphism, and so, 
by \eqref{eq:mapgr}, the map $\Psi_G^{(i)}$ is an isomorphism, too. 

By Lemma \ref{lem:grfilt}, the map $\varphi^{(i)}$ induces an 
isomorphism of complete, filtered Lie algebras between the 
degree completion of $\fg/\fg^{(i)}$ and $\fm/\overline{\fm^{(i)}}$.
As shown above, $\Psi_G^{(i)}$ is an isomorphism; 
hence, its degree completion is also an isomorphism. 
Composing with the isomorphism from Theorem \ref{thm:chenLieAlgebra}, 
we obtain  an isomorphism between the degree 
completion $\widehat{\gr}(G/G^{(i)};\k)$ and the Malcev Lie 
algebra $\fm(G/G^{(i)};\k)$.  This shows that the solvable 
quotient $G/G^{(i)}$ is filtered-formal.
\end{proof}

\begin{remark}
\label{rem:labute fab}
As shown in \cite[\S 3]{Labute08}, the analogue of Theorem \ref{thm:chenlieiso} 
does not hold if the ground field $\k$ has characteristic $p>0$. 
More precisely, there are pro-$p$ groups $G$ for which the 
morphisms $\Psi_G^{(i)}$ ($i\ge 2$) are not isomorphisms. 
\end{remark}
 
Returning now to the setup from Lemma \ref{lem:holoepi}, let us 
compose the canonical projection $\gr(q_i)\colon$ 
$\gr(G;\k)\surj \gr(G/G^{(i)};\k)$ with the epimorphism 
$\Phi_G\colon \fh(G;\k)\surj \gr(G;\k)$.  We obtain in 
this fashion an epimorphism $\fh(G;\k)\surj \gr(G/G^{(i)};\k)$, 
which fits into the following commuting diagram (we will 
suppress the coefficient field $\k$):
\begin{equation}
\label{eq:cd}
\begin{tikzpicture}[baseline=(current  bounding  box.center)]
\matrix (m) [matrix of math nodes, 
row sep={2em,between origins}, 
column sep={5em,between origins}]{
 & \fh(G) && \gr(G)   \\
\fh(G/G^{(i)}) && \gr(G/G^{(i)}) &  \\ 
 & \fh(G)/\fh(G)^{(i)}&   &\gr(G)/\gr(G)^{(i)}\,. \\ };
\path[->>]
(m-1-2) edge node[above]{$\Phi_G$} (m-1-4) 
edge (m-2-3) edge  (m-2-1)
(m-2-1) edge (m-3-2)
(m-3-2) edge (m-2-3) edge (m-3-4)
(m-1-4) edge (m-2-3) edge (m-3-4)
(m-3-4) edge (m-2-3)
;
\end{tikzpicture} 
\end{equation}

Putting things together, we obtain the following corollary, which 
recasts Theorem 4.2 from \cite{Papadima-Suciu04} in a setting which is both 
functorial, and holds in wider generality. This corollary provides a way to 
detect non-$1$-formality of groups. 

\begin{corollary}
\label{cor:chenlie bis}
For each $i\geq 2$, there is a natural epimorphism 
of graded $\k$-Lie algebras, 
$\Phi_G^{(i)}\colon$ $ \fh(G;\k)/\fh(G;\k)^{(i)} \surj 
\gr(G/G^{(i)};\k)$. 
Moreover, if $G$ is $1$-formal, then $\Phi_G^{(i)}$ is an isomorphism. 
\end{corollary}

\begin{corollary}
\label{cor:chen formal}
Suppose the group $G$ is $1$-formal. Then, for each for $i\geq 2$, 
the solvable quotient $G/G^{(i)}$ is graded-formal if and only if 
$\fh(G;\k)^{(i)}$ vanishes. 
\end{corollary}

\begin{proof}
By Proposition \ref{prop:derivedquo}, the projection 
$q_i\colon G\to G/G^{(i)}$ induces an isomorphism 
$\fh(q_i)\colon \fh(G;\k)\rightarrow \fh(G/G^{(i)};\k)$. 
Since $G$ is $1$-formal,  
Corollary \ref{cor:chenlie bis} guarantees that the map  
$\Phi_G^{(i)} \colon \fh(G;\k)/\fh(G;\k)^{(i)} \to \gr(G/G^{(i)};\k)$ 
is an isomorphism. 
The claim follows from the left square of diagram \eqref{eq:cd}.
\end{proof}

\section{Torsion-free nilpotent groups}
\label{sect:nilp}

In this section we study the graded-formality and filtered-formality properties of a 
well-known class of groups: that of finitely generated, torsion-free 
nilpotent groups. In the process, we prove Theorem \ref{thm:nilp intro} 
from the Introduction.

\subsection{Nilpotent groups and Lie algebras}
\label{subsec:nilpotent groups}

We start by reviewing the construction of the Malcev Lie algebra of 
a finitely generated, torsion-free nilpotent group $G$  
(see \cite{Cenkl-Porter00, Lambe-Priddy82, Malcev51} for more details).
There is a refinement of the upper central series of such a group,  
$G=G_1 >  \cdots > G_n > G_{n+1}=1$, 
with each subgroup $G_i$ a normal subgroup of $G_{i+1}$, and 
each quotient $G_i/G_{i+1}$ an infinite cyclic group. 
(The integer $n$ is an invariant of the group, called the length of $G$.) 
Using this fact, we can choose a \textit{Malcev basis}\/ $\{u_1,\dots, u_n\}$ 
for $G$, which satisfies $G_i=\langle G_{i+1}, u_i\rangle$. Consequently, 
each element $u\in G$ can be  written  uniquely as 
$u_1^{a_1}u_2^{a_2}\cdots u_n^{a_n}$. 

Using this basis, the group $G$, as a set, can be identified with $\Z^n$ 
via the map sending $u_1^{a_1}\cdots u_n^{a_n}$ to $a=(a_1,\dots, a_n)$. 
The multiplication in $G$  takes the form
$u_1^{a_1}\cdots u_n^{a_n} \cdot u_1^{b_1}\cdots u_n^{b_n}=
u_1^{\rho_1(a,b)}\cdots u_n^{\rho_n(a,b)}$, 
where each map $\rho_i\colon \Z^n\times \Z^n \to \Z$ 
is a rational polynomial function.  This procedure identifies 
the group $G$ with  the group $(\Z^n,\rho)$, with multiplication the map 
$\rho=(\rho_1,\dots,\rho_n)\colon \Z^n\times \Z^n \to \Z^n$. 
Thus, we can define a simply-connected nilpotent Lie group 
$G\otimes \k=(\k^n,\rho)$ by extending the domain of $\rho$, 
which is called the {\em Malcev completion}\/ of $G$. 

The discrete group $G$ is a subgroup of the real Lie group $G\otimes \R$. 
The quotient space, $M=(G\otimes \R)/G$, is a compact manifold, called 
a {\em nilmanifold}.  As shown  in \cite{Malcev51}, the Lie 
algebra of $G\otimes \R$ 
is isomorphic to $\fm(G;\R)$. It is readily apparent that 
the nilmanifold $M$ is an Eilenberg--MacLane space of type $K(G,1)$.  
As shown by K.~Nomizu, the cohomology ring $H^{\hdot}(M,\R)$ is isomorphic 
to the cohomology ring of the Lie algebra $\fm(G;\R)$. 

The polynomial functions $\rho_i$ have the form 
$\rho_i(a,b)=a_i+b_i+\tau_i(a_1,\dots,a_{i-1},b_1,\dots,b_{i-1})$. 
Denote by $\sigma=(\sigma_1,\dots,\sigma_n)$ the quadratic part of $\rho$.
Then $\k^n$ can be given a Lie algebra structure, with bracket 
$[a,b]=\sigma(a,b)-\sigma(b,a)$. As shown in
\cite{Lambe-Priddy82}, this Lie algebra is isomorphic to  $\fm(G;\k)$.

The group $(\Z^n,\rho)$ has canonical basis $\{e_i\}_{i=1}^n$, where 
$e_i$ is the $i$-th standard basis vector. 
Then the Malcev Lie algebra $\fm(G;\k)=(\k^n,[\:,\:])$ 
has Lie bracket given by 
$[e_i,e_j]=\sum_{k=1}^ns_{i,j}^ke_k$, 
where $s_{i,j}^k=b_k(e_i,e_j)-b_k(e_j,e_i)$.

As is well-known, the Chevalley--Eilenberg complex $\bwedge^{\!\hdot}\big(\fm(G;\k)^*\big)$ is a 
minimal model for $M=K(G,1)$; see e.g.~\cite[Thm.~3.18]{FOT}.  
Clearly, this model is generated in degree $1$; thus, it is also a 
$1$-minimal model for $G$.  As shown in \cite{Hasegawa}, 
the nilmanifold $M$ is formal if and only if $M$ is a torus. 

\subsection{Nilpotent groups and filtered-formality}
\label{subsec:nilp}

Let $G$ be a finitely generated, torsion-free nilpotent group,
and let $\fm=\fm(G;\k)$ be its Malcev Lie algebra, as described above.   
Note that $\gr(\fm)=\k^n$ has the same basis $e_1,\dots, e_n$ as
$\fm$, but, as we shall see, the Lie bracket on $\gr(\fm)$ may be different. 
The Lie algebra $\fm$ (and thus, the group $G$) is filtered-formal 
if and only if $\fm \cong \widehat{\gr}(\fm)=\gr(\fm)$, as filtered Lie algebras. 
In general, though, this isomorphism need not preserve the chosen basis.  

\begin{example}
\label{ex:nstep}
For any finitely generated free group  $F$, the $k$-step, 
free nilpotent group $F/\Gamma_{k+1}F$ is filtered-formal.
Indeed,  $F$ is $1$-formal, and thus filtered-formal.  
Hence, by Theorem \ref{thm:formalityQuo}, each nilpotent 
quotient of $F$ is also filtered-formal. 
In fact, as shown in \cite[Cor.~2.14]{Massuyeau12}, 
$\fm(F/\Gamma_{k+1} F)\cong \mathbf{L}/(\Gamma_{k+1}\mathbf{L})$, 
where $\mathbf{L}=\Lie(F)$.
\end{example}

\begin{example}
\label{ex:3step}
Let $G$ be the $3$-step, rank $2$ free nilpotent group $F_2/\Gamma_4F_2$.
Identifying $G$ with $\Z^5$ as a set,
then the Malcev Lie algebra $\fm(G;\k)=\k^5$ has 
Lie brackets given by $[e_1,e_2]=e_3-e_4/2-e_5$, $[e_1,e_3]=e_4$, 
$[e_2,e_3]=e_5$, and $[e_i,e_j]=0$, otherwise (see \cite{Lambe-Priddy82, Cenkl-Porter00}).
It is readily checked that the identity map of $\k^5$ 
is not a Lie algebra isomorphism
between $\fm=\fm(G;\k)$ and $\gr(\fm)$.
Moreover, the differential of the $1$-minimal model $\cM(G)=\bwedge^{\!\hdot}(\fm^*)$
is not homogeneous on the Hirsch weights, although $\fm$ (and $G$) 
are filtered-formal. 
\end{example}

Now consider a finite-dimensional, nilpotent Lie algebra $\fm$ 
over a field $\k$ of characteristic $0$.  It is readily seen that the 
filtered-formality of such a Lie algebra coincides with the notions 
of `Carnot', `naturally graded', `homogeneous' and `quasi-cyclic' which 
appear in \cite{Cornulier14, DL, Jo, Kasuya, Leger63}.  

The question whether the Carnot property descends from $\k=\R$ to $\Q$ was first 
raised by Johnson in \cite{Jo}.  A positive answer was given 
in \cite[Cor.~4.2]{DL}, but, as pointed out by Cornulier in 
\cite[Rem.~3.15]{Cornulier14}, the proof of that result had a gap. 
The following proposition gives a complete solution to Johnson's question.  

\begin{prop}[\cite{Cornulier14}]
\label{prop:carnot}
Let $\fm$ be a finite-dimensional, nilpotent Lie algebra over a field 
$\k$ of characteristic $0$, and let $\k\subset \K$ be a field extension.  
Then $\fm$ is Carnot over $\k$ if and only if $\fm\otimes_{\k} \K$ is Carnot over $\K$.
\end{prop}

Our more general Theorem \ref{thm:ffdescent} allows us to recover 
Proposition \ref{prop:carnot} as an immediate corollary. 

\subsection{Torsion-free nilpotent groups and filtered-formality}
\label{subsec:nilp ff}

We now study in more detail the filtered-formality properties of 
torsion-free nilpotent groups.  We start by singling out a rather 
large class of groups which enjoy this property. 

\begin{theorem}
\label{thm:step2nilpotent}
Let $G$ be a finitely generated, torsion-free, $2$-step nilpotent group. 
If $G_{\ab}$ is torsion-free, then $G$ is filtered-formal. 
\end{theorem}

\begin{proof}
The lower central series of our group takes the form 
$G=\Gamma_1G > \Gamma_2G > \Gamma_3G=1$. 
Let $\{x_1,\dots,x_n\}$ be a basis for $G/\Gamma_2G=\Z^n$,  
and let $\{y_1,\dots,y_m\}$ be a basis for $\Gamma_2G=\Z^m$.
Then, as shown for instance by Igusa and Orr in \cite[Lem.~6.1]{Igusa-Orr01}, 
the group $G$ has presentation
\begin{equation}
\label{eq:2step}
G=\Bigl\langle x_1,\dots,x_n,y_1,\dots,y_m \:\Big| \: [x_i,x_j]=
\prod\nolimits_{k=1}^my_k^{c^k_{i,j}}, \, 
[y_i,y_j]=1, \text{for  $i<j$}; [x_i,y_j]=1\Big. \Bigl\rangle.
\end{equation}

Let $a,b\in \Z^{n+m}$. A routine computation shows that 
$\rho_i(a,b)=a_i+b_i$ for $1\le i\le n$ and 
$\rho_{n+k}(a,b)=a_{n+k}+b_{n+k}-\sum_{j=1}^k\sum_{i=j+1}^n c^k_{j,i}a_ib_j$
for $1\le k\le m$.
Set $c^k_{j,i}=-c^k_{i,j}$ if $j>i$. It follows that the Malcev Lie 
algebra $\fm(G;\k)=(\k^{n+m},[~,~])$ has Lie bracket  
given on generators by $[e_i,e_j]= \sum_{k=1}^mc_{i,j}^k 
e_{n+k}$ for $1\leq i\ne j\leq n$, and zero otherwise. 

Turning now to the associated graded Lie algebra of our group, 
we have an additive decomposition, 
$\gr(G;\k)=\gr_1(G;\k)\oplus \gr_2(G;\k)=\k^n\oplus \k^m$, 
where the first factor has basis $\{e_1,\dots,e_n\}$, the second factor has 
basis $\{e_{n+1},\dots,e_{n+m} \}$, and the Lie bracket is given as above. 
Therefore, $\fm(G;\k)\cong \gr(G;\k)$, as filtered Lie algebras. 
Hence, $G$ is filtered-formal.
\end{proof}

It is known that all nilpotent Lie algebras of dimension $4$ or less are 
filtered-formal, see for instance \cite{Cornulier14}. In general, though, 
finitely generated, torsion-free nilpotent groups need not be filtered-formal. 
We illustrate this phenomenon with two examples: the first one extracted 
from the work of Cornulier \cite{Cornulier14}, and the second one adapted 
from the work of Lambe and Priddy \cite{Lambe-Priddy82}.  In both examples, 
the nilpotent Lie algebra $\fm$ in question may be 
realized as the Malcev Lie algebra of a finitely generated, 
torsion-free nilpotent group $G$.  

\begin{example}
\label{ex:Cornulier}
Let $\fm$ be the $5$-dimensional Lie algebra with non-zero 
Lie brackets given by $[e_1,e_3]=e_4$ and $[e_1,e_4]=[e_2,e_3]=e_5$. 
The center of $\fm$ is $1$-dimensional, generated by $e_5$, while the 
center of $\gr(\fm)$ is $2$-dimensional, generated by $e_2$ and $e_5$.
Therefore, $\fm\not\cong \gr(\fm)$, and so $\fm$ is not filtered-formal. 
It follows that the nilpotent group $G$ is not filtered-formal, either.
Using Theorem \ref{thm:filtmin}, it is readily checked that 
the $1$-minimal model  $\cM(G)=\bwedge^{\!\hdot}(\fm^*)$ 
does not have positive Hirsch weights; 
nevertheless, $\cM(G)$ has positive weights, given by the index 
of the chosen basis.  
\end{example}

\begin{example}
\label{ex:LPex7}
Let $\fm$ be the $7$-dimensional $\k$-Lie algebra with 
non-zero Lie brackets given on basis elements by $[e_2,e_3]=e_6$, 
$[e_2,e_4]=e_7$, $[e_2,e_5]=-e_7$, $[e_3,e_4]=e_7$, and
$[e_1,e_i]=e_{i+1}$ for $2\leq i\leq 6$.  
Then $\gr(\fm)$ has the same additive basis as $\fm$, with non-zero 
brackets given by $[e_1,e_i]=e_{i+1}$ for $2\leq i\leq 6$. Plainly, 
$\gr(\fm)$ is metabelian, (i.e., its derived subalgebra is abelian), 
while $\fm$ is not metabelian.
Thus, once again, $\fm\not\cong \gr(\fm)$, and so both $\fm$  
and $G$ are not filtered-formal.  In this case, though, we cannot 
use the indexing of the basis to put positive weights on 
$\cM(G)$.
\end{example}

\subsection{Filtered-formality and Koszulness}
\label{subsec: Nilpotent-graded}

Carlson and Toledo \cite{Carlson-Toledo95} classified finitely generated,
$1$-formal, nilpotent groups with first Betti number $5$ or less, while 
Plantiko \cite{Plantiko96} gave sufficient conditions for the associated 
graded Lie algebras of such groups to be non-quadratic. The following 
proposition follows from Theorem 4.1 in 
\cite{Plantiko96} and Lemma 2.4 in \cite{Carlson-Toledo95}.

\begin{prop}[\cite{Carlson-Toledo95, Plantiko96}]
\label{prop:non-gradedformal}
Let $G$ be a finitely generated, torsion-free, nilpotent group,  
and suppose there exists a non-zero decomposable element in the kernel 
of the cup product map $H^1(G;\k)\wedge H^1(G;\k)\to H^2(G;\k)$.  
Then $G$ is not graded-formal.
\end{prop}

Here an element $u\in H^2(G;\k)$ is said to decomposable if
 $u=v\wedge w$ for some $v,w\in H^1(G;\k)$. 

\begin{example}
\label{ex:upper}
Let $U_n(\R)$ be the nilpotent Lie group of upper triangular matrices 
with $1$'s along the diagonal. The quotient $M=U_n(\R)/U_n(\Z)$ 
is a nilmanifold of dimension $N=n(n-1)/2$. The unipotent group 
$U_n(\Z)$ has canonical basis $\{u_{ij} \mid 1\leq i<j\leq n\}$, 
where $u_{ij}$ is the matrix obtained from the identity matrix by putting 
$1$ in position $(i,j)$. Moreover, $U_n(\Z)\cong (\Z^N,\rho)$, where 
$\rho_{ij}(a,b)=a_{ij}+b_{ij}+\sum_{i<k<j} a_{ik}b_{kj}$, see \cite{Lambe-Priddy82}. 
The unipotent group $U_n(\Z)$ is filtered-formal; nevertheless, Proposition 
\ref{prop:non-gradedformal} shows that this group is 
not graded-formal for $n\geq 3$.
\end{example}
 
\begin{prop}
\label{prop:koszul-formality}
Let $G$ be a finitely generated, torsion-free, nilpotent group, and  
suppose $G$ is filtered-formal.  Then $G$ is abelian if and only if 
the algebra $U(\gr(G;\k))$ is Koszul.
\end{prop}

\begin{proof}
We only need to prove the non-trivial direction. 
If the algebra $U=U(\gr(G;\k))$ is Koszul, then the Lie algebra 
$\gr(G;\k)$ is quadratic, i.e., the group $G$ is graded-formal. 
Under the assumption that $G$ is filtered-formal, we 
then have that $G$ is $1$-formal.

Let $M$ be the nilmanifold with fundamental group $G$. Then 
$M$ is also $1$-formal.  By Nomizu's theorem, the cohomology 
ring $A=H^{\hdot}(M;\k)$ is isomorphic to the Yoneda algebra 
$\Ext_{U}^{\hdot}(\k,\k)$. On the other hand, since $U$ is Koszul, 
the Yoneda algebra is isomorphic to $U^!$, which is also Koszul.  
Hence, $A$ is a Koszul algebra.  As shown in 
\cite{Papadima-Yuzvinsky}, if $M$ is $1$-formal 
and if $A$ is Koszul, then $M$ is formal. 
By \cite{Hasegawa}, this happens if and only if $M$ is a torus.
This completes the proof.
\end{proof}

\begin{corollary}
\label{cor:koszul-ff}
Let $G$ be a finitely generated, torsion-free, $2$-step nilpotent group.
If $G_{\ab}$ is torsion-free, then $U(\gr(G;\k))$ is not Koszul.
\end{corollary}

\begin{example}
\label{ex:not koszul nil}
Let  $G=\langle x_1,x_2,x_3,x_4 \mid  
[x_1,x_3], [x_1,x_4], [x_2,x_3], [x_2,x_4], [x_1,x_2][x_3,x_4]\rangle$. 
The group $G$ is a $2$-step, commutator-relators nilpotent group. 
Hence, by the above corollary, the enveloping algebra $U(\fh(G;\k))$ 
is not Koszul.  In fact, $U(\fh(G;\k))^!$ is isomorphic to the quadratic 
algebra from Example \ref{ex:non-koszul}, which is not Koszul.
\end{example}
  
\section{Seifert fibered manifolds} 
\label{sect:seifert}

We conclude with an analysis of the fundamental groups of 
orientable Seifert manifolds from a rational homotopy viewpoint. 

\subsection{Riemann surfaces and Seifert fibered spaces}
\label{subsec:surf}
To start with, let $\Sigma_g$ be a closed, orientable surface of genus $g$. 
The fundamental group $\Pi_g=\pi_1(\Sigma_g)$ is a $1$-relator group, 
with generators $ x_1,y_1,\dots, x_g,y_g$ and a single relation, 
$[x_1,y_1]\cdots [x_g,y_g]=1$.  Since this group is trivial for 
$g=0$, we will assume for now that $g>0$. 
The cohomology algebra $A=H^{\hdot}(\Sigma_g;\k)$ is the quotient 
of  the exterior algebra on generators $a_1,b_1,\dots, a_{g},b_{g}$, 
in degree $1$ by the ideal $I$ generated by $a_ib_i-a_j b_j$, 
for $1\le i< j\le g$, together with $a_ia_j$, $b_ib_j$, $a_i b_j$, 
$a_jb_i$, for $1\le i< j\le g$.  It is readily seen that the generators 
of $I$ form a quadratic Gr\"{o}bner basis for this ideal; therefore, 
$A$ is a Koszul algebra. 

The Riemann surface $\Sigma_g$ is a compact K\"{a}hler manifold, 
and thus, a formal space.  It follows from Theorem \ref{thm:py} that the 
minimal model of $\Sigma_g$ is generated in degree one, i.e., 
$\cM(\Sigma_g)=\cM(\Sigma_g,1)$.  The formality of $\Sigma_g$ 
also implies the $1$-formality of $\Pi_g$.  As a consequence, the associated 
graded Lie algebra $\gr(\Pi_g;\k)$ is isomorphic to the holonomy Lie 
algebra $\fh(\Pi_g;\k)$. 
Using again the fact that $A$ is a Koszul algebra, we deduce from 
Corollary \ref{cor:holokoszul}  that 
$\prod_{k\geq 1}(1-t^k)^{\phi_k(\Pi_g)}=1-2gt+t^2$.  

We will consider here only orientable, closed Seifert manifolds 
with orientable base.   Every such manifold $M$ admits an effective circle action, 
with orbit space an orientable surface of genus $g\ge 0$, and finitely many 
exceptional orbits, encoded in pairs of coprime integers $(\alpha_1,\beta_1), 
\dots,  (\alpha_s,\beta_s)$ with $\alpha_j\ge 2$. The obstruction to trivializing 
the bundle $\eta\colon M\to \Sigma_g$ outside tubular neighborhoods of the 
exceptional orbits is given by an integer $b=b(\eta)$.  The 
group  $\pi_{\eta}:=\pi_1(M)$ has presentation 
with generators $x_1, y_1,\dots, x_g, y_g, z_1,\dots,z_s, h$
and relators $[x_1,y_1]\cdots[x_g,y_g]z_1\cdots z_s=h^{b}$ and 
$ z_i^{\alpha_i}h^{\beta_i}=1$ $(i=1,\dots,s)$, where $h$ is central.

For instance, if $s=0$, the corresponding manifold, $M_{g,b}$, 
is the $S^1$-bundle over $\Sigma_g$ with Euler number $b$. 
Let $\pi_{g,b}:=\pi_1(M_{g,b})$ be the fundamental group of this 
manifold.  If $b=0$, then $\pi_{g,0}=\Pi_g\times \Z$, whereas if 
$b=1$, then  
$\pi_{g,1}=\langle x_1, y_1,\dots, x_g, y_g, h\mid [x_1,y_1]\cdots[x_g,y_g]=h, 
h~{\rm central}\rangle$. 
In particular, $M_{1,1}$ is the Heisenberg $3$-dimensional nilmanifold 
and $\pi_{1,1}$ is the group from Example~\ref{ex:filt massey}.

\subsection{Minimal model}
\label{subsec:quad model}

As shown in \cite{Scott83}, the Euler number $e(\eta)$ of  the 
Seifert bundle $\eta \colon M\rightarrow \Sigma_g$ satisfies
$e(\eta)=-b(\eta)-\sum_{i=1}^{s}\beta_i/\alpha_i$.  
If $g=0$, the group $\pi_\eta$ 
has first Betti number $0$ or $1$, according to whether $e(\eta)$ is 
non-zero or $0$. Thus, $\pi_{\eta}$ is $1$-formal, and the Malcev 
Lie algebra $\fm(\pi_\eta;\k)$ is either $0$, or the completed free 
Lie algebra of rank $1$. To analyze the case when $g>0$, 
we will employ the minimal model of $M$, 
as constructed by Putinar in \cite{Putinar98} (see also \cite[\S8.8]{FHT2}).

 \begin{theorem}[\cite{Putinar98}]
\label{thm:seifertmodel}  
Let $\eta\colon M\to \Sigma_g$ be an orientable Seifert fibered space 
with $g>0$. The minimal model $\cM(M)$ is the Hirsch extension 
$\cM(\Sigma_g)\otimes_{\k} (\bigwedge(c),d)$, where the differential
is given by $d(c)=0$ if $e(\eta)=0$, and $d(c)\in \cM^2(\Sigma_g)$ 
represents a generator of $H^2(\Sigma_g;\k)$ if $e(\eta)\neq 0$.
\end{theorem}

More precisely, recall that $\Sigma_g$ is formal, and so there is a 
quasi-isomorphism $f\colon  \cM(\Sigma_g) \to (H^{\hdot}(\Sigma_g;\k),d=0)$.  
Thus, there is an element 
$a\in \mathcal{M}^2(M)$ such that $d(a)=0$ and $H(f)([a])\ne 0$ in  
$H^2(\Sigma_g;\k)=\k$.  We then set $d(c)=a$ in the second case. 

To each Seifert fibration $\eta\colon M\to \Sigma_g$ as above, let us associate the 
$S^1$-bundle $\bar{\eta}\colon M_{g,\epsilon(\eta)}\to \Sigma_g$, where 
$\epsilon(\eta)=0$ if $e(\eta)=0$, and $\epsilon(\eta)=1$ if $e(\eta)\neq 0$. 
For instance, $M_{0,0}=S^2\times S^1$ and $M_{0,1}=S^3$. 
The above theorem implies that $\cM(M) \cong \cM(M_{g,\epsilon(\eta)})$.
Hence, we have the following corollary.

\begin{corollary}
\label{cor:seifertMalcev} 
Let $\eta \colon M \rightarrow \Sigma_g$ be an orientable Seifert fibered space. 
The Malcev Lie algebra of the fundamental group $\pi_{\eta}=\pi_1(M)$ is given by 
$\fm(\pi_{\eta};\k)\cong \fm(\pi_{g,\epsilon(\eta)};\k)$. 
\end{corollary}

\begin{corollary}
\label{cor:seifertHirsch} 
Let $\eta\colon M\to \Sigma_g$ be an orientable Seifert fibered space with $g>0$.   
Then $M$ admits a minimal model with positive Hirsch weights. 
\end{corollary}

\begin{proof}
We know from \S\ref{subsec:surf} that the minimal 
model $\cM(\Sigma_g)$ is formal, and generated in degree one 
(since $g>0$). 
By Theorem \ref{thm:filtmin}, $\cM(\Sigma_g)$ is isomorphic to  
a minimal model of $\Sigma_g$ with positive Hirsch weights;  
denote this model by  $\mathcal{H}(\Sigma_g)$.
By Theorem \ref{thm:seifertmodel} and Lemma \ref{lem:Hirschext}, 
the Hirsch extension $\mathcal{H}(\Sigma_g)\otimes_{\k} \bigwedge(c)$ 
is a minimal model for $M$,  generated in degree one.
Moreover, the weight of $c$ equals $1$ if $e(\eta)= 0$, and 
equals $2$ if $e(\eta)\neq 0$.
Clearly, the differential $d$ is homogeneous 
with respect to these weights, and this completes the proof.
\end{proof}
  
\begin{corollary}
\label{cor:ff seifert}
The fundamental groups of orientable Seifert manifolds are filtered-formal.
\end{corollary} 

\begin{proof}
The claim follows at once from Theorem \ref{thm:filtmin} and Corollary \ref{cor:seifertHirsch}.
Alternatively, the claim also follows from Theorem \ref{thm:seifertmalcev} 
and the definition of filtered-formality.
\end{proof}  
  
Using Theorem \ref{thm:seifertmodel} and Lemma \ref{lem:Hirschext} again,
we obtain a quadratic model for the Seifert manifold $M$ 
in the case when the base has positive genus.

\begin{corollary}
\label{cor:seifertmodel}
Suppose $g>0$.  Then $M$ has a quadratic model of the form 
$\big(H^{\hdot}(\Sigma_g;\k)\otimes \bigwedge(c), d\big)$, where $\deg(c)=1$  
and the differential $d$ is given by $d(a_i)=d(b_i)=0$ for $1\leq i \leq g$, 
$d(c)=0$ if $e(\eta)= 0$, and $d(c)=a_1\wedge b_1$ if $e(\eta)\neq 0$.
\end{corollary}

\subsection{Malcev Lie algebra}
\label{subsec:MalcevLieSeifert}

We give now an explicit presentation for the Malcev Lie algebra of $\pi_{\eta}$ 
as the degree completion of a certain graded Lie algebra. 

\begin{theorem}
\label{thm:seifertmalcev}
The Malcev Lie algebra of $\pi_{\eta}$ is the degree completion of the 
graded Lie algebra
\begin{equation}
\label{eq:seifertmalcev}
L(\pi_{\eta})=
\begin{cases}
\Lie(x_1,y_1,\dots,x_g,y_g,z )/\langle 
\sum_{i=1}^{g} [x_i,y_i]=0,\: z {~\rm central} \rangle 
 & \text{if $e(\eta)=0$};
 \\[2pt]
\Lie( x_1,y_1,\dots,x_g,y_g,w)/\langle 
\sum_{i=1}^{g} [x_i,y_i]=w,\: w {~\rm central} \rangle 
 & \text{if $e(\eta)\ne 0$},
\end{cases}
\end{equation}
where $\deg(w)=2$ and the other generators have degree $1$.
Moreover, $\gr(\pi_{\eta};\k)\cong L(\pi_{\eta})$. 
\end{theorem} 

\begin{proof}
The case $g=0$ was already dealt with, so assume $g>0$. 
There are two cases to consider.  

If $e(\eta)=0$, Corollary \ref{cor:seifertMalcev} 
says that $\fm(\pi_{\eta};\k)$ is isomorphic to the Malcev Lie algebra of 
$\pi_{g,0}=\Pi_g\times \Z$, which is a $1$-formal group. 
Furthermore, we know that $\gr(\Pi_g;\k)$ is the 
quotient of  $\Lie(2g )$ by the ideal generated by 
$\sum_{i=1}^{g} [x_i,y_i]$.
Hence, $\fm(\pi_{\eta};\k)$ is isomorphic to the degree completion of 
$\gr(\Pi_g\times \Z)=\gr(\Pi_g;\k)\times \gr (\Z;\k)$, 
which is precisely the Lie algebra $L(\pi_{\eta})$ from 
\eqref{eq:seifertmalcev}.

If $e(\eta)\neq 0$, Corollary \ref{cor:seifertmodel} provides a quadratic 
model for our Seifert manifold. Taking the Lie algebra dual to this quadratic 
model and using \cite[Thm.~4.3.6]{Bibby-H14} or \cite[Thm.~3.1]{BMPP}, 
we infer that $\fm(\pi_{\eta})$ is isomorphic to 
the degree completion of the graded Lie algebra $L(\pi_{\eta})$.  
Furthermore, by formula \eqref{eq:quillen}, there is an 
isomorphism $\gr(\fm(\pi_{\eta};\k))\cong \gr(\pi_{\eta};\k)$.   
This completes the proof.
\end{proof}
 
In follow-up work \cite{SW-holo}, we give a presentation for the holonomy 
Lie algebra of an orientable Seifert manifold group, and derive the following result.

\begin{prop}
\label{cor:seifertgraded}
If $g=0$, the group $\pi_{\eta}$ is always $1$-formal,  
while if $g>0$, the group $\pi_{\eta}$ is graded-formal if and only if $e(\eta)= 0$. 
\end{prop} 

\begin{ack}
We wish to thank Yves Cornulier, \c{S}tefan Papadima, 
and Richard Porter for several useful comments regarding this work.
\end{ack}

\newcommand{\arxiv}[1]
{\texttt{\href{http://arxiv.org/abs/#1}{arXiv:#1}}}
\newcommand{\arx}[1]
{\texttt{\href{http://arxiv.org/abs/#1}{arxiv:}}
\texttt{\href{http://arxiv.org/abs/#1}{#1}}}
\newcommand{\doi}[1]
{\texttt{\href{http://dx.doi.org/#1}{doi:#1}}}
\renewcommand*\MR[1]{%
\StrBefore{#1 }{ }[\temp]%
\href{http://www.ams.org/mathscinet-getitem?mr=\temp}{MR#1}}

\def\cprime{$'$}

\bibliographystyle{amsplain}

\begin{thebibliography}{10}

\bibitem{Bartholdi-E-E-R}
Laurent Bartholdi, Benjamin Enriquez, Pavel Etingof, and Eric Rains,
\href{http://dx.doi.org/10.1016/j.jalgebra.2005.12.006}
{\em Groups and {L}ie algebras corresponding to}
\href{http://dx.doi.org/10.1016/j.jalgebra.2005.12.006} 
{\em  the {Y}ang--{B}axter equations}, 
J. Algebra \textbf{305} (2006), no.~2, 742--764. 
  
\bibitem{BMPP}
Barbu Berceanu, Daniela~Anca M\u{a}cinic, {\c{S}}tefan Papadima, and
Clement~Radu Popescu, \href{http://dx.doi.org/10.2140/agt.2017.17.1163}
{\emph{On the geometry and topology of partial configuration spaces of 
{R}iemann surfaces}}, Algebr. Geom. Topol. \textbf{17} (2017), no.~2, 1163--1188.

\bibitem{Berceanu-Papadima94}
Barbu Berceanu and {\c{S}}tefan Papadima,
  \href{http://dx.doi.org/10.1007/BF01459745} {\emph{Cohomologically generic
  {$2$}-complexes and {$3$}-dimensional {P}oincar\'e complexes}}, Math. Ann.
  \textbf{298} (1994), no.~3, 457--480. 
 
\bibitem{Bez}
Roman Bezrukavnikov, \href{http://dx.doi.org/10.1007/BF01895836} {\emph{Koszul
  {DG}-algebras arising from configuration spaces}}, Geom. Funct. Anal.
  \textbf{4} (1994), no.~2, 119--135. 

\bibitem{Bibby-H14}
Christin Bibby and Justin Hilburn, 
\href{http://dx.doi.org/10.2140/agt.2016.16.2637}{\emph{Quadratic-linear duality 
and rational homotopy theory of chordal arrangements}}, Algebr. Geom. Topol. 
\textbf{16} (2016), 2637--2661.

\bibitem{BMSS}
Richard Body, Mamoru Mimura, Hiroo Shiga, and Dennis Sullivan,
  \href{http://dx.doi.org/10.1007/s000140050063} {\emph{{$p$}-universal spaces
  and rational homotopy types}}, Comment. Math. Helv. \textbf{73} (1998),
  no.~3, 427--442. 

\bibitem{Calaque-E-E}
Damien Calaque, Benjamin Enriquez, and Pavel Etingof,
  \href{http://dx.doi.org/10.1007/978-0-8176-4745-2_5} {\emph{Universal {KZB}
  equations: the elliptic case}}, 
Algebra, arithmetic, and geometry: in honor
 of {Y}u. {I}. {M}anin. {V}ol. {I}, 
 Progr. Math., vol. 269, Birkh\"auser
  Boston, Inc., Boston, MA, 2009, pp.~165--266. 

\bibitem{Carlson-Toledo95}
James~A. Carlson and Domingo Toledo,
  \href{http://dx.doi.org/10.1007/BF02921801} {\emph{Quadratic presentations
  and nilpotent {K}\"ahler groups}}, J. Geom. Anal. \textbf{5} (1995), no.~3,
  359--377. 

\bibitem{Cenkl-Porter81}
Bohumil Cenkl and Richard Porter,
  \href{http://projecteuclid.org/euclid.jdg/1214435841} {\emph{Mal{\cprime\!}cev's
  completion of a group and differential forms}}, J. Differential Geom.
  \textbf{15} (1980), no.~4, 531--542 (1981). 

\bibitem{Cenkl-Porter00}
Bohumil Cenkl and Richard Porter,
\href{http://dx.doi.org/10.2140/pjm.2000.193.5} {\emph{Nilmanifolds and
associated {L}ie algebras over the integers}}, Pacific J. Math. \textbf{193}
(2000), no.~1, 5--29. 
 
\bibitem{Chen51}
Kuo-Tsai Chen, \href{http://dx.doi.org/10.2307/1969316} {\emph{Integration in
  free groups}}, Ann. of Math. (2) \textbf{54} (1951), 147--162. 

\bibitem{Chen73}
Kuo-Tsai Chen, \href{http://dx.doi.org/10.2307/1970846} {\emph{Iterated
  integrals of differential forms and loop space homology}}, Ann. of Math. (2)
  \textbf{97} (1973), 217--246. 

\bibitem{Cornulier14} Yves Cornulier, 
\href{http://smf4.emath.fr/Publications/Bulletin/144/html/smf_bull_144_693-744.php}
{\emph{Gradings on Lie algebras, systolic growth, and cohopfian
properties of nilpotent groups}}, Bull. Math. Soc. France \textbf{14} (2016), no.~4, 
693--744. 

\bibitem{DL} Karel Dekimpe and Kyung Bai Lee,
\href{http://dx.doi.org/10.1090/S0002-9947-02-03084-2}
{\em Expanding maps on infra-nilmanifolds of homogeneous type},   
Trans. Amer. Math. Soc. \textbf{355} (2003), no.~3, 1067--1077.

\bibitem{DGMS}
Pierre Deligne, Phillip Griffiths, John Morgan, and Dennis Sullivan,
  \href{http://dx.doi.org/10.1007/BF01389853} {\emph{Real homotopy theory of
  {K}\"ahler manifolds}}, Invent. Math. \textbf{29} (1975), no.~3, 245--274.

\bibitem{Dimca-Papadima-Suciu}
Alexandru Dimca, {\c{S}}tefan Papadima, and Alexander~I. Suciu,
\href{http://dx.doi.org/10.1215/00127094-2009-030} {\emph{Topology and
geometry of cohomology jump loci}}, Duke Math. J. \textbf{148} (2009), no.~3,
405--457. 
  
\bibitem{Dr} Vladimir Drinfel'd, 
\href{http://mi.mathnet.ru/eng/aa199}
{\em On quasitriangular quasi-{H}opf algebras and on a group that is closely 
connected with $\operatorname{Gal}(\overline{\Q}/\Q)$}, 
Leningrad Math. J. \textbf{2} (1991), no.~4, 829--860.

\bibitem{Dwyer}
William~G. Dwyer, \href{http://dx.doi.org/10.1016/0022-4049(75)90006-7}
{\emph{Homology, {M}assey products and maps between groups}}, J. Pure Appl.
Algebra \textbf{6} (1975), no.~2, 177--190. 

\bibitem{Ekedahl-Merkulov11}
Torsten Ekedahl and Sergei Merkulov, \href{http://tinyurl.com/ozx2cnv}
{\em Grothendieck--Teichm\"{u}ller group in algebra, geometry and
quantization: A} \href{http://tinyurl.com/ozx2cnv}{survey}, preprint (2011).
  
\bibitem{Enriquez} Benjamin Enriquez, 
\href{http://dx.doi.org/10.1007/s00029-013-0137-3}
{\em Elliptic associators}, 
Selecta Math. (N.S.) \textbf{20} (2014), no.~2, 491--584. 

\bibitem{FalkRandell}
Michael Falk and Richard Randell, \href{http://dx.doi.org/10.1007/BF01394780}
  {\emph{The lower central series of a fiber-type arrangement}}, Invent. Math.
  \textbf{82} (1985), no.~1, 77--88. 

\bibitem{FHT}
Yves F{\'e}lix, Stephen Halperin, and Jean-Claude Thomas,
  \href{http://dx.doi.org/10.1007/978-1-4613-0105-9} {\emph{Rational homotopy
  theory}}, Graduate Texts in Mathematics, vol.~205, Springer-Verlag, New York,
  2001. 
  
\bibitem{FHT2}
Yves F{\'e}lix, Stephen Halperin, and Jean-Claude Thomas,
\href{http://dx.doi.org/10.1142/9473} 
{\emph{Rational homotopy theory \textup{II}}},  
World Scientific Publishing, Hackensack, NJ, 2015.

\bibitem{FOT} Yves~F\'{e}lix, John~Oprea, and Daniel~Tanr\'{e}, 
\href{https://global.oup.com/academic/product/algebraic-models-in-geometry-9780199206520}
{\em Algebraic models in geometry}, Oxford Grad. Texts in Math., 
vol.~17, Oxford Univ. Press, Oxford, 2008. 

\bibitem{Fenn-S}
Roger Fenn and Denis Sjerve, \href{http://dx.doi.org/10.4153/CJM-1987-015-5}
  {\emph{Massey products and lower central series of free groups}}, Canad. J.
  Math. \textbf{39} (1987), no.~2, 322--337. 

\bibitem{FM05}
Marisa Fern{\'a}ndez and Vicente Mu{\~n}oz,
  \href{http://dx.doi.org/10.1007/s00209-004-0747-8} {\emph{Formality of
  {D}onaldson submanifolds}}, Math. Z. \textbf{250} (2005), no.~1, 149--175.

\bibitem{Freedman-Hain-Teichner}
Michael Freedman, Richard~M. Hain, and Peter Teichner,
  \href{http://dx.doi.org/10.1142/9789812385215_0045} {\emph{Betti number
  estimates for nilpotent groups}}, Fields {M}edallists' lectures, World Sci.
  Ser. 20th Century Math., vol.~5, World Sci. Publ., River Edge, NJ, 1997,
  pp.~413--434. 

\bibitem{Froberg97}
Ralph Fr{\"o}berg, \emph{Koszul algebras}, Advances in commutative ring theory
  ({F}ez, 1997), Lecture Notes in Pure and Appl. Math., vol. 205, Dekker, New
  York, 1999, pp.~337--350. 

\bibitem{Griffiths-Morgan13}
Phillip Griffiths and John Morgan,
  \href{http://dx.doi.org/10.1007/978-1-4614-8468-4} {\emph{Rational homotopy
  theory and differential forms}}, Second ed., Progr. Math.,
  vol.~16, Springer, New York, 2013. 

\bibitem{Hain85}
Richard~M. Hain, \href{http://dx.doi.org/10.1016/0040-9383(85)90044-8}
  {\emph{Iterated integrals, intersection theory and link groups}}, Topology
  \textbf{24} (1985), no.~1, 45--66. 

\bibitem{Hain97}
Richard~M. Hain, \href{http://dx.doi.org/10.1090/S0894-0347-97-00235-X}
  {\emph{Infinitesimal presentations of the {T}orelli groups}}, J. Amer. Math.
  Soc. \textbf{10} (1997), no.~3, 597--651. 

\bibitem{Hain15}
Richard~M. Hain, \href{http://dx.doi.org/10.1112/jtopol/jtu020} {\emph{Genus 3
  mapping class groups are not {K}\"ahler}}, J. Topol. \textbf{8} (2015),
  no.~1, 213--246. 

\bibitem{Halperin-Stasheff79}
Stephen Halperin and James Stasheff,
  \href{http://dx.doi.org/10.1016/0001-8708(79)90043-4} {\emph{Obstructions to
  homotopy equivalences}}, Adv. in Math. \textbf{32} (1979), no.~3, 233--279.

\bibitem{Hasegawa}
Keizo Hasegawa, \href{http://dx.doi.org/10.2307/2047375} {\emph{Minimal models
  of nilmanifolds}}, Proc. Amer. Math. Soc. \textbf{106} (1989), no.~1, 65--71.
  
\bibitem{Hilton-Stammbach97}
Peter~J. Hilton and Urs~Stammbach,
  \href{http://dx.doi.org/10.1007/978-1-4419-8566-8} {\emph{A course in
  homological algebra}}, Second ed., Grad. Texts in Math., vol.~4,
  Springer-Verlag, New York, 1997. 

\bibitem{Igusa-Orr01}
Kiyoshi Igusa and Kent~E. Orr,
\href{http://dx.doi.org/10.1016/S0040-9383(00)00002-1} {\emph{Links, pictures
and the homology of nilpotent groups}}, Topology \textbf{40} (2001), no.~6,
1125--1166. 

\bibitem{Jo} R. Warren Johnson, 
\href{http://dx.doi.org/10.1090/S0002-9939-1975-0374217-4}%
{\em Homogeneous {L}ie algebras and expanding automorphisms},  
Proc. Amer. Math. Soc. \textbf{48} (1975), no.~2, 292--296. 

\bibitem{Kasuya}
Hisashi Kasuya, \href{http://dx.doi.org/10.1142/S1793525316500114}
 {\emph{Singularity of the varieties of representations of lattices in
solvable {L}ie groups}}, J. Topol. Anal. \textbf{8} (2016), no.~2

\bibitem{Kohno}
Toshitake Kohno, \href{http://projecteuclid.org/euclid.nmj/1118787354}
{\em On the holonomy {L}ie algebra and the nilpotent completion of the
  fundamental group of the com-} 
\href{http://projecteuclid.org/euclid.nmj/1118787354}{\em plement of 
  hypersurfaces}, Nagoya Math. J. \textbf{92} (1983), 21--37. 

\bibitem{Labute08}
John~P. Labute,
  \href{http://www.labmath.uqam.ca/~annales/volumes/32-2/PDF/189-197.pdf}
  {\emph{Fabulous pro-{$p$}-groups}}, Ann. Sci. Math. Qu\'ebec \textbf{32}
  (2008), no.~2, 189--197. 

\bibitem{Lambe86}
Larry~A. Lambe, \href{http://dx.doi.org/10.2307/2046181} {\emph{Two exact
  sequences in rational homotopy theory relating cup products and
  commutators}}, Proc. Amer. Math. Soc. \textbf{96} (1986), no.~2, 360--364.

\bibitem{Lambe-Priddy82}
Larry~A. Lambe and Stewart~B. Priddy, \href{http://dx.doi.org/10.2307/1999191}
  {\emph{Cohomology of nilmanifolds and torsion-free, nilpotent groups}},
  Trans. Amer. Math. Soc. \textbf{273} (1982), no.~1, 39--55. 

\bibitem{Lazard54}
Michel Lazard, \href{http://www.numdam.org/item?id=ASENS_1954_3_71_2_101_0}
  {\emph{Sur les groupes nilpotents et les anneaux de {L}ie}}, Ann. Sci. \'{E}cole
  Norm. Sup. (3) \textbf{71} (1954), 101--190. 

\bibitem{LazarevMarkl}
Andrey Lazarev and Martin Markl,
\href{http://dx.doi.org/10.1016/j.aim.2015.07.009}
{\emph{Disconnected rational homotopy theory}}, 
Adv. Math. \textbf{283} (2015), 303--361.

\bibitem{Lee}
Peter Lee, \href{http://dx.doi.org/10.1007/s00029-012-0107-1}
{\emph{The pure virtual braid group is quadratic}}, 
Selecta Math. (N.S.) \textbf{19} (2013), no.~2, 461--508. 

\bibitem{Leger63}
George Leger, \href{http://dx.doi.org/doi:10.1215/S0012-7094-63-03067-9}
  {\emph{Derivations of {L}ie algebras. \textup{III}}}, Duke Math. J.
  \textbf{30} (1963), 637--645. 

\bibitem{Lichtman80}
Alexander~I. Lichtman, \href{http://dx.doi.org/10.1016/0022-4049(80)90117-6}
{\emph{On {L}ie algebras of free products of groups}}, J. Pure Appl. Algebra
\textbf{18} (1980), no.~1, 67--74. 

\bibitem{Lofwall}
Clas L{\"o}fwall, \href{http://dx.doi.org/10.1007/BFb0075468} {\emph{On the
subalgebra generated by the one-dimensional elements in the {Y}oneda
{E}xt-algebra}}, Algebra, algebraic topology and their interactions
({S}tockholm, 1983), Lecture Notes in Math., vol. 1183, Springer, Berlin,
1986, pp.~291--338. 

\bibitem{Maassarani}
Mohamad Maassarani, {\em Sur certains espaces de configurations 
associ\'{e}s aux sous-groupes finis de $\operatorname{PSL}_2(\C)$}, 
preprint (2015), \arxiv{1510.00617v1}.

\bibitem{Macinic}
Anca~Daniela M{\u{a}}cinic, \href{http://dx.doi.org/10.1016/j.jpaa.2009.12.025}
{\emph{Cohomology rings and formality properties of nilpotent groups}}, 
J.~Pure Appl. Algebra \textbf{214} (2010), no.~10, 1818--1826. 

\bibitem{Magnus-K-S}
Wilhelm Magnus, Abraham Karrass, and Donald Solitar, \emph{Combinatorial group
theory: {P}resentations of groups in terms of generators and relations},
Interscience Publishers, New York-London-Sydney, 1966. 

\bibitem{Malcev51}
Anatoli~I. Malcev, \href{http://mi.mathnet.ru/eng/izv3161}%
{\emph{On a class of homogeneous spaces}}, Amer. Math. Soc.
Translation \textbf{1951} (1951), no.~39, 33pp. 

\bibitem{Markl-Papadima}
Martin Markl and {\c{S}}tefan Papadima,
  \href{http://www.numdam.org/item?id=AIF_1992__42_4_905_0} {\emph{Homotopy
  {L}ie algebras and fundamental groups via deformation theory}}, Ann. Inst.
  Fourier (Grenoble) \textbf{42} (1992), no.~4, 905--935. 

\bibitem{Massey98}
William~S. Massey, \href{http://dx.doi.org/10.1142/S0218216598000206}
  {\emph{Higher order linking numbers}}, J. Knot Theory Ramifications
  \textbf{7} (1998), no.~3, 393--414. 
 
\bibitem{Massuyeau12}
Gw{\'e}na{\"e}l Massuyeau,
  \href{http://www.numdam.org/item?id=BSMF_2012__140_1_101_0}
  {\emph{Infinitesimal {M}orita homomorphisms and the tree-level of the {LMO}
  invariant}}, Bull. Soc. Math. France \textbf{140} (2012), no.~1, 101--161.

\bibitem{MateiSuciu00}
Daniel Matei and Alexander~I. Suciu,
  \href{http://dx.doi.org/10.1016/S0040-9383(98)00058-5} {\emph{Homotopy types
  of complements of {$2$}-arrangements in {${\bf R}\sp 4$}}}, Topology
 \textbf{39} (2000), no.~1, 61--88. 

\bibitem{Morgan}
John~W. Morgan, \href{http://dx.doi.org/10.1007/BF02684316} {\emph{The
  algebraic topology of smooth algebraic varieties}}, Inst. Hautes \'Etudes
  Sci. Publ. Math. (1978), no.~48, 137--204. 

\bibitem{Neisendorfer-Miller78}
Joseph Neisendorfer and Timothy Miller,
  \href{http://projecteuclid.org/euclid.ijm/1256048467} {\emph{Formal and
  coformal spaces}}, Illinois J. Math. \textbf{22} (1978), no.~4, 565--580.

\bibitem{Papadima-Suciu04}
Stefan Papadima and Alexander~I. Suciu,
  \href{http://dx.doi.org/10.1155/S1073792804132017} {\emph{Chen {L}ie
  algebras}}, Int. Math. Res. Not. (2004), no.~21, 1057--1086. 

\bibitem{Papadima-Suciu09}
Stefan Papadima and Alexandru Suciu,
\href{http://ssmr.ro/bulletin/pdf/52-3/Papadima.pdf} {\emph{Geometric and
algebraic aspects of 1-formality}}, Bull. Math. Soc. Sci. Math. Roumanie
(N.S.) \textbf{52(100)} (2009), no.~3, 355--375. 

\bibitem{Papadima-Suciu14product}
{\c{S}}tefan Papadima and Alexander~I. Suciu,
  \href{http://dx.doi.org/10.1007/978-3-319-09186-0_17} {\emph{Non-abelian
  resonance: product and coproduct formulas}}, Bridging algebra, geometry, and
  topology, Springer Proc. Math. Stat., vol.~96, Springer, Cham, 2014,
  pp.~269--280. 

\bibitem{PS16}
Stefan Papadima and Alexander~I. Suciu,
{\em The topology of compact Lie group actions through the lens of finite models}, 
to appear in International Mathematics Research Notices, \doi{10.1093/imrn/rnx294}.

\bibitem{PS17}
Stefan Papadima and Alexander~I. Suciu,
\href{http://dx.doi.org/10.1112/jlms.12169}%
{\em Infinitesimal finiteness obstructions}, 
J. London Math. Soc. \textbf{99} (2019), no.~1, 173--193. 

\bibitem{Papadima-Yuzvinsky}
Stefan Papadima and Sergey Yuzvinsky,
  \href{http://dx.doi.org/10.1016/S0022-4049(98)00058-9} {\emph{On rational
  {$K[\pi,1]$} spaces and {K}oszul algebras}}, J. Pure Appl. Algebra
  \textbf{144} (1999), no.~2, 157--167. 

\bibitem{Plantiko96}
R{\"u}diger Plantiko, \href{http://dx.doi.org/10.1515/form.1996.8.569}
  {\emph{The graded {L}ie algebra of a {K}\"ahler group}}, Forum Math.
  \textbf{8} (1996), no.~5, 569--583. 

\bibitem{Polishchuk-Positselski}
Alexander Polishchuk and Leonid Positselski,
  \href{http://dx.doi.org/10.1090/ulect/037} {\emph{Quadratic algebras}},
  University Lecture Series, vol.~37, American Mathematical Society,
  Providence, RI, 2005. 

\bibitem{Porter80}
Richard Porter, \href{http://dx.doi.org/10.2307/1998124} {\emph{Milnor's {$\bar
  \mu $}-invariants and {M}assey products}}, Trans. Amer. Math. Soc.
  \textbf{257} (1980), no.~1, 39--71. 

\bibitem{Putinar98}
Gabriela Putinar,
  \href{http://www.dmi.unict.it/ojs/index.php/lematematiche/article/view/371}
  {\emph{Minimal models and the virtual degree of {S}eifert fibered spaces}},
  Matematiche (Catania) \textbf{53} (1998), no.~2, 319--329 (1999).
  
\bibitem{Quillen68}
Daniel~G. Quillen, \href{http://dx.doi.org/10.1016/0021-8693(68)90069-0}
  {\emph{On the associated graded ring of a group ring}}, J. Algebra
  \textbf{10} (1968), 411--418. 
  
\bibitem{Quillen69}
Daniel Quillen, \href{http://dx.doi.org/10.2307/1970725} {\emph{Rational
  homotopy theory}}, Ann. of Math. (2) \textbf{90} (1969), 205--295.

\bibitem{Scott83}
Peter Scott, \href{http://dx.doi.org/10.1112/blms/15.5.401} {\emph{The
  geometries of {$3$}-manifolds}}, Bull. London Math. Soc. \textbf{15} (1983),
  no.~5, 401--487. 

\bibitem{Serre}
Jean-Pierre Serre, \href{http://dx.doi.org/10.1007/978-3-540-70634-2}
  {\emph{Lie algebras and {L}ie groups}}, second ed., Lecture Notes in
  Mathematics, vol. 1500, Springer-Verlag, Berlin, 1992, ~1964 lectures given
  at Harvard University. 

\bibitem{Stallings}
John Stallings, \href{http://dx.doi.org/10.1016/0021-8693(65)90017-7}
{\emph{Homology and central series of groups}}, J. Algebra \textbf{2} (1965),
170--181. 

\bibitem{Suciu-3manifold}
Alexander~I. Suciu, \emph{Cohomology jump loci of 3-manifolds}, 
preprint (2019), \arxiv{1901.01419v1}.

\bibitem{SW-pvb}
Alexander~I. Suciu and He~Wang, 
\href{http://dx.doi.org/10.1007/s00209-016-1811-x}
{\em Pure virtual braids, resonance, and formality}, 
Math. Zeit. \textbf{286} (2017), no.~3--4, 1495--1524.

\bibitem{SW-braids}
Alexander~I. Suciu and He~Wang, 
\href{http://dx.doi.org/10.1007/978-3-319-58971-8_15}
{\em The pure braid groups and their relatives}, in: {\em Perspectives in {L}ie theory}, 
403--426, Springer INdAM series, vol.~19, Springer, Cham, 2017.

\bibitem{SW-holo}
Alexander~I. Suciu and He~Wang, 
\href{https://doi.org/10.1016/j.jpaa.2018.11.006}%
{\em Cup products, lower central series, and holonomy Lie algebras}, 
J. Pure. Appl. Algebra. \textbf{223} (2019), no.~8, 3359--3385.   

\bibitem{SW-mccool}
Alexander~I. Suciu and He~Wang, \emph{Chen ranks and resonance varieties 
of the upper McCool groups}, preprint (2018), \arxiv{1804.06006v1}.

\bibitem{SW-expansion}
Alexander~I. Suciu and He~Wang, 
\emph{Taylor expansions of groups and filtered-formality}, preprint (2019).

\bibitem{Sullivan75}
Dennis Sullivan, \href{http://dx.doi.org/10.1016/0040-9383(75)90009-9}
{\emph{On the intersection ring of compact three manifolds}}, Topology
\textbf{14} (1975), no.~3, 275--277. 

\bibitem{Sullivan}
Dennis Sullivan, \href{http://dx.doi.org/10.1007/BF02684341}
{\emph{Infinitesimal computations in topology}}, Inst. Hautes \'Etudes Sci.
Publ. Math. (1977), no.~47, 269--331. 

\end{thebibliography}

\end{document}